\documentclass[12pt]{article}
\usepackage{amsmath,amsfonts,amssymb,amsthm}
\usepackage{graphicx}
\usepackage{enumerate}
\usepackage[round]{natbib}
\usepackage{url} % not crucial - just used below for the URL 
\usepackage{booktabs}
\usepackage{subcaption}
\usepackage[ruled,vlined]{algorithm2e}
\usepackage{color}

\newcommand{\blind}{0}

\newcommand{\E}{\ensuremath{\mathbf{E}}}
\DeclareMathOperator*{\argmin}{arg\,min}
\newtheorem{theorem}{Theorem}
\newtheorem{proposition}{Proposition}
\newtheorem{lemma}{Lemma}
\newtheorem{corollary}{Corollary}

\newtheorem{remark}{Remark}

\newtheorem{assumption}{Assumption}

\begin{document}

\def\spacingset#1{\renewcommand{\baselinestretch}%
{#1}\small\normalsize} \spacingset{1}

%%%%%%%%%%%%%%%%%%%%%%%%%%%%%%%%%%%%%%%%%%%%%%%%%%%%%%%%%%%%%%%%%%%%%%%%%%%%%%

\if0\blind
{
  \title{\bf Functional Partial Least-Squares: Adaptive Estimation and Inference\thanks{An earlier version of this paper circulated under the title ``Theoretical Comparison of Functional Principal Component Analysis and Functional Partial Least Squares." The authors thank participants at the SCSE Conference in Quebec (2019), the EC² Conference in Paris (2020), the NBER-NSF Time Series Conference at Rice University (2021), the ESIF Economics and AI+ML Meeting in Ithaca (2024), the Triangle Econometrics Conference (2024), and Econometrics in Rio (2024) for their valuable feedback. We are especially grateful to Ryan Borhani for outstanding research assistance. Marine Carrasco acknowledges partial financial support from NSERC and FQRSC.}}
  \author{Andrii Babii\\
    Department of Economics, UNC-Chapel Hill \\
    and \\
    Marine Carrasco\\
    Department of Economics, University of Montreal \\
    and \\
    Idriss Tsafack\\
    Department of Economics, University of Montreal \\}
  \maketitle
} \fi

\if1\blind
{
  \bigskip
  \bigskip
  \bigskip
  \begin{center}
    {\LARGE\bf Functional Partial Least-Squares: Adaptive Estimation and Inference\thanks{An earlier version of this paper circulated under the title ``Theoretical Comparison of Functional Principal Component Analysis and Functional Partial Least Squares." The authors thank participants at the SCSE Conference in Quebec (2019), the EC² Conference in Paris (2020), the NBER-NSF Time Series Conference at Rice University (2021), the ESIF Economics and AI+ML Meeting in Ithaca (2024), the Triangle Econometrics Conference (2024), and Econometrics in Rio (2024) for their valuable feedback. We are especially grateful to Ryan Borhani for outstanding research assistance. Marine Carrasco acknowledges partial financial support from NSERC and FQRSC.}}
\end{center}
  
} \fi

\bigskip
\begin{abstract}
We study the functional linear regression model with a scalar response and a Hilbert space-valued predictor, a canonical example of an ill-posed inverse problem. We show that the functional partial least squares (PLS) estimator attains nearly minimax-optimal convergence rates over a class of ellipsoids and propose an adaptive early stopping procedure for selecting the number of PLS components. In addition, we develop new test that can detect local alternatives converging at the parametric rate which can be inverted to construct confidence sets. Simulation results demonstrate that the estimator performs favorably relative to several existing methods and the proposed test exhibits good power properties. We apply our methodology to evaluate the nonlinear effects of temperature on corn and soybean yields.
\end{abstract}
\noindent%
\vfill

\newpage
\spacingset{1.9} % DON'T change the spacing!

\section{Introduction}
With the increasing availability of data, functional data analysis has become widely applied across fields such as chemometrics, climate science, and economics. In this paper, we study a linear functional regression model with a scalar response $Y$ and functional predictor $X$: 
\begin{equation}                                            
	Y=\int_{0}^{1}\beta (s)X(s)\mathrm{d}s+\varepsilon ,\qquad \mathbf{E}%
	[\varepsilon X]=0.  \label{eq:regression}
\end{equation}

The primary objective is to estimate the functional slope $\beta $ to predict the response $Y$. Both $X$ and $\beta $ lie in an infinite-dimensional Hilbert space, and the high dimensionality of $\beta$ leads to an ill-posed inverse problem. Consistent estimation of the slope coefficient thus requires either a dimension-reduction technique or some form of regularization.  

Two main approaches have been popular in the functional data analysis literature: penalization methods, e.g.,  \cite{cardot2003spline}, \cite{li2007rates}, \cite{crambes2009smoothing}, \cite{yuan2010reproducing}, \cite{cai2012minimax}, \cite{florens2015instrumental}, and functional principal component analysis (PCA), e.g., \cite{cardot1999functional}, \cite{ramsay2002applied}, \cite{cai2006prediction}, \cite{hall2007methodology}. The PCA-based approach approximates $\beta$ using a finite expansion over the leading principal components, which correspond to the eigenfunctions of the covariance operator of X. As noted by \cite{jolliffe1982note}, this method performs well when the response Y is primarily correlated with the leading principal components. Moreover, accurate recovery of the principal components requires a sufficient separation between the eigenvalues of the covariance operator. 

In this paper, we explore an alternative method, functional partial least squares (PLS). PLS is a widely used technique in the statistical learning, see \cite{friedman2009elements}, \cite{frank1993}; economics and finance, see \cite{carrasco2016sample} and \cite{kelly2015three}; chemometrics, see  \cite{helland1988structure}, \cite{wold1984collinearity}. However, it has been somewhat less commonly employed in the empirical functional data analysis. The method constructs components as linear transformations of the predictors X designed to maximize their correlation with the response Y. As a result, fewer components are typically needed to achieve good predictive performance compared to PCA. It was first introduced in \cite{preda2005pls} to solve the high-dimensional problems with multicolinearity associated with the scalar-on-function linear model. Several interesting papers compare PCA and various functional PLS estimators with data-driven choice of latent components in simulations; see \cite{reiss2007functional}, \cite{Kraemer2008}, \cite{baillo2009}, \cite{aguilera2010basis}, \cite{febrero2017overview}, and \cite{saricam2022partial}. This literature concludes that the prediction ability of functional PCA and PLS approaches is similar but functional PLS requires fewer components and provides a much more accurate estimation of the parameter function than PCA.

However, the existing literature lacks supporting theoretical results, which are challenging to obtain given that the PLS estimator depends non-linearly on the response variable and is computed iteratively. An important step toward the theoretical analysis of functional PLS was made by \cite{delaigle2012methodology}, who introduced an alternative PLS (APLS). Nevertheless, it remains unknown whether: (1) functional PLS achieves (nearly) minimax-optimal convergence rates under weak conditions; (2) a rate-adaptive method to select the number of PLS components exists; and (3) inference for functional PLS can be conducted.\footnote{Deriving an approximation to the distribution of PLS is particularly challenging because PLS is a nonlinear estimator converging at a nonparametric rate.} This paper fills these gaps and provides a comprehensive, rigorous theoretical analysis of functional PLS.

The paper makes several original contributions. First, we derive the convergence rate of our estimator and the prediction error under a source condition. This condition measures the complexity of the problem through the so-called degree of ill-posedness and relates the slope coefficient $\beta$ to the spectral decomposition of the covariance operator. Our results do not require assuming a separation between adjacent eigenvalues, as is common for PCA (see, e.g., \cite{hall2007methodology}), and hold even in the presence of repeated eigenvalues. Second, we establish a lower bound on the minimax convergence rate and show that our estimator is (nearly) minimax-optimal. Because the optimal number of PLS components depends on the unknown degree of ill-posedness, we propose an adaptive early stopping rule for selecting the number of components. We show that this single, early stoping rule yields a rate-optimal PLS estimator \textit{simultaneously} for estimation and prediction errors with high probability. We also characterize how the selected number of components evolves with sample size under various scenarios. 

Lastly, we develop new test and confidence sets for PLS. We show that the test can detect alternatives converging at a parametric rate. Interestingly, while early stopping is crucial for estimation and prediction to prevent overfitting, it should be avoided for inference. An efficient iterative algorithm to compute the estimator is provided. Our simulation results reveal that the estimator outperforms several alternative methods combined with cross-validation in terms of estimation error and remains competitive for prediction. We also find that our test has excellent power properties with samples as small as $n=100$ observations.

To establish our theoretical results, we rely on the inverse problem literature and exploit the close connection between PLS and the conjugate gradient method, as presented in \cite{hanke1995conjugate}, \cite{engl1996regularization}, and \cite{blanchard2016convergence}. To the best of our knowledge, this connection has not previously been established in the functional data analysis literature. Recently and independently of our work, \cite{gupta2023convergence} proposed an estimator in a reproducing kernel Hilbert space (RKHS) generated by a specified kernel and used conjugate gradient methods to regularize the solution. Although their method is similar to ours, the results are not directly comparable. Their main result concerns the convergence rate of an estimator in an RKHS, established under different assumptions. In particular, they assume a polynomial decay rate of certain operator eigenvalues, while we do not require any decay assumptions. Their source condition is also different, and it is unclear which one is weaker. Similarly, \cite{lin2021kernel} studied conjugate gradient methods in the context of linear approximations to nonparametric regression with randomized sketches and Nyström sampling. While their results apply to general Hilbert spaces under stronger assumptions, they do not explicitly connect to the functional data analysis literature; for example, their simulations focus on linear approximations of $\E[Y|X]=f(X)$ with real-valued data $Y,X\in\mathbb{R}$ and Sobolev RKHS kernels. Neither paper shows that the proposed estimators are rate-adaptive or provide a practical early stopping rule for selecting the number of PLS components.

To illustrate the practical relevance of our results, we apply our method to climate science. Using a fine-grained county-level dataset of U.S. crop yields and temperatures recorded over 70 years, we estimate the impact of temperature on crop yields. We find that the critical temperature at which annual crop yields begin to decline is around 30${}^\circ$C, consistent with \cite{schlenker2009nonlinear}, who relied on highly parameterized least-squares estimators. Our method provides additional insights by examining how the temperature effect curves have evolved over time. Interestingly, we find that the detrimental effects of high temperatures on corn and soybean yields have diminished over time, likely reflecting farmers’ adaptive actions, including the use of more resilient crops and improved irrigation systems. However, this finding is not conclusive when accounting for statistical uncertainty due to the small number of observed extreme temperatures.

The rest of the paper is organized as follows. Section~\ref{sec:estimator} introduces the functional regression model and the functional PLS estimator. Section~\ref{sec:adaptive} establishes the theoretical properties, including the convergence rates for estimation and prediction errors, the minimax lower bound, and the adaptivity of the early stopping rule. Section~\ref{sec:inference} develops inference for PLS. Section~\ref{sec:mc} presents a Monte Carlo study. Section~\ref{sec:application} discusses an empirical application to nonlinear temperature effects in agriculture. Section~\ref{sec:conc} concludes. Supplementary Material provides all proofs, comparisons to alternative estimators, and additional simulation results.

\section{Functional Regression and PLS Estimator}\label{sec:estimator}
\subsection{Functional Linear Regression}
Throughout the paper, we consider a generalized version of the functional linear regression model
\begin{equation*}
	Y = \langle \beta,X\rangle + \varepsilon,\qquad \mathbf{E} [\varepsilon X] =
	0,
\end{equation*}
where $(Y,X)\in \mathbb{R}\times \mathbb{H}$, $\beta \in \mathbb{H}$ is the unknown functional slope coefficient, and $(\mathbb{H}%
,\langle.,.\rangle)$ is a separable Hilbert space with the induced norm $\|.\|=\sqrt{\langle.,.\rangle}$. The model in equation (\ref{eq:regression}) corresponds to the Hilbert space of square integrable functions, $\mathbb{H}=L_2[0,1]$, with the norm induced by the inner product $\langle f,g\rangle = \int_0^1f(s)g(s)$.

If $\mathbb{E}[X]=0$, the covariance restriction $\mathbf{E} [\varepsilon X]=0$ implies that the slope coefficient $\beta \in \mathbb{H}$ solves the moment condition 
\begin{equation}  \label{eq:moment}
	r := \mathbf{E} [YX] = \mathbf{E} [(X\otimes X)\beta] =: K\beta,
\end{equation}
where $r\in \mathbb{H}$ and $K:\mathbb{H}\to \mathbb{H}$ is a compact covariance operator with summable eigenvalues whenever $\mathbf{E}
\|X\|^2<\infty$. It is well-known that the inverse operator $K^{-1}$ is discontinuous and solving the equation $K\beta = r$ for $\beta$ is an
ill-posed inverse problem; see \cite{carrasco2007linear},  \cite{engl1996regularization}, and \cite{hanke1995conjugate}.

Roughly speaking, there are two popular strategies to regularize such problems:

\begin{itemize}
\item[(a)] replace $K^{-1}$ with a continuous operator $R_\alpha(K)$ for some function $R_\alpha:\mathbb{R}_+\to \mathbb{R}_+$ satisfying  	$\lim_{\alpha \to 0^+}R_\alpha(\lambda)=\lambda^{-1}$. 
\item[(b)] solve the problem in a finite-dimensional subspace $\mathbb{H}_m\subset \mathbb{H}$, spanned by some fixed basis vectors $	h_1,h_2,\dots,h_m\in \mathbb{H}$. 
\end{itemize}

Examples of (a) include the Tikhonov regularization when $R_\alpha(\lambda)=(\alpha+\lambda)^{-1}$, the spectral cut-off when $R_\alpha(\lambda)=\lambda^{-1}\mathbf{1}_{\lambda \geq \alpha}$ and the 
Landweber iterations. On the other hand, the estimators in group (b), often solve the
empirical least-squares problem 
\begin{equation}  \label{eq:ls_problem}
	\min_{b\in \mathbb{H}_m}\| \mathbf{y} - T_nb\|^2_n,
\end{equation}
where $\|v\|_n^2=v^\top v/n,v\in \mathbb{R}^n$ and we put $\mathbf{y}%
=(Y_1,\dots,Y_n)^\top$ and 
\begin{equation*}
	\begin{aligned} T_n: \mathbb{H} & \to \mathbb{R}^n \\ b & \mapsto (\langle
		X_1,b\rangle,\dots,\langle X_n,b\rangle)^\top \end{aligned}
\end{equation*}
for an i.i.d. sample $(Y_i,X_i)_{i=1}^n$. The basis $(h_j)_{j=1}^m$ spanning $\mathbb{H}_m$ can be either fixed (e.g. Fourier, polynomials, splines,
wavelets) or adaptively constructed from the data. 

The data-driven bases are especially attractive since they can
adapt to the features of the population represented by the data and can
approximate the slope parameter $\beta \in \mathbb{H}$ more efficiently; see 
\cite{delaigle2012methodology}. The principal component analysis (PCA)%
\footnote{Using the PCA basis is also related to the spectral cut-off method described
	in (a).} and the partial least-squares (PLS) are two widely used methods to
construct adaptive bases in practice. The PCA basis is constructed by
identifying the directions in $\mathbb{H}$ where $X$ varies the most while
the PLS basis is constructed in a supervised way taking into account the
response variable as well. While the first $m$ elements of the PCA basis $%
h_1,\dots,h_m$ usually capture most of the variation of $X$, these are not
necessarily the most important vectors for approximating $\beta$ or
predicting the response variable $Y$. It is easy to find empirical examples,
where some of the last few low-variance components \textit{are} important;
see \cite{jolliffe1982note} who documented the issue on datasets used in
economics, climate science, chemical engineering, and meteorology.

\subsection{PLS estimator}

{\normalsize The PLS estimator constructs a data-driven basis iteratively maximizing the covariance with the response variable $Y$; see \cite{preda2005pls} who introduced it in the functional data analysis setting. The iterative nature of the estimator
makes it difficult to analyze its statistical properties. This prompted \cite%
{delaigle2012methodology} to formulate an alternative functional PLS solving
the problem in equation~(\ref{eq:ls_problem}) over the so-called Krylov
subspace 
\begin{equation*}
	\mathbb{H}_m = \mathrm{span}\left \{ \hat r, \hat K \hat r, \hat K^2 \hat
	r,\dots, \hat K^{m-1}\hat r \right \},
\end{equation*}
where 
\begin{equation*}
	\hat r = \frac{1}{n}\sum_{i=1}^nY_iX_i\qquad \text{and}\qquad \hat K = \frac{%
		1}{n}\sum_{i=1}^nX_i\otimes X_i
\end{equation*}
are the estimators of $r$ and $K$; see also \cite{wold1984collinearity}, \cite{helland1988structure}, and \cite{phatak2002exploiting} for the link between PLS and Krylov subspaces.

In this
paper, we study a version of the PLS estimator with $m\geq 1$ components,
denoted $\hat \beta_m$, characterized as a solution to the least-squares
problem 
\begin{equation*}
	\min_{b\in \mathbb{H}_m}\|T_n^*(\mathbf{y} - T_nb)\|^2
\end{equation*}
over the Krylov subspace $\mathbb{H}_m$. The least-squares objective
function is weighted by the adjoint operator of $T_n$ 
\begin{equation*}
	\begin{aligned} T_n^*: \mathbb{R}^n & \to \mathbb{H} \\ 
		\phi=(\phi_1,\dots,\phi_n)^\top & \mapsto
		\frac{1}{n}\sum_{i=1}^nX_i\phi_i
	\end{aligned}
\end{equation*}
and corresponds to minimizing the first-order conditions to the problem in
equation~(\ref{eq:ls_problem}), often called the normal equations.
Equivalently, $\hat \beta_m$ fits the empirical counterpart to the equation (%
\ref{eq:moment}) 
\begin{equation}  \label{eq:pls_problem}
	\min_{b\in \mathbb{H}_m}\left \| \hat r - \hat Kb \right \|^2
\end{equation}
as it is easy to see that $\hat r = T_n^*\mathbf{y}$ and $\hat K = T_n^*T_n$. Importantly, the PLS estimator formalized in equation~(\ref{eq:pls_problem}) corresponds to the conjugate gradient method with a self-adjoint operator $\hat K$, cf. \cite{hestenes1952methods}, known for its excellent regularization properties; see also \cite{hanke1995conjugate} and \cite{nemirovski1986regularizing}.\footnote{The method of conjugate gradients is one of the most efficient algorithms for solving high-dimensional systems of linear equations; see also \cite[Chapter 5]{nocedal1999numerical} and references therein.} We provide a more detailed
comparison between the two PLS estimators in the Supplementary Material,
Section~\ref{appn:suppl}. A related formulation of the PLS in reproducing
kernel Hilbert spaces (RKHS) was recently studied in an independent work of 
\cite{gupta2023convergence} who focus on the estimation error only and
impose assumptions different from ours. Our work can be seen as using a
kernel naturally adapted to the data which is unknown in practice. }

{\normalsize The estimator is uniquely defined for every $m\leq n_*$, where $%
n_*$ is the number of distinct non-zero eigenvalues of $\hat K$; see
Proposition~\ref{prop:uniquness} in the Supplementary Material. It is also
easy to see that for every $m\geq 1$, we have $\hat \beta_m =
\hat P_m(\hat K)\hat r$ for a polynomial $\hat P_m(\hat K) =
\sum_{j=1}^ma_j\hat K^{j-1}$ with coefficients $\mathbf{a}%
:=(a_1,\dots,a_m)^\top$, where $\hat P_0=0$ and $\hat \beta_0=0$.  The coefficients vector solves the system of $m$ linear equations 
\begin{equation*}
	\mathbf{K}\mathbf{a}= \mathbf{r},
\end{equation*}
where $\mathbf{K}:=\langle \hat K^j\hat r,\hat K^k\hat r \rangle_{1\leq
	j,k\leq m}$ and $\mathbf{r}:= \langle \hat K^j\hat r,\hat r\rangle_{1\leq
	j\leq m}$. From the practical point of view, it is more efficient to use an
iterative conjugate gradient algorithm that bypasses the (potentially
unstable) matrix inversion with an iterative multiplication by the operator $%
\hat K$; see Algorithm~\ref{algor1} in Section~\ref{sec:mc}. }

\section{Adaptive Estimation}\label{sec:adaptive}
In this section, we will show that the functional PLS estimator achieves the (nearly) optimal convergence rate on a class of ellipsoids. We consider an early stopping rule to select the number of PLS components and show that it adapts to the complexity of the ellipsoid. Lastly, we study how rapidly, the number of selected components increases with the sample size and make some comparisons to the PCA estimator.

\subsection{Optimal Convergence Rates}
Since the operators $K:\mathbb{H}\to\mathbb{H}$ and $\hat K:\mathbb{H}\to \mathbb{H}$ are self-adjoint and compact, by the spectral theorem
\begin{equation*}
	K = \sum_{j=1}^\infty\lambda_jv_j\otimes v_j\qquad \text{and}\qquad \hat K = \sum_{j=1}^{n}\hat \lambda_j \hat v_j\otimes \hat v_j,
\end{equation*}
where $\lambda_1\geq \lambda_2\geq \dots\geq 0$ and $\hat\lambda_1\geq \hat\lambda_2\geq \dots\geq \hat\lambda_{n}\geq0$ are the eigenvalues of $K$ and $\hat K$ and $(v_j)_{j=1}^\infty$ and $(\hat v_j)_{j=1}^{n}$ are the corresponding eigenvectors; see \cite{kress1999linear}, Theorem 15.16. Note that the sample covariance operator $\hat K$ is a finite-rank operator with at most $n_*\leq n$ distinct non-zero eigenvalues.

For any bounded and measurable function $\phi:\mathbb{R}_+\to\mathbb{R}_+$, we define functions of operators through their spectral decompositions:
\begin{equation*}
	\phi(K) := \sum_{j=1}^\infty\phi(\lambda_j)v_j\otimes v_j\qquad\text{and}\qquad \phi(\hat K) := \sum_{j=1}^{n_*} \phi(\hat\lambda_j)\hat v_j\otimes \hat v_j.
\end{equation*}
These definitions  are commonly used in the inverse problems literature; see \cite{engl1996regularization}.

The following inequalities for the operator norm will be often used:
\begin{equation}\label{eq:operator_bound}
	\|\phi(K)\|_{\rm op} \leq \sup_{\lambda\in[0,\lambda_1]}|\phi(\lambda)|\qquad \text{and}\qquad \|\phi(\hat K)\|_{\rm op} \leq \sup_{\lambda\in[0, \hat \lambda_1]}|\phi(\lambda)|,
\end{equation}
where $\|A\|_{\rm op} =\sup_{\|x\|=1}\|Ax\|$.

We shall introduce several relatively mild assumptions on the distribution of the data next.
\begin{assumption}\label{as:data}
	$(X_{i},Y_{i})_{i=1}^n$ are i.i.d.\ copies of  $(X,Y)$ with $\E[X]=0$, $\E\|X\|^4<\infty$, and $\E[\varepsilon^2|X]\leq \sigma^2<\infty$.
\end{assumption}
Assumption~\ref{as:data} imposes mild restrictions on the data-generating process. Note that $\E\|X\|^4<\infty$ is satisfied when $X$ is a Gaussian process in $\mathbb{H}$. It implies that $K$ is a nuclear operator and, hence, compact.

\begin{assumption}\label{as:id}
	The operator $K:\mathbb{H}\to\mathbb{H}$ does not have zero eigenvalues.
\end{assumption}
Assumption~\ref{as:id} ensures that the slope parameter $\beta$ is identified. If this assumption is violated, the focus would shift to the identified component of $\beta$ within the orthogonal complement of the null space of $K$; see \cite{babii2017completeness} and \cite{engl1996regularization}.

\begin{assumption}\label{as:complexity}
	For some $\mu,R,C>0$,	the slope parameter $\beta$ and the operator $K$ belong to the class
	\begin{equation*}
		\mathcal{S}(\mu,R,C) = \left\{\beta\in\mathbb{H},\; K:\mathbb{H}\to\mathbb{H}:\quad \sum_{j=1}^\infty\frac{\langle \beta,v_j\rangle^2}{\lambda_j^{2\mu}} \leq R^2,\quad \sum_{j=1}^\infty\lambda_j\leq C \right\}.
	\end{equation*}
\end{assumption}
Assumption~\ref{as:complexity} describes the complexity of the ill-posed inverse problem in terms of the smoothness of $\beta$ and the smoothing properties of the operator $K$. The parameter $\mu$ is known as the degree of ill-posedness. It restricts the rate of decline of the generalized Fourier coefficients $\langle \beta,v_j\rangle_{j\geq 1}$ relatively to the eigenvalues of $K$. A larger value of $\mu$ means that it is easier to estimate the slope coefficient $\beta$; see also \cite{carrasco2007linear}. Recall also that the summability of eigenvalues holds whenever $\E\|X\|^2<\infty$. Note that Assumption~\ref{as:complexity} is weaker than what is typically used to analyze the PCA estimators and does not require restricting the spacing between eigenvalues, cf. \cite{hall2007methodology}.

Consider now the so-called residual polynomial $\hat Q_m(\lambda)=1-\lambda\hat P_m(\lambda)$, deriving its name from the identity $\hat r - \hat K\hat\beta_m = \hat Q_m(\hat K)\hat r$. It is known that the polynomial, $\hat Q_m$, has $m$ distinct real roots, denoted $\hat\theta_1>\hat\theta_2>\dots>\hat\theta_m>0$. The sum of inverse of these roots,
\begin{equation*}
	|\hat Q_m'(0)| = \sum_{j=1}^m\frac{1}{\hat\theta_j},
\end{equation*}
plays an important role in the analysis of the conjugate gradient regularization; see Lemma~\ref{lemma:polynomials} in the Supplementary Material. 

Our first result characterizes the convergence rate of the estimation and prediction errors of the PLS estimator.
\begin{theorem}\label{thm:pls_rate}
	Suppose that Assumptions~\ref{as:data}, \ref{as:id}, and \ref{as:complexity} are satisfied. Then for every $s\in[0,1]$, we have
	\begin{equation*}
		\left\|K^s(\hat{\beta}_m - \beta)\right\|^2  = O_P\left(|\hat Q_m'(0)|^{2(1-s)}n^{-1} + |\hat Q_m'(0)|^{-2(\mu+s)} + |\hat Q'_m(0)|^{-2s}n^{-\mu\wedge 1}\right),
	\end{equation*}
	provided that $|\hat Q_m'(0)|=O_P(n^{1/2})$.
\end{theorem}
Note that the last condition in Theorem~\ref{thm:pls_rate} imposes that the number of components $m$ does not increase too fast with the sample size and is not binding. In fact, it is optimal to have $|\hat Q_m'(0)|\sim n^{\frac{1}{2(\mu+1)}}$, in which case we obtain the following convergence rate
\begin{equation*}
	\left\|K^s(\hat{\beta}_{m} -\beta )\right\|^{2} = O_P\left(n^{-\frac{\mu+s}{\mu+1}}\right).
\end{equation*}
When $s=0$, this shows that the convergence rate of PLS in the Hilbert space norm is of order $n^{-\frac{\mu}{\mu+1}}$. On the other hand, when $s=1/2$, we obtain the convergence rate of the out-of-sample prediction error, since
\begin{equation*}
	\E_X\langle X,\hat\beta_{m} - \beta\rangle^2  = \left\|K^{1/2}(\hat\beta_{m} - \beta)\right\|^2,
\end{equation*}
where $\E_X$ is taken with respect  to $X$, independent of $(Y_i,X_i)_{i=1}^n$. 

\begin{remark}
	Note that the consistency of functional PLS has been previously established in \cite{delaigle2012methodology} assuming that the eigenvalues are summable only which can be stated as $\mu=0$. Characterizing the speed of convergence, however, requires more regularity with $\mu>0$.
\end{remark}
The following result shows that no estimator can achieve a faster than $n^{-\frac{\mu +s}{\mu +1}}\log^{-b} n$ rate on the class $\mathcal{S}(\mu,R,C)$.
\begin{theorem}\label{thm:lower_bound}
	For every $s\in[0,1/2]$, there exists $A<\infty$ such that
	\begin{equation*}
		\liminf_{n\to\infty} \inf_{\hat\beta}\sup_{(\beta,K)\in\mathcal{S}(\mu,R,C)}\mathrm{Pr}\left(\left\|K^s(\hat\beta - \beta) \right\| \geq An^{-\frac{\mu+s}{2(\mu+1)}}\log^{-b/2} n\right)>0,
	\end{equation*}
	where $b>2(\mu+s)$ and the infimum is over all estimators. 
\end{theorem}
Therefore, we conclude that the PLS estimator $\hat\beta_m$ achieves the (nearly) optimal convergence rate on $\mathcal{S}(\mu,R,C)$, simultaneously for the estimation $(s=0)$ and prediction $(s=1/2)$ errors.\footnote{It is possible to avoid the $1/\log n$ factor by considering the larger class of Hilbert--Schmidt operators.} 

\begin{remark}\label{remark:rkhs}
	$\beta$ in $\mathcal{S}(\mu,R,C)$ also belongs to the RKHS generated by the covariance with $\mu=1/2$. Our paper shows that the minimax-optimal rate in this case is $O_P(n^{-2/3})$. Our result does not seem to be directly comparable to \cite{cai2012minimax} in this case, since their minimax-optimal $O_P(n^{-2r/(2r+1)})$ is obtained under the additional assumption that the eigenvalues decline polynomially fast, i.e. $\lambda_j\sim j^{-r}$.\footnote{Note that similar assumptions are also made in \cite{gupta2023convergence} and \cite{lin2021kernel}.} Note also that $r=1$ cannot be satisfied since the existence of the covariance operator holds under $\E\|X\|^2<\infty$ which is equivalent to $\sum_{j=1}^\infty\lambda_j<\infty$. In contrast, our class $\mathcal{S}(\mu,R,C)$ with $\mu=1/2$ does not require specifying the decay rate of eigenvalues. 
\end{remark}

\subsection{Adaptive PLS estimator}
Next, we look at the adaptive PLS estimator, where the number of components is selected using the early stopping rule described in the following assumption.
\begin{assumption}\label{as:stopping}
	We select $\hat m$ such that
	\begin{equation*}
		\min\left\{m\geq 0:\; \left\|\hat r - \hat K\hat{\beta}_m \right\| \leq \tau\sigma\sqrt{\frac{2\E\|X\|^2}{\delta n}} \right\}.
	\end{equation*}
	for $\tau > 1$ and some $\delta\in(0,1)$.
\end{assumption}

Assumption~\ref{as:stopping} states that the selected number of PLS components $\hat m$ equals the first non-negative integer $m$ such that the norm of the fitted ``moment" is smaller than a certain threshold; see Supplementary Material, Section~\ref{suppl:simulations} for a practical implementation of this early stopping rule. Note that the number of selected components is finite since $\hat m\leq n_*$, where $n_*$ is the number of distinct non-zero eigenvalues of $\hat K$; see Proposition~\ref{prop:uniquness} in the Supplementary Material. In fact, the norm of ``residual" is zero for $m\geq n_*$ in which case we have perfect overfitting.

The following result shows that the early stopping rule in Assumption~\ref{as:stopping} is adaptive to the unknown degree of ill-posedness $\mu>0$.
\begin{theorem}\label{thm:pls_adaptation}
	Suppose that Assumptions~\ref{as:data}, \ref{as:id}, \ref{as:complexity}, and \ref{as:stopping} hold with $\delta\geq 1/n$. Then
	\begin{equation*}
		\left\|K^s(\hat{\beta}_{\hat m} -\beta )\right\|^{2} = O\left((\delta n)^{-\frac{\mu+s}{\mu+1}}\right)
	\end{equation*}
	with probability at least $1-\delta$ for every $s\in[0,1]$.
\end{theorem}
Taking $\delta_n=1/\log n$ in Assumption~\ref{as:stopping}, we obtain from Theorem~\ref{thm:pls_adaptation} the convergence rate of the estimation and prediction errors of PLS with the number of components is selected iteratively with the early stopping rule:
\begin{equation*}
	\left\|K^s(\hat{\beta}_{\hat m} -\beta )\right\|^{2} = O_P\left(\left(\frac{\log n}{n}\right)^{\frac{\mu+s}{\mu+1}}\right).
\end{equation*}
Therefore, the adaptive PLS achieves the (nearly) optimal convergence rate simultaneously for the estimation and prediction errors without knowing the degree of ill-posedness $\mu>0$.

\subsection{Number of Selected Components}
In this section, we look at how rapidly the number of selected components in Assumption~\ref{as:stopping} increases with the sample size. First, we consider a somewhat conservative bound that does not impose any assumptions on the spectrum of the operator $K$.

\begin{theorem}\label{thm:number_of_components}
	Suppose that Assumptions~\ref{as:data}, \ref{as:id}, \ref{as:complexity}, and \ref{as:stopping} are satisfied with $\delta\geq 1/n$ and $\mu\geq 1$. Then with probability at least $1-\delta$
	\begin{equation*}
		\hat m = O\left( (n\delta)^\frac{1}{4(\mu+1)}\right).
	\end{equation*}
\end{theorem}
Taking $\delta = 1/\log n$, we obtain from Theorem~\ref{thm:number_of_components} that $\hat m = O_P\left((n/\log n)^\frac{1}{4(\mu+1)}\right)$. Next, we consider sharper estimates under additional assumptions imposed on the spectrum of the operator $K$.
\begin{theorem}\label{thm:number_of_components2}
	Suppose that Assumptions~\ref{as:data}, \ref{as:id}, \ref{as:complexity}, and \ref{as:stopping} are satisfied with $\delta\geq e/n$ and $\mu\geq 1$. Then with probability at least $1-\delta$
	\begin{itemize}
		\item[(i)] If $\lambda_{j} = O(j^{-2\kappa})$ for some $\kappa >0$, then
		\begin{equation*}
			\hat m = O\left((n\delta)^\frac{1}{4(\kappa+1)(\mu+1)}\right).
		\end{equation*}
		\item[(ii)] If $\lambda_{j} = O(q^{j})$ for some $q\in(0,1)$, then 
		\begin{equation*}
			\hat m = O\left(\log(n\delta)\right).
		\end{equation*}
	\end{itemize}
\end{theorem}
Theorem~\ref{thm:number_of_components2} shows that if the eigenvalues decline polynomially fast, then the selected number of components is $\hat m = O_P(n/\log n)^{\frac{1}{4(\kappa +1)(\mu+1)}}$ while in the case of the geometric decline, the number of selected components increases slowly with the sample size. Therefore, the adaptive early stopping rule will choose a smaller number of components if the eigenvalues of the operator $K$ decline faster and vice versa.

\section{Inference}\label{sec:inference}
\subsection{A Hypothesis Test}
We aim to test the null hypothesis against the fixed alternative hypothesis:
\begin{equation*}
	H_0:\; \beta = b \qquad \text{vs.} \qquad H_1:\; \beta \ne b,
\end{equation*}
for some known $b \in \mathbb{H}$. We are also interested in testing against a sequence of local alternative hypotheses
\begin{equation*}
	H_{1,n}:\; \beta = b + n^{-1/2}\Delta
\end{equation*}
for some $\Delta\in\mathbb{H}$.

Recall that our main result in Theorem~\ref{thm:pls_rate} shows that
\begin{equation*}
	\left\|K^s(\hat\beta_m - \beta)\right\|^2 = O_P\left(n^{-\frac{\mu+s}{\mu+1}}\right),\qquad \forall s\in[0,1],
\end{equation*}
provided  that the number of components $m$ is properly selected. The convergence rate is $O_P(n^{-1})$ for $s=1$ and it is slower than $O_P(n^{-1})$ for $s<1$. This suggests that the following statistic
\begin{equation*}
	T_n = n \left\| \hat{K}(\hat{\beta}_m - b) \right\|^2
\end{equation*}
may have a well-defined asymptotic distribution under $H_0$.\footnote{The statistic is similar to Wald's statistics which under homoskedasticity roughly corresponds to $s=-1/2$ and is not expected to converge at the parametric rate.}

Let $V = \mathbb{E}[\varepsilon^2 X \otimes X]$ be the variance operator with spectral decomposition
\begin{equation*}
	V = \sum_{j=1}^\infty\omega_j\varphi_j\otimes\varphi_j
\end{equation*}
and let $Z_j \overset{\text{i.i.d.}}{\sim} {N}(0,1)$. The following theorem establishes the limiting distribution of $T_n$ under the null and alternative hypotheses.
\begin{theorem}\label{thm:test}
	Suppose that Assumption~\ref{as:data} and \ref{as:id} are satisfied and the PLS estimator is computed using $m$ components such that $\|\hat{r} - \hat{K} \hat{\beta}_m\| = o_P(n^{-1/2})$. Then, under $H_0$,
	\begin{equation*}
		T_n \xrightarrow{d} \sum_{j=1}^\infty \omega_j Z_j^2,
	\end{equation*}
	while under $H_1$, $T_n\xrightarrow{\rm a.s.}\infty$. Moreover, under $H_{1,n}$,
	\begin{equation*}
		T_n \xrightarrow{d} \sum_{j=1}^\infty\omega_jZ_j^2 + 2\sum_{j=1}^\infty\omega_j^{1/2}Z_j\langle \varphi_j,K\Delta\rangle + \|K\Delta\|^2.
	\end{equation*}
\end{theorem}

\begin{remark}
	The requirement $\|\hat r - \hat K\hat\beta_m\| = o_P(n^{-1/2})$ in Theorem~\ref{thm:test} can be easily satisfied when the number of components $m$ is sufficiently large. Indeed, we know that $\|\hat r - \hat K\hat\beta_m\| =0$ for $m= n_*$, where $n_*\leq n$ is the number of non-zero eigenvalues of the sample covariance operator $\hat K$, cf. Supplementary Material, Proposition~\ref{prop:uniquness}.
\end{remark}

\begin{remark}
	Theorem~\ref{thm:test} illustrates an interesting phenomenon related to the optimal choice of the number of PLS components which amounts to early stopping of the conjugate gradient descent. While it is optimal to select a small number of components to prevent overfitting for estimation, this should be avoided for inference. In fact, it is evident from the proof of Theorem~\ref{thm:test} that the overfitted PLS with zero fitting error leads to more accurate asymptotic approximation; cf. \cite{bartlett2020benign} for benign overfitting in the linear regression.
\end{remark}

\begin{remark}
	The test has more power against the alternatives with larger values of
	\begin{equation*}
		\|K\Delta\|^2 = \sum_{j=1}^\infty\lambda_j^2\langle\Delta,v_j\rangle^2,
	\end{equation*}
	where $(\lambda_j,v_j)_{j=1}^\infty$ are the eigenvalues and eigenvectors of the operator $K$. Such alternatives are aligned with directions defined by the leading eigenvectors of $K$, i.e. directions in which it is easier to identify the slope function $\beta$.
\end{remark}

\begin{remark}
	Under homoskedasticity, $\E[\varepsilon^2|X]=\sigma^2$,  we have $V = \sigma^2K$. In this case, $V$ has eigenvectors $\varphi_j=v_j$ with corresponding eigenvalues $\omega_j = \sigma^2\lambda_j$.
\end{remark}

Theorem~\ref{thm:test} implies that the test based on $T_n$ has correct size under the null hypothesis  and is consistent under the fixed alternative hypotheses which we state as a trivial corollary below. To that end, let $z_{1-\alpha}$ be the $1-\alpha$ quantile of
\begin{equation*}
	T=\sum_{j=1}^\infty\omega_jZ_j^2.
\end{equation*}
The test rejects $H_0$ if $T_n>z_{1-\alpha}$. Then, the following result holds:
\begin{corollary}\label{cor:size_power}
	Suppose that the conditions of Theorem~\ref{thm:test} are satisfied. Then under $H_0$, we have
	\begin{equation*}
		\lim_{n\to\infty}\Pr(T_n > z_{1-\alpha}) = \alpha
	\end{equation*}
	while under $H_1$, we have
	\begin{equation*}
		\lim_{n\to\infty}\Pr(T_n > z_{1-\alpha}) = 1.
	\end{equation*}
	Moreover, under $H_{1,n}$
	\begin{equation*}
		\lim_{n\to\infty}\Pr(T_n > z_{1-\alpha}) \uparrow 1\qquad \mathrm{as}\qquad \|K\Delta\|\uparrow\infty.
	\end{equation*}
\end{corollary}

It is also easy to show that the critical value $z_{1-\alpha}$ can be consistently estimated with simple nonparametric bootstrap based on the i.i.d. draws from the sample $(Y_i,X_i)_{i=1}^n$. Alternatively, one can simulate the asymptotic critical values using the eigenvalues $\hat\omega_j$ of the sample variance operator $\hat V$.

\subsection{A Confidence Set}
Inverting the test, we obtain a confidence set
\begin{equation*}
	C_{1-\alpha} = \left\{b\in\mathbb{H}:\; T_n(b)\leq z_{1-\alpha}\right\},
\end{equation*}
where $z_{1-\alpha}$ is the quantile of order $1-\alpha$ of $T=\sum_{j=1}^\infty\omega_jZ_j^2$. Indeed, since $T_n\xrightarrow{d}T$, the confidence set has the coverage probability $1-\alpha$:
\begin{equation*}
	\lim_{n\to\infty}\Pr(\beta \in C_{1-\alpha}) = \lim_{n\to\infty}\Pr(T_n(\beta)\leq z_{1-\alpha}) = \Pr\left(T \leq z_{1-\alpha}\right) = 1-\alpha.
\end{equation*}

Since $\beta$ is an infinite-dimensional object, computing this confidence set numerically requires finite-dimensional approximations which can be done as follows. Let $(h_j)_{j=1}^\infty$ be a basis of $\mathbb{H}$. Then we approximate $z_{1-\alpha}$ numerically by computing
\begin{equation*}
	C_{1-\alpha}^J = \left\{(b_1,\dots,b_J)\in\mathbb{R}^J:\; T_n\left(\sum_{j=1}^Jb_jh_j\right)\leq z_{1-\alpha}\right\}.
\end{equation*}
For instance, in the functional data analysis, we often have $\mathbb{H}=L_2[0,1]$ in which case, we can use the Fourier basis or Legendre polynomials. If $\mathbb{H}=L_2(\mathbb{R})$, the Hermite polynomials are a natural choice; see Supplementary Material, Section~\ref{suppl:simulations} for numerical implementation and simulation results.

\section{Monte Carlo Experiments}\label{sec:mc}
In this section, we conduct several Monte Carlo experiments to evaluate the finite sample performance of the PLS estimator. We simulate the i.i.d. samples $(Y_i,X_i)_{i=1}^n$ from the functional linear model
\begin{equation*}
	Y_i = \int_0^1X_i(s)\beta(s)ds + \varepsilon_i, \qquad \varepsilon_i\sim_{\rm i.i.d.} N(0,1),
\end{equation*}
where the predictors $X_i$ belong to the Hilbert space of square-integrable functions with respect to the Lebesgue measure, denoted $\mathbb{H}=L^{2}[0,1]$. The functional predictor is generated as
\begin{equation*}
	X_i(s) = \sum^{100}_{j=1}\sqrt{\lambda_{j}}u_{j}v_{j}(s),\qquad u_j\sim_{\rm i.i.d.} N(0,1).
\end{equation*}

We specify the slope parameter $\beta\in L_2[0,1]$ and the spectrum $(\lambda_j,v_j)_{j\geq 1}$ correspond to one of the following three models:
\begin{itemize}
	\item  \textbf{Model 1:} $\beta(s)  = \sum^{100}_{j=1}\beta_{j}v_{j}(s) $ with $\beta_{j} = 4j^{-2.7}$ and $\lambda_{j} = 2j^{-1.1}$ for all $j\geq 1$, $v_{j}(s) = \sqrt{2}\cos(j\pi s),$ $j\geq 1,2,3,\dots$, and we redefine $v_1(s)=1$.
	\item \textbf{Model 2:} same as Model 1, but with $\beta_j=4$ for all $j=1,\dots,5$.
	\item \textbf{Model 3:} same as Model 1, but with $\lambda_j=2$ for all $j=1,\dots,5$.
\end{itemize}
All three models satisfy Assumption~\ref{as:complexity} with the same complexity parameter. Model 2 emphasizes the importance of the first five coefficients of $\beta_j$, making the estimation and prediction tasks more challenging. Model 3 introduces repeated eigenvalues, causing the first five eigenvectors to be non-identifiable, although the slope parameter remains identifiable.

We compute the PLS estimator using Algorithm~\ref{algor1}, which is numerically equivalent to the estimator given by equation (\ref{eq:pls_problem}); see \cite{hanke1995conjugate}, Algorithm 2.1 and Proposition 2.1. This approach avoids direct inversion of the empirical operator $\hat K$ by utilizing iterative multiplication, making it suitable for solving high-dimensional linear systems of the form $\hat K\hat\beta = \hat r$ with a symmetric matrix $\hat K$.

\begin{algorithm}[H]\label{algor1}
	\SetAlgoLined
	\KwResult{$\hat{\beta}_{m}$ }
	\textbf{Initialisation:}
	$\hat{\beta}_{0} = 0$,  $d_0=e_0=\hat r - \hat K\hat\beta_0$\;
	\For{$j = 0,1,\dots, m-1$}{
		1. Compute step size: $\alpha_j = \frac{\langle e_j,\hat K e_j\rangle}{\|\hat Kd_j\|^2}$\;
		2. Update slope coefficient: $\hat\beta_{j+1} = \hat\beta_j + \alpha_jd_j$\;
		3. Update fitted moment: $e_{j+1}  = e_j - \alpha_j\hat K d_j$\;
		4. Compute step size for the conjugate direction update: $\gamma_{j+1} = \frac{\langle e_{j+1},\hat Ke_{j+1}\rangle }{\langle e_j,\hat Ke_j\rangle}$\;
		5. Update conjugate direction vector: $d_{j+1} = e_{j+1} + \gamma_{j+1}d_j$\;
	}
	\caption{PLS algorithm for solving $\hat K\hat\beta = \hat r$.}
\end{algorithm}
The integrals in inner products and the operator $K$ are discretized using a simple Riemann sum approximation over a grid of $T=200$ equidistant points in $[0,1]$.\footnote{This provides a satisfactory approximation under weak assumption. Alternatively, one could employ quadrature rules with fewer points to reduce computational load.} The experiments feature $5,000$ replications, with each replication generating samples of size $n=100$. For each simulation experiment and an estimator $\hat\beta$, we compute:
\begin{itemize}
	\item The integrated squared error (ISE):
	\begin{equation*}
		\mathrm{ISE}(\hat\beta) = \int_0^1|\hat\beta(s) - \beta(s)|^2\mathrm{d}s;
	\end{equation*}
	\item the mean-squared prediction error (MSPE):
	\begin{equation*}
		\text{MSPE}(\hat\beta)  = \frac{1}{n}\sum_{i=1}^n(Y_i - \langle X_i,\hat\beta\rangle)^2,
	\end{equation*}
	where the estimator $\hat\beta$ is computed from an auxiliary independent sample of size $n$.
\end{itemize}
We compare the performance of our functional PLS estimator with the number of components chosen using our early stopping rule, against several alternative methods:\footnote{See Supplementary Material, Section~\ref{suppl:simulations}, for more details on the practical implementation.}
\begin{enumerate}
	\item A spline estimator, jointly selecting the spline degree and smoothing parameter via generalized cross-validation (GCV), following  \cite{crambes2009smoothing}.
	\item A principal component regression (PCA) estimator, where component selection is also done by GCV.
	\item A reproducing kernel Hilbert space (RKHS) estimator using the same kernel as in \cite{yuan2010reproducing}, equation (16).
	\item The alternative PLS method from \cite{delaigle2012methodology}, selecting components through 5-fold cross-validation.
\end{enumerate}
Note that \cite{crambes2009smoothing}, Proposition 2, demonstrates that the prediction error for the smoothing spline estimator is asymptotically equivalent to the infeasible optimal selection of tuning parameters. On the other hand, the kernel function employed in \cite{yuan2010reproducing} is not adaptively selected from the data.  We are not aware of any adaptivity results for other competing estimators, including PCA and the alternative PLS estimator with cross-validation. In contrast, our method for selecting the number of PLS components is adaptive for both estimation and prediction according to Theorem~\ref{thm:pls_adaptation}.

Figure~\ref{fig:estimation_prediction} summarizes the distributions of estimation (ISE) and prediction (MSPE) errors from $5,000$ simulations. The early stopped PLS estimator achieves the lowest median estimation error, with consistently lower variability across all models. The penalized spline estimator performs effectively when coefficients decline rapidly but exhibits notably larger errors in Model 2, which involves high-frequency slope coefficients. Nevertheless, it shows excellent predictive performance consistent with \cite{crambes2009smoothing}. The PLS estimator also demonstrates competitive predictive accuracy, although it does not consistently yield the lowest median prediction. The RKHS estimator achieves good prediction performance but inferior estimation across all models. The alternative PLS estimator generally performs comparably to our proposed PLS approach but exhibits notably larger prediction errors in Models 2 and 3. These higher errors may be attributed to numerical instability associated with finite-precision arithmetic, as discussed in \cite{delaigle2012methodology}. Finally, PCA demonstrates relatively modest performance across all the considered models. Moreover, we find in simulations that the FPLS estimator outperforms FPCA with a smaller number of components; see \cite{frank1993}. 

\begin{figure}[htbp]
	\centering
	
	% Row 1: Model 1
	\begin{subfigure}[b]{0.48\textwidth}
		\centering
		\includegraphics[width=\textwidth]{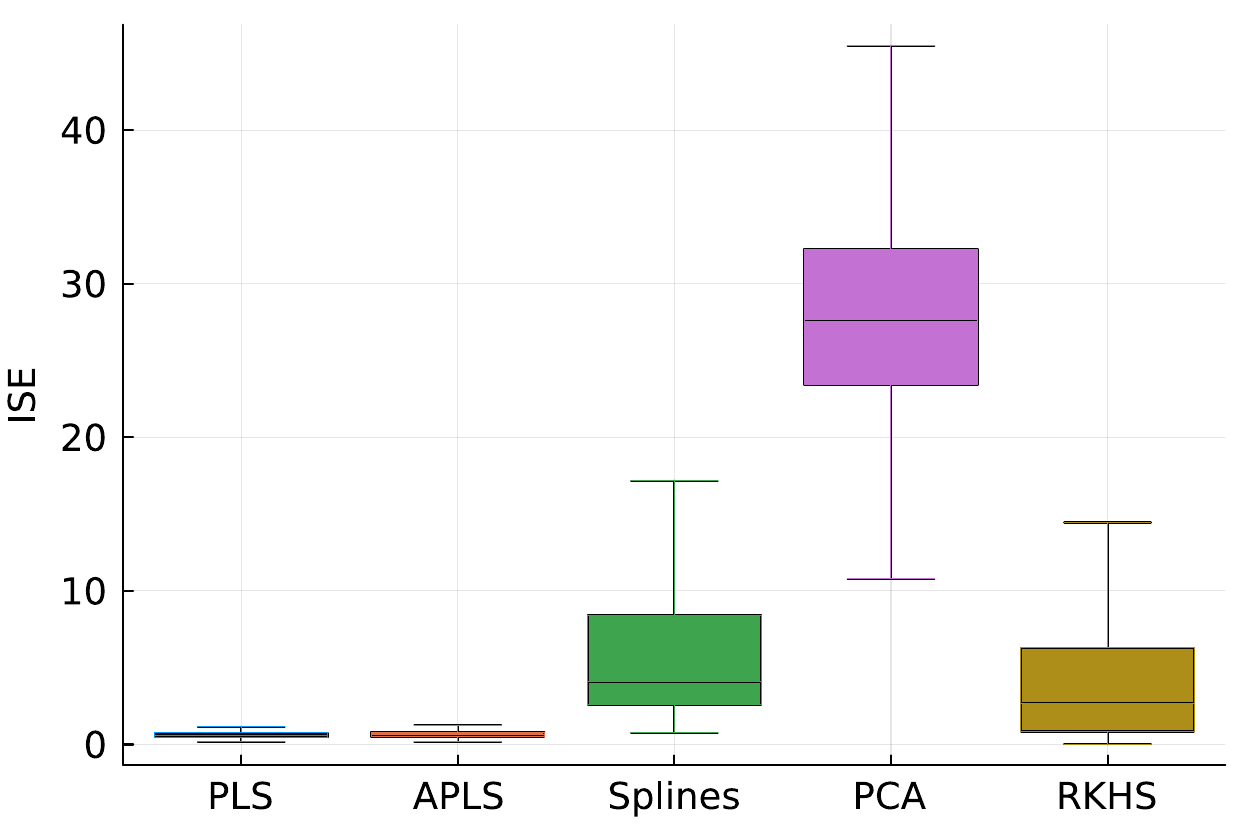}
		\caption{Model 1: Estimation Error (ISE)}
		\label{fig:model1_ise}
	\end{subfigure}
	\hfill
	\begin{subfigure}[b]{0.48\textwidth}
		\centering
		\includegraphics[width=\textwidth]{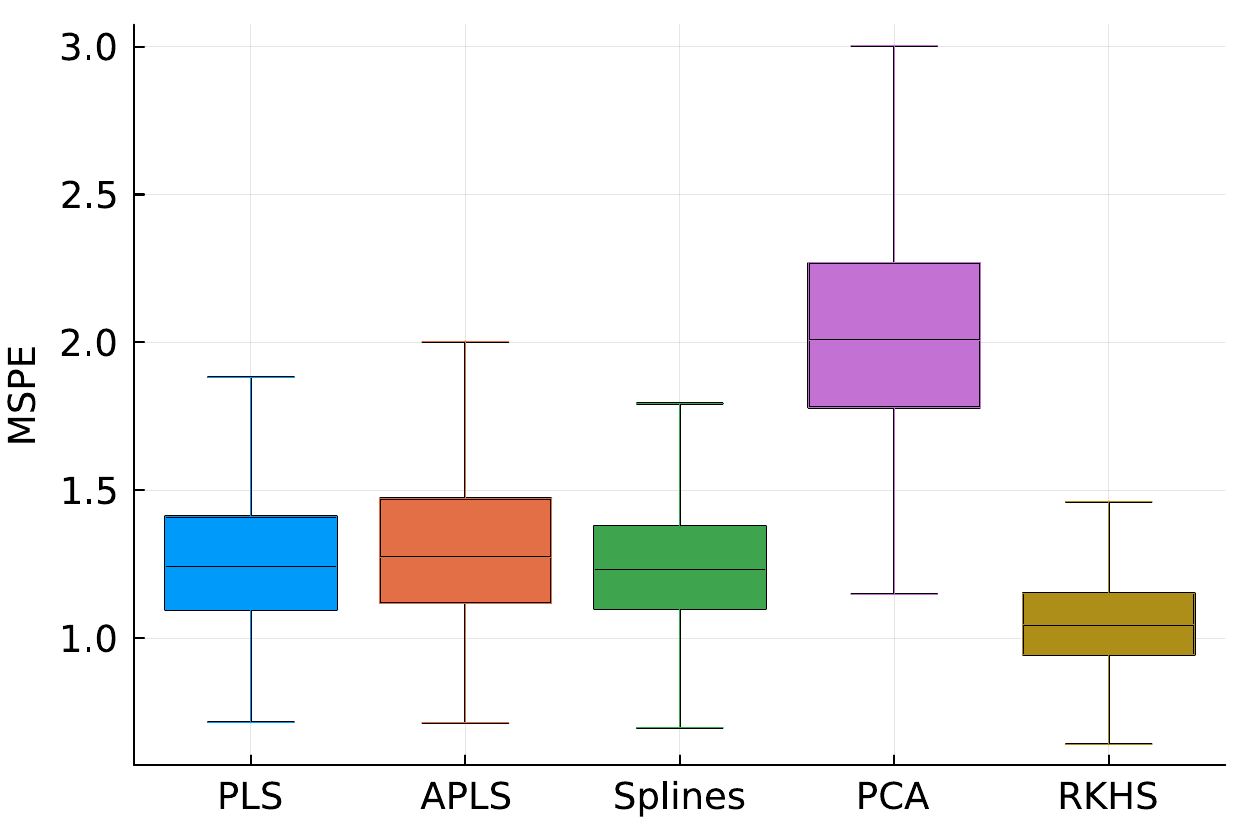}
		\caption{Model 1: Prediction Error (MSPE)}
		\label{fig:model1_spe}
	\end{subfigure}
	
	% Row 2: Model 2
	\begin{subfigure}[b]{0.48\textwidth}
		\centering
		\includegraphics[width=\textwidth]{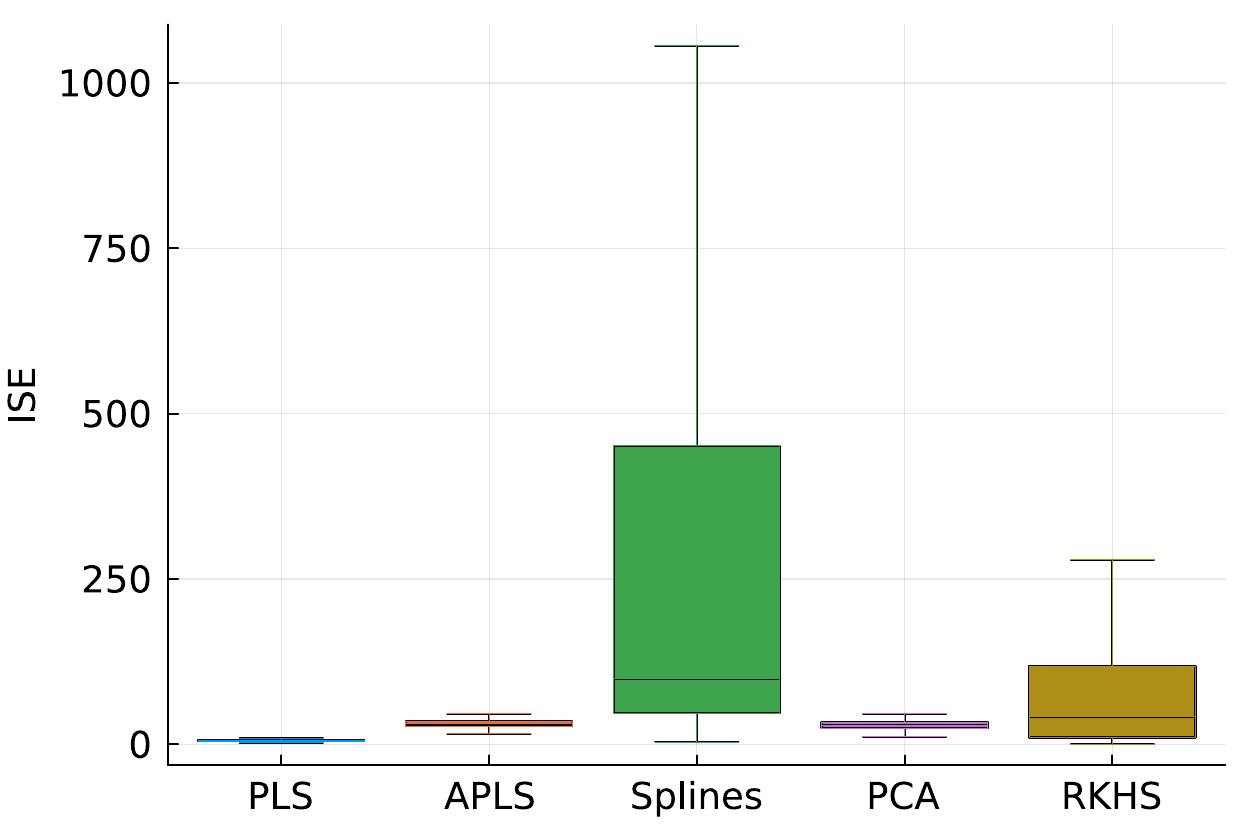}
		\caption{Model 2: Estimation Error (ISE)}
		\label{fig:model2_ise}
	\end{subfigure}
	\hfill
	\begin{subfigure}[b]{0.48\textwidth}
		\centering
		\includegraphics[width=\textwidth]{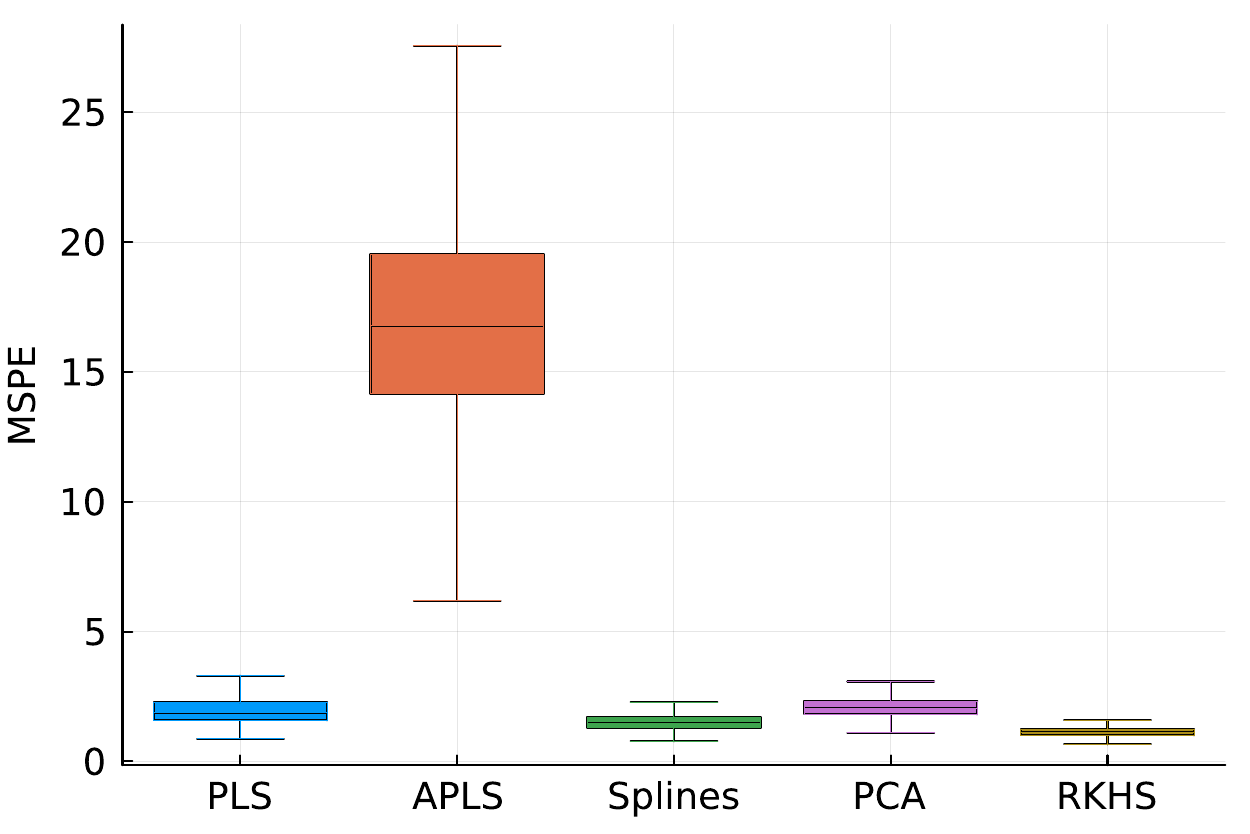}
		\caption{Model 2: Prediction Error (MSPE)}
		\label{fig:model2_spe}
	\end{subfigure}
	
	% Row 3: Model 3
	\begin{subfigure}[b]{0.48\textwidth}
		\centering
		\includegraphics[width=\textwidth]{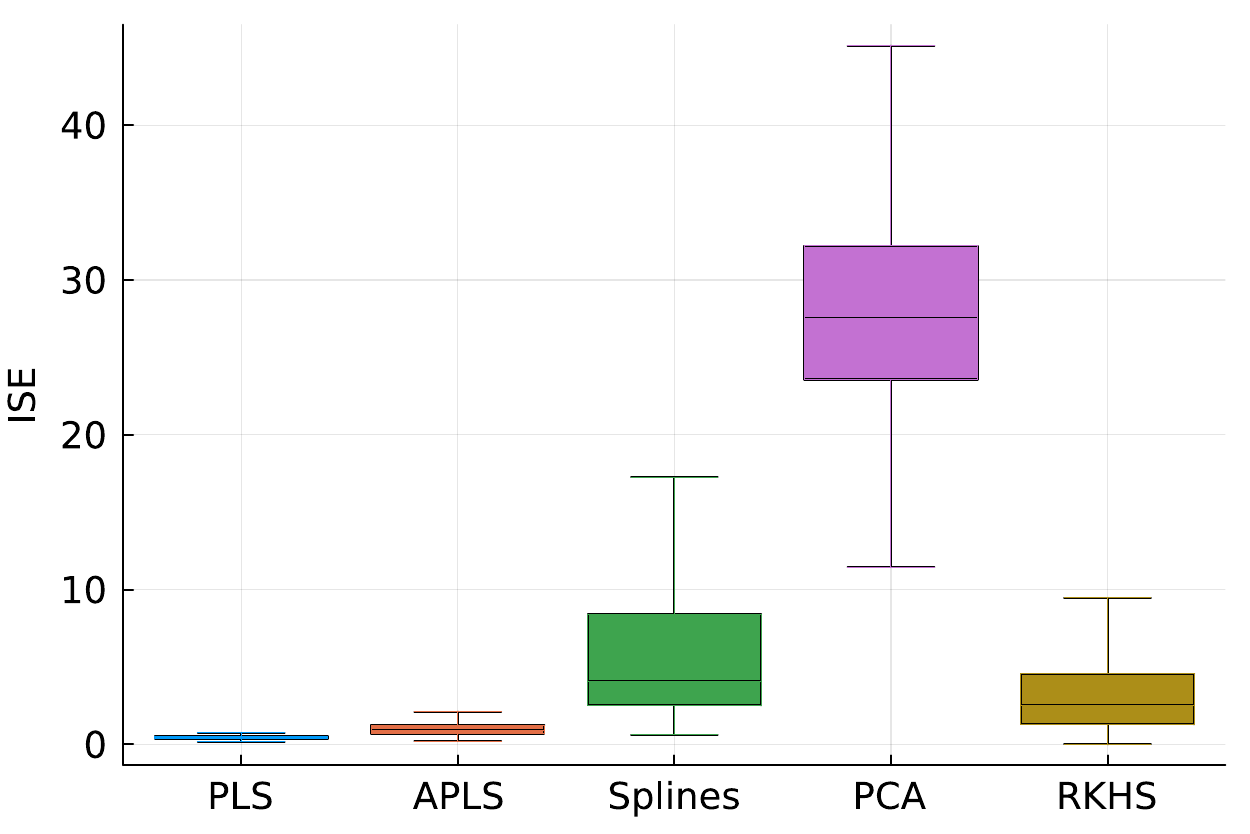}
		\caption{Model 3: Estimation Error (ISE)}
		\label{fig:model3_ise}
	\end{subfigure}
	\hfill
	\begin{subfigure}[b]{0.48\textwidth}
		\centering
		\includegraphics[width=\textwidth]{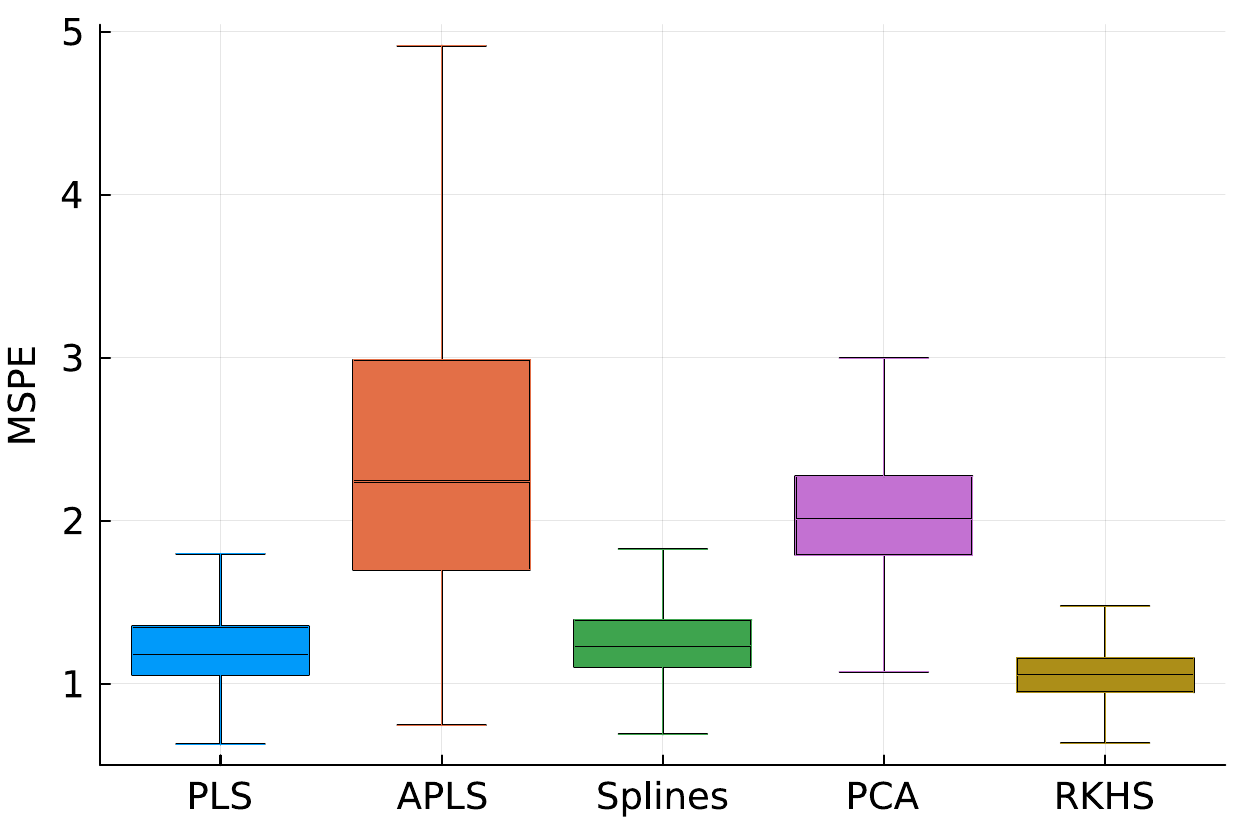}
		\caption{Model 3: Prediction Error (MSPE)}
		\label{fig:model3_spe}
	\end{subfigure}
	
	\caption{Distribution of estimation errors (left columns) and prediction errors (right columns) for Models 1–3. Each boxplot summarizes the distribution of errors across 5,000 Monte Carlo simulations.}
	\label{fig:estimation_prediction}
\end{figure}

Next, we evaluate the performance of the proposed test. We first assess whether the asymptotic distribution provides a good approximation to the finite-sample distribution under the null hypothesis $H_0$.  Given that the PLS fitting error does not substantially decrease after approximately 10 components, we employ PLS with $m=70$ components to ensure that conditions of Theorem~\ref{thm:test} are satisfied.\footnote{We find that for these models the fitting error does not substantially decrease after $m=10$ components.} The results are presented in Figure~\ref{fig:distribution_and_qq}.

For each of the three models, the left columns in Figure~\ref{fig:distribution_and_qq} display the finite-sample distribution of $T_n$ (in blue) under $H_0$, overlaid with the simulated $\sum_{j=1}^{100} \omega_j Z_j^2$ (in orange) obtained from Theorem~\ref{thm:test}. The right columns present the corresponding QQ-plots, comparing the empirical quantiles of $T_n$ to the asymptotic quantiles of $T$. Overall, the simulated asymptotic distribution closely matches the finite-sample distribution, with only minor discrepancies observed in the extreme right tail.

\begin{figure}[htbp]
	\centering
	\begin{subfigure}[b]{0.45\textwidth}
		\includegraphics[width=\textwidth]{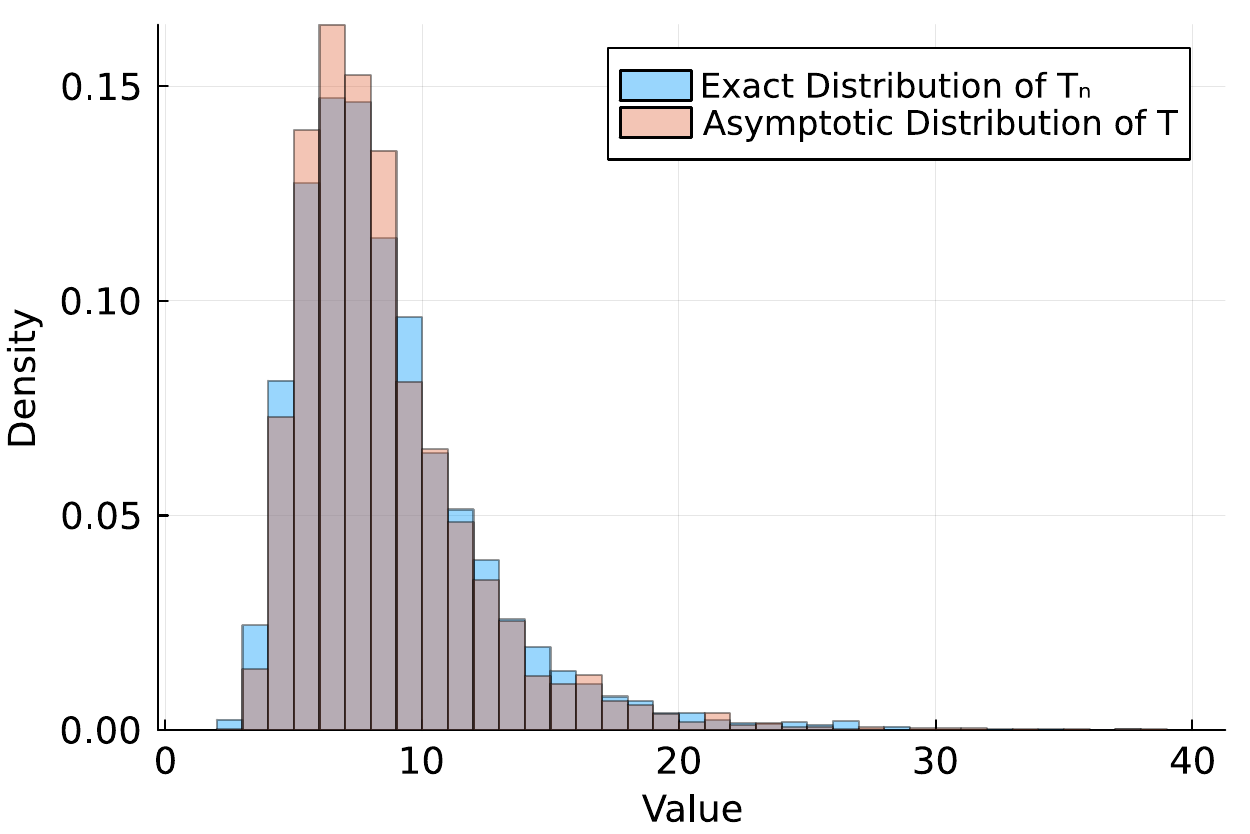}
		\caption{Model 1 }
	\end{subfigure}\hfill
	\begin{subfigure}[b]{0.45\textwidth}
		\includegraphics[width=\textwidth]{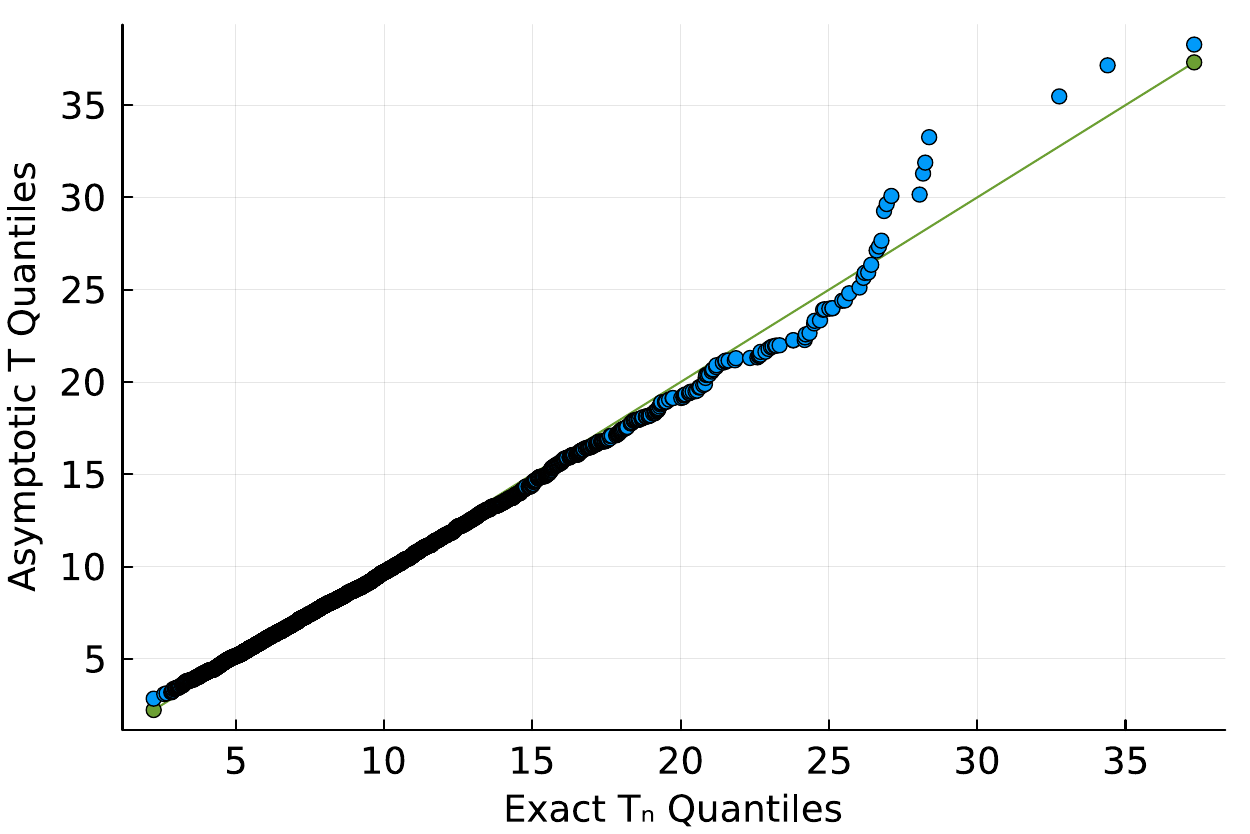}
		\caption{Model 1 QQ‐plot}
	\end{subfigure}
		\begin{subfigure}[b]{0.45\textwidth}
		\includegraphics[width=\textwidth]{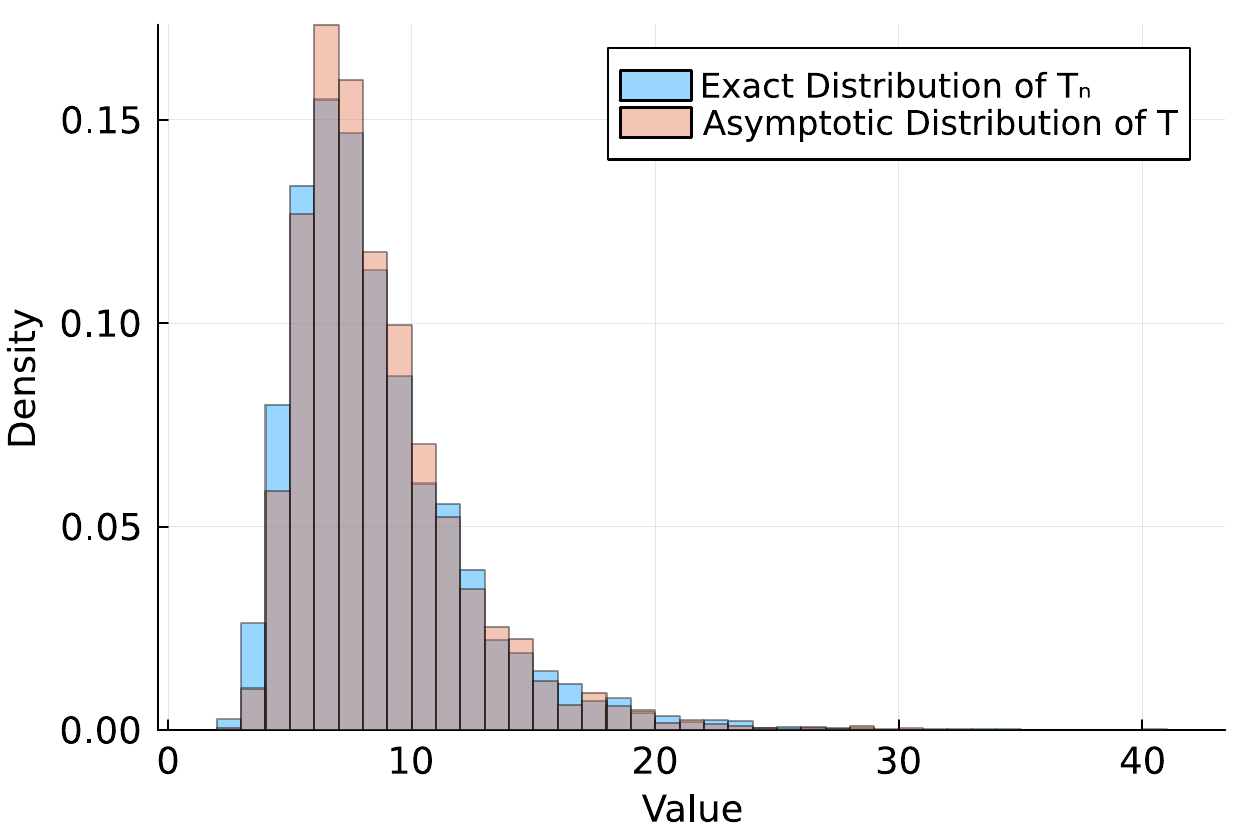}
		\caption{Model 2}
	\end{subfigure}\hfill
	\begin{subfigure}[b]{0.45\textwidth}
		\includegraphics[width=\textwidth]{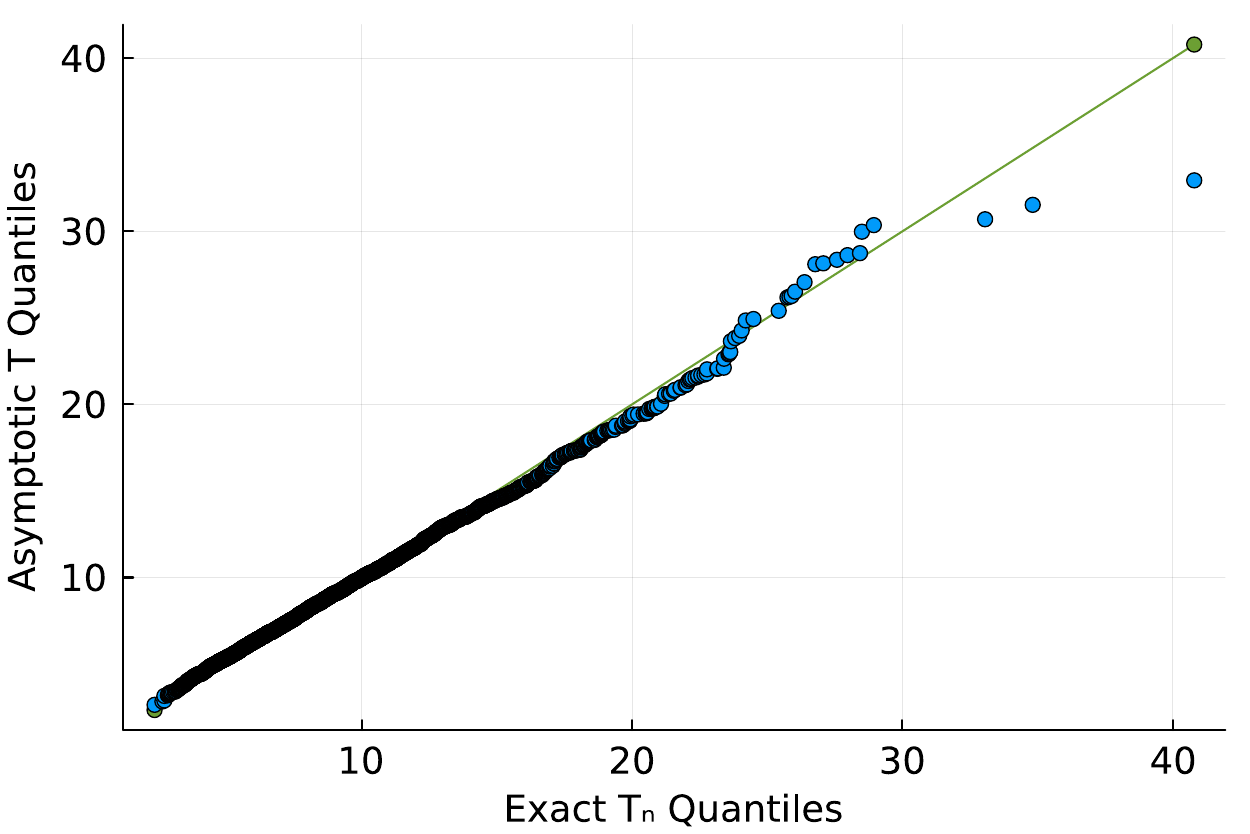}
		\caption{Model 2 QQ‐plot}
	\end{subfigure}
	\begin{subfigure}[b]{0.45\textwidth}
		\includegraphics[width=\textwidth]{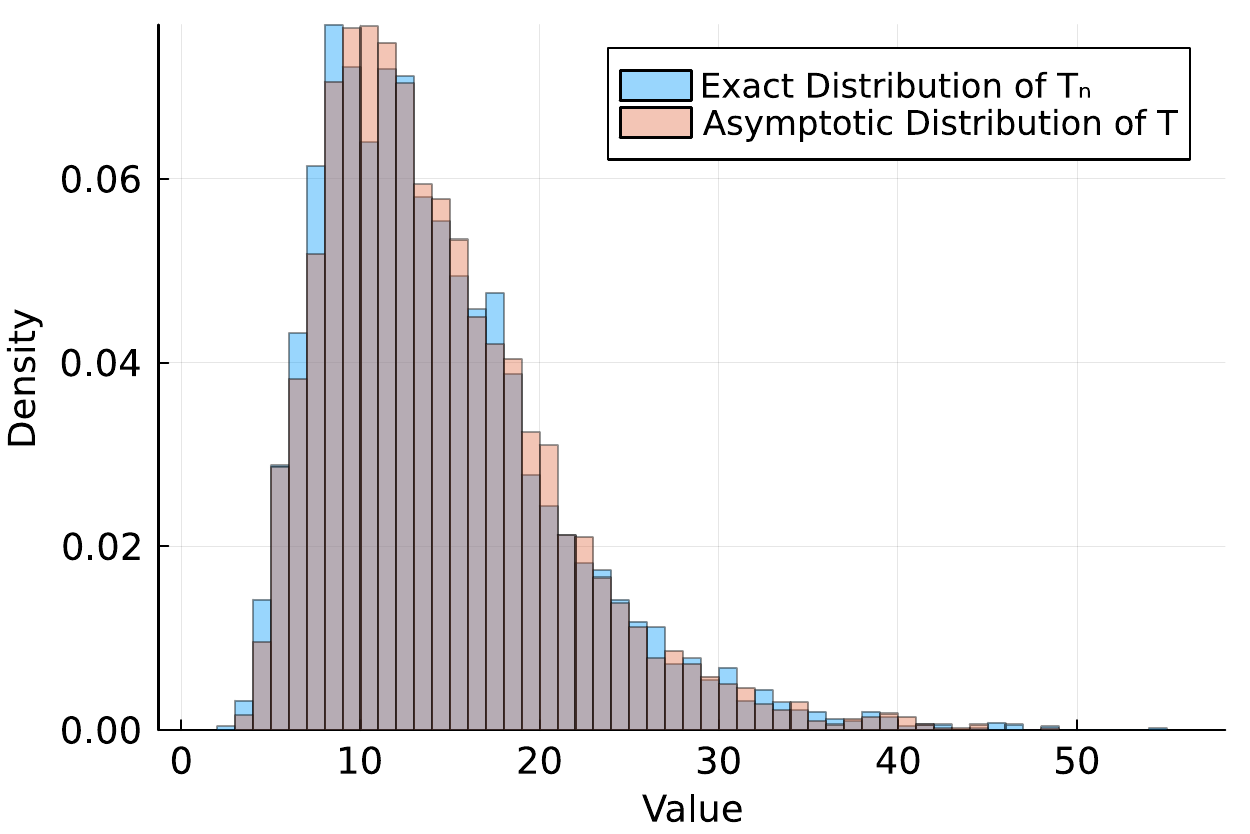}
		\caption{Model 3}
	\end{subfigure}\hfill
	\begin{subfigure}[b]{0.45\textwidth}
		\includegraphics[width=\textwidth]{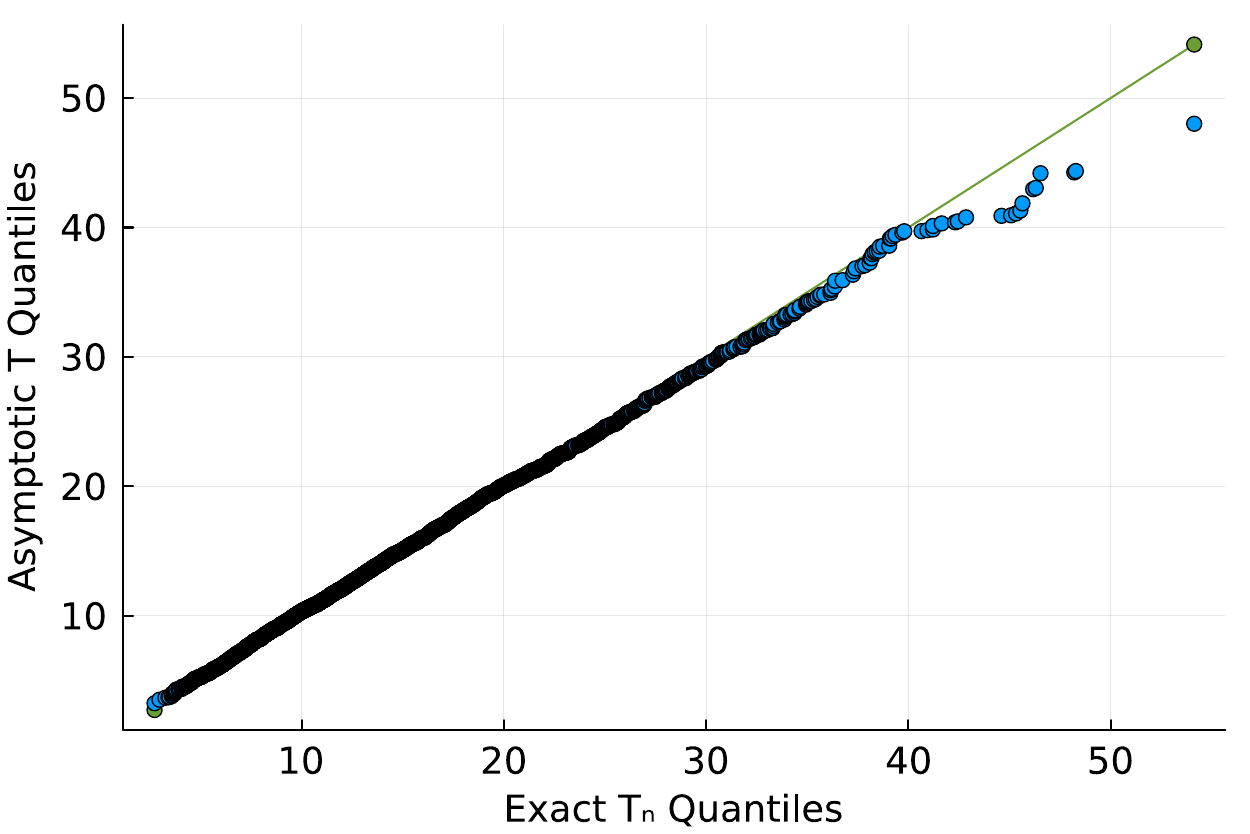}
		\caption{Model 3 QQ‐plot}
	\end{subfigure}
	\caption{Accuracy of asymptotic approximations. The left panels display the empirical finite-sample distribution of the test statistic $T_n$ under $H_0$ (blue) overlaid with its asymptotic distribution (orange) for each of the three models. The right panels present the corresponding QQ-plots, comparing empirical quantiles of $T_n$ to the theoretical asymptotic quantiles.}
	\label{fig:distribution_and_qq}
\end{figure}

To study the finite sample power of the test, we simulate the power curves corresponding to the local alternatives $\beta(s)+\delta s$ with deviations measured by the scale factor $\delta\in[-1,1]$. $\delta=0$ corresponds to $H_0$ while $|\delta|>0$ to $H_1$. We also consider doubling the sample size from $n=100$ to $n=200$. The results are displayed on Figure~\ref{fig:power_curves}, confirming that the test has more power once the null and the alternative hypotheses become sufficiently separated. The power also increases with the sample size as expected. Lastly, we provide additional simulation results in the Supplementary Material.
\begin{figure}[htbp]
	\centering
	% Model 1
	\begin{subfigure}[b]{0.5\textwidth}
		\includegraphics[width=\textwidth]{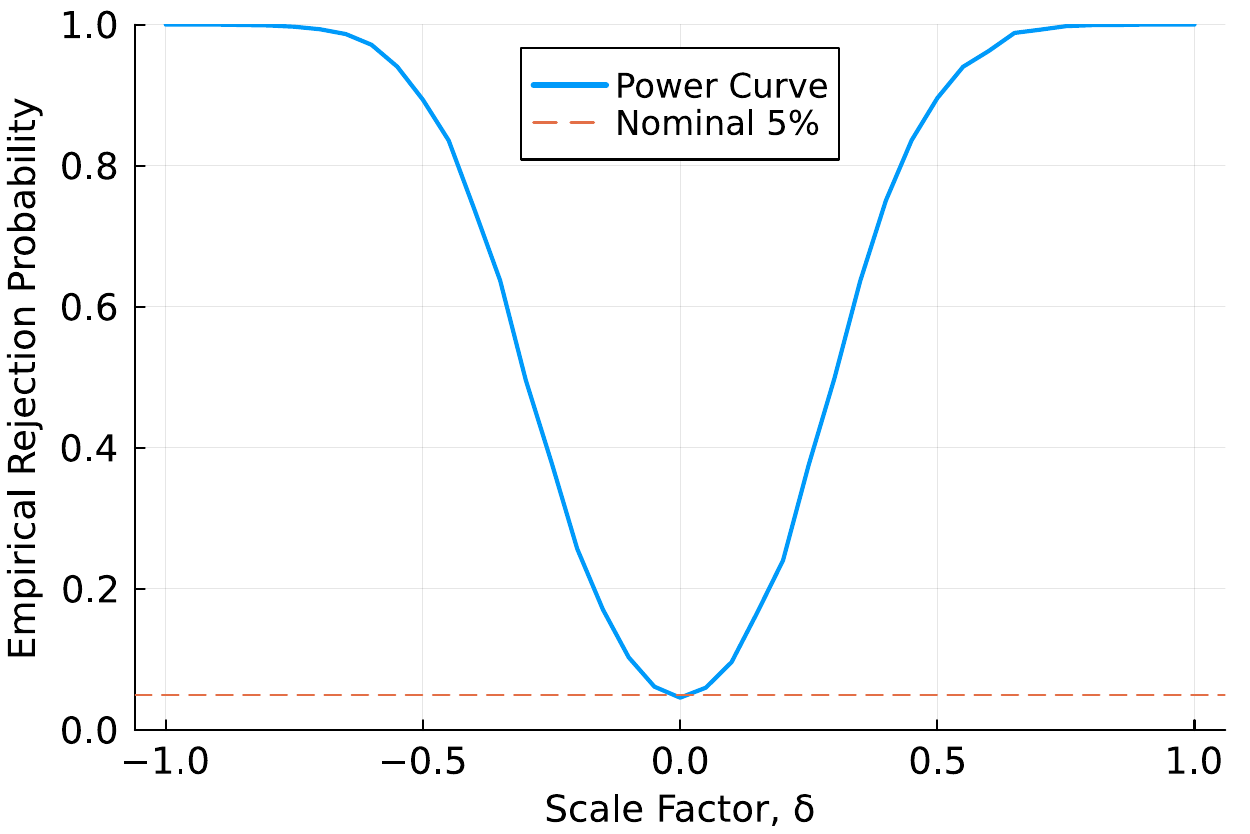}
		\caption{Model 1, sample $n=100$}
		\label{fig:power1}
	\end{subfigure}\hfill
	\begin{subfigure}[b]{0.5\textwidth}
		\includegraphics[width=\textwidth]{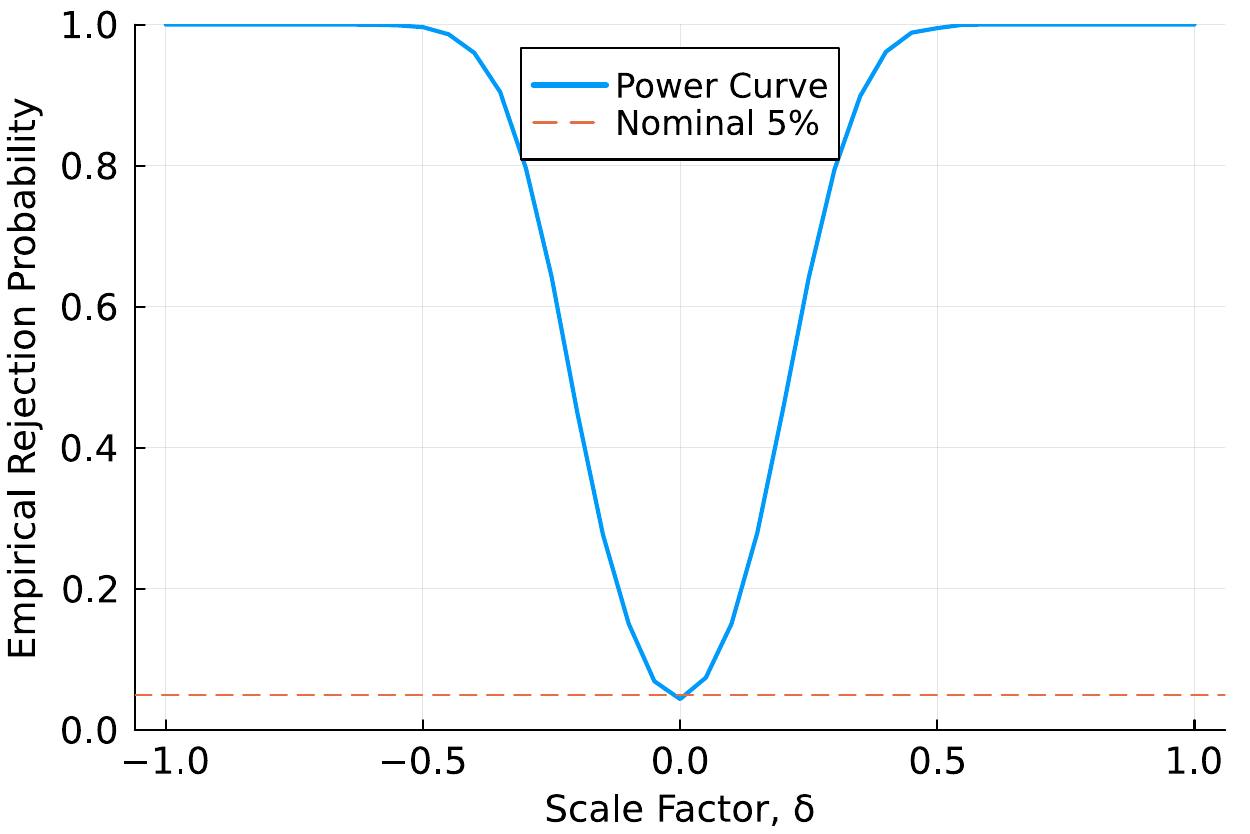}
		\caption{Model 1, sample $n=200$}
		\label{fig:power1_2}
	\end{subfigure}\hfill
	% Model 2
	\begin{subfigure}[b]{0.5\textwidth}
		\includegraphics[width=\textwidth]{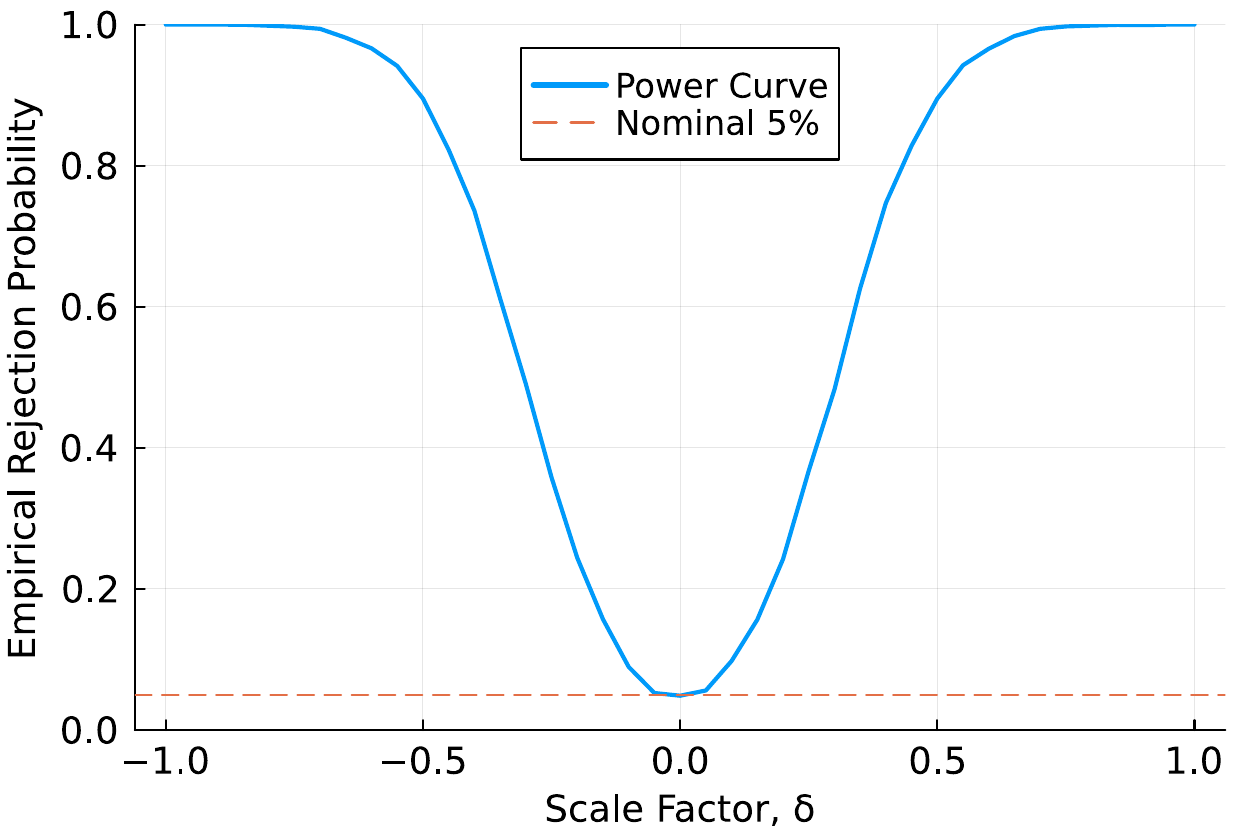}
		\caption{Model 2, sample $n=100$}
		\label{fig:power2}
	\end{subfigure}\hfill
	\begin{subfigure}[b]{0.5\textwidth}
		\includegraphics[width=\textwidth]{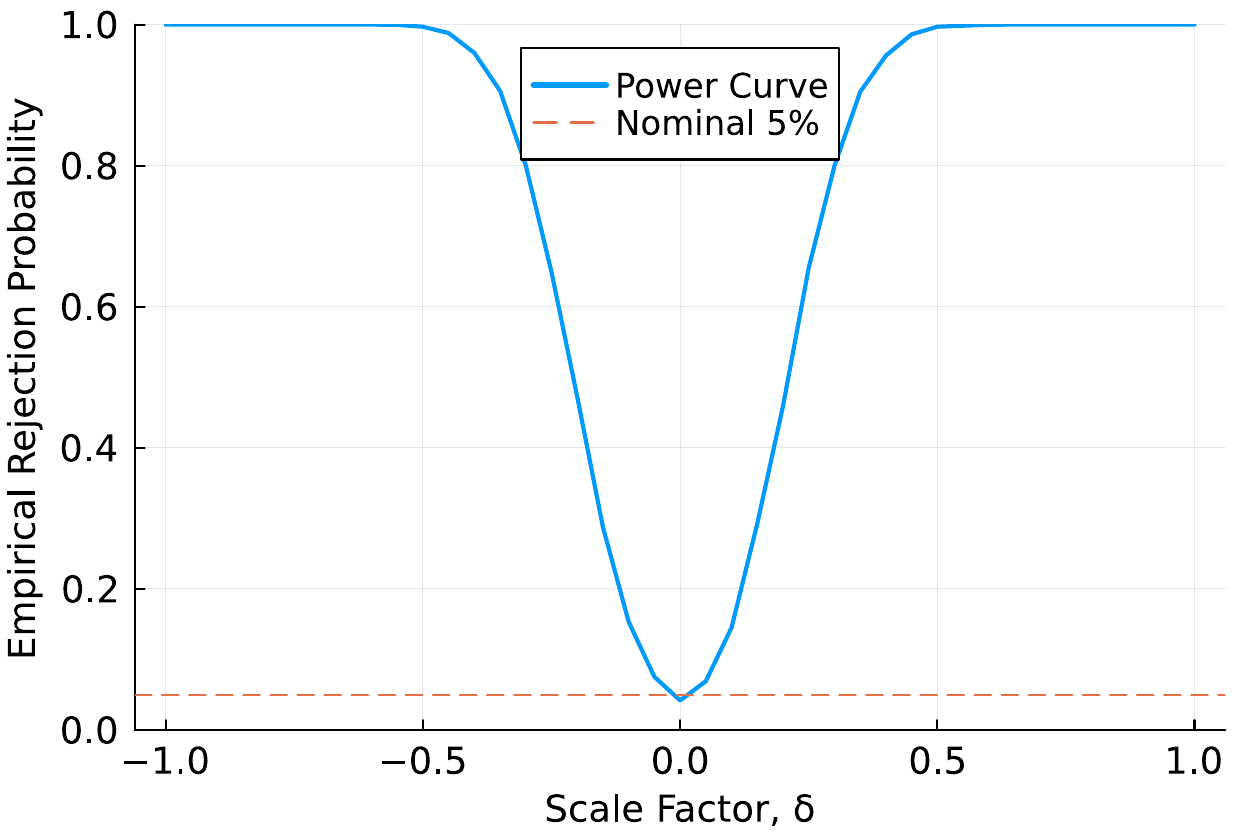}
		\caption{Model 3, sample $n=200$}
		\label{fig:power2_2}
	\end{subfigure}\hfill
	% Model 3
	\begin{subfigure}[b]{0.5\textwidth}
		\includegraphics[width=\textwidth]{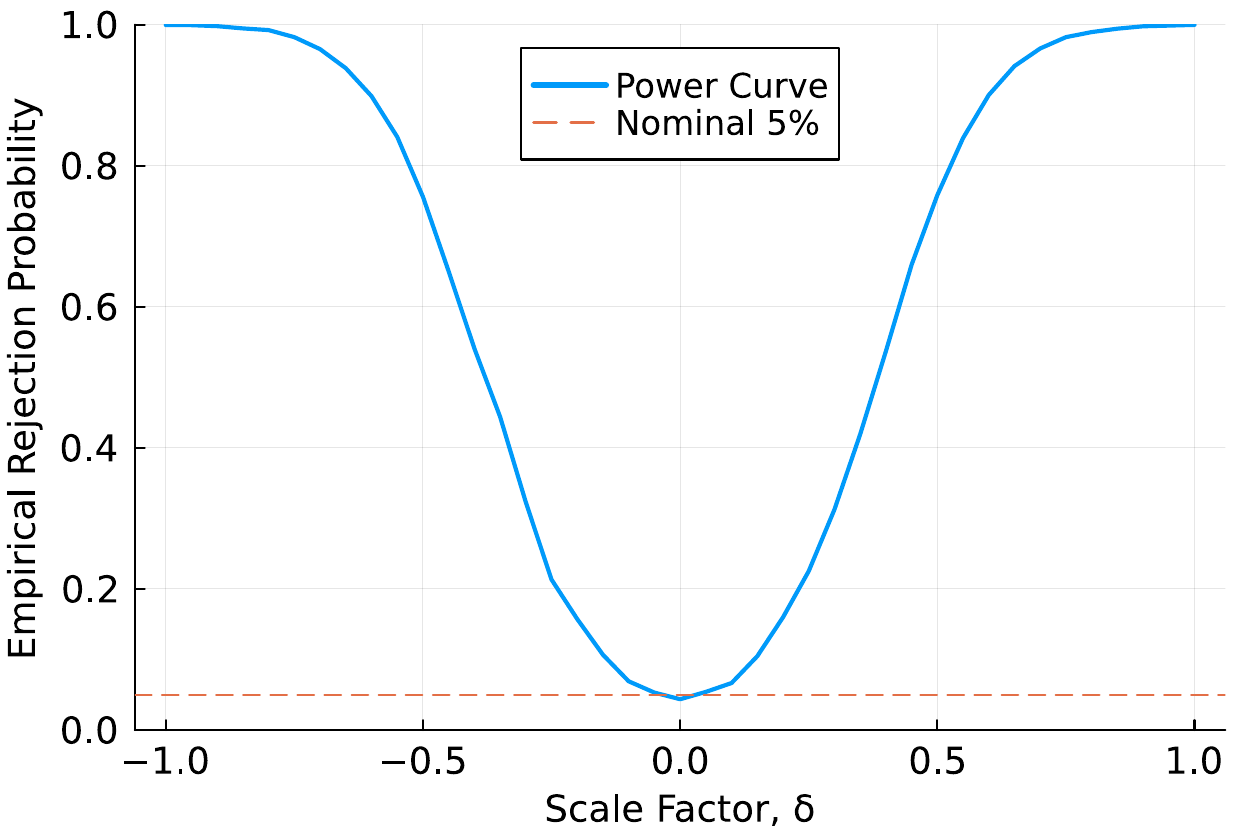}
		\caption{Model 3, sample $n=100$}
		\label{fig:power3}
	\end{subfigure}
	\begin{subfigure}[b]{0.49\textwidth}
		\includegraphics[width=\textwidth]{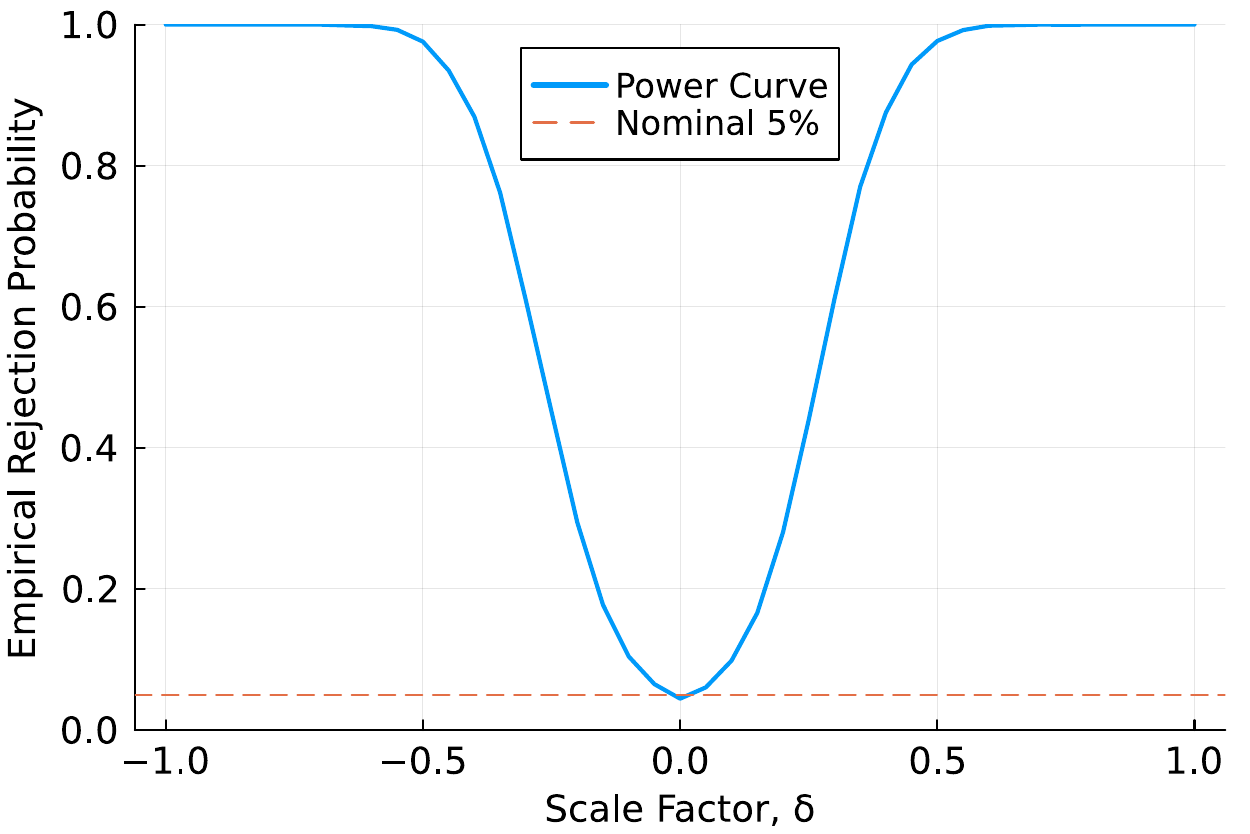}
		\caption{Model 3, sample $n=200$}
		\label{fig:power3_2}
	\end{subfigure}
	
	\caption{Simulated power curves of the test. Power curves for the three models are shown with $n=100$ (left columns) and $n=200$ (right columns). The solid blue line plots the empirical rejection probability against the scale factor~$\delta$, while the red dashed line marks the 5\% nominal significance level.}
	\label{fig:power_curves}
\end{figure}

\section{Nonlinear Temperature Effects in US Agriculture}\label{sec:application}
The global surface temperature has increased by 1.1°C above pre-industrial levels and could increase up to 3.6°C to 4.5°C by the end of the century if current  $CO_2$ emissions rise steadily according to the latest studies; see \cite{lee2023ipcc}. The global warming will likely lead to more frequent and severe heatwaves, altered precipitation patterns, and intensified droughts. Of all major sectors, agriculture is arguably the most sensitive to climate change. While constituting a modest share of developed economies, it is vital for food security. Indeed, the intensified droughts could cause food shortages which in turn may potentially exacerbate mass migration and violent conflicts. Some have argued that the current climates are already warmer than is optimal for agriculture in many parts of Asia, Africa, and Latin America; see \cite{nordhaus2013climate}.

Determining the precise functional form of the relationship between crop yields and temperature has recently attracted lots of attention; see \cite{schlenker2006nonlinear,schlenker2009nonlinear} and \cite{cui2024model}.\footnote{The influential study of \cite{schlenker2009nonlinear} has more than 4,000 Google Scholar citations at the time of writing.} We argue that the methodology used to estimate such nonlinear temperature effects can be understood as a functional linear regression, where the outcome $Y_{it}$ is the log yield of a crop of a county $i$ in a year $t$, measured in bushels per acre, and the functional regressor $(X_{it}(s))_{s\in[0,40]}$ is a temperature curve, representing the crop exposure to temperatures between 0°C to 40°C during the growing season. Following \cite{schlenker2009nonlinear}, the temperature curves are computed at discrete points $s_j\in\{0,1,\dots,40\}$ as $X_{it}(s_j) = \Phi_{it}(s_j+1) - \Phi_{it}(s_j)$, where $\Phi_{it}(s_j)$ is the length of time (measured in days) the crop was continuously exposed to temperature larger than $s_j$ for county $i$ in year $t$.

We focus on corn and soybeans which are the two major crops grown in the US. The dataset is comprised of fine-scale county-level crop yields and weather outcomes, spanning US counties east of the 100 degree meridian from 1950 to 2020.\footnote{The dataset is publicly available at the time of writing at \url{www.wolfram-schlenker.info/replicationFiles/SchlenkerRoberts2009.zip}.} We use the same set of controls as in \cite{schlenker2009nonlinear}, namely: a constant, precipitation measured in mm from March through August, precipitation$^2$, county fixed effects, and a state-specific quadratic time trend to capture technological change. The crop yields $Y$ and the temperature curve $X$ are regressed on these controls to obtain the residuals which are subsequently used for the functional data analysis.

The slope coefficient is then estimated using: 1) our functional PLS estimator; and 2) a highly parameterized least-squares estimator with a step function approximation as in \cite{schlenker2009nonlinear}. The latter fits a separate temperature effect for each 3°C bin from 0°C to 40°C, hence, it involves 13 parameters. On the other hand, our early stopping rule finds $\hat m = 4$ functional PLS components both for corn and soybeans; see Appendix Section~\ref{suppl:simulations} for more details on the implementation.

\begin{figure}[h]
	\caption{Nonlinear relationship between temperature and crop yields fitted using functional PLS (red curve) and step function approximation (black dash).}
	\begin{center}
		\begin{subfigure}{0.49\linewidth}
			\includegraphics[width = \linewidth]{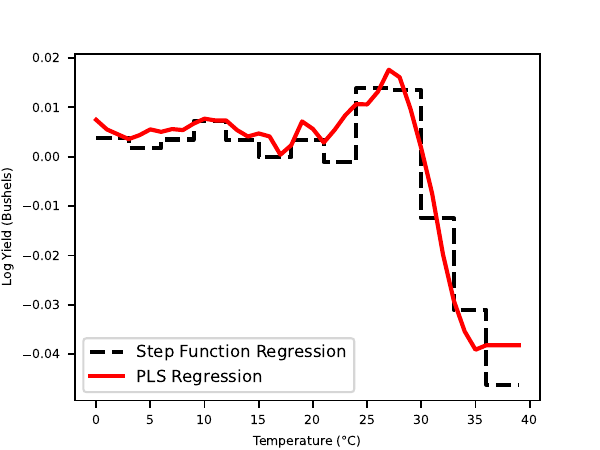} 
			\caption{Impact of Temperature on Corn Yield}
		\end{subfigure}
		\begin{subfigure}{0.49\linewidth}
			\includegraphics[width = \linewidth]{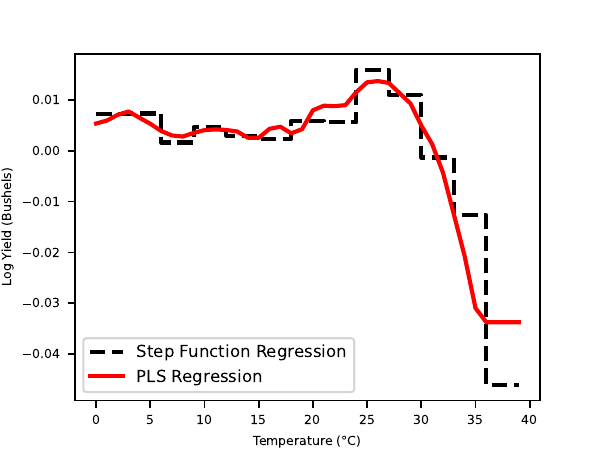}
			\caption{Impact of Temperature on Soybean Yield}
		\end{subfigure}
	\end{center}
	\label{fig:temperature_effects}
\end{figure}

Figure~\ref{fig:temperature_effects} displays the estimated functional slope coefficient $\beta$ corresponding to our functional PLS (red cure) and step function approximation (black dash) for corn and soybeans. We find that the critical temperature after which the crop yields start declining is around 29-30°C which is similar to findings reported in \cite{schlenker2009nonlinear}.

\begin{figure}[ht]
	\caption{Adaptation effects in nonlinear relationship between temperature and crop yields.}
	\begin{center}
		\begin{subfigure}{0.49 \linewidth}
			\includegraphics[width = \linewidth]{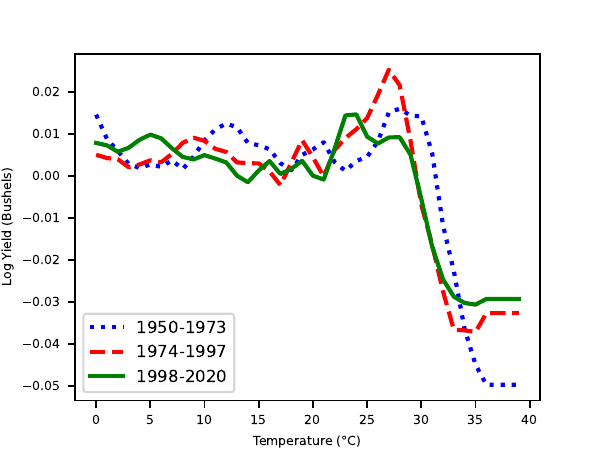} 
			\caption{Impact of Temperature on Corn Yield}
		\end{subfigure}
		\begin{subfigure}{0.49 \linewidth}
			\includegraphics[width = \linewidth]{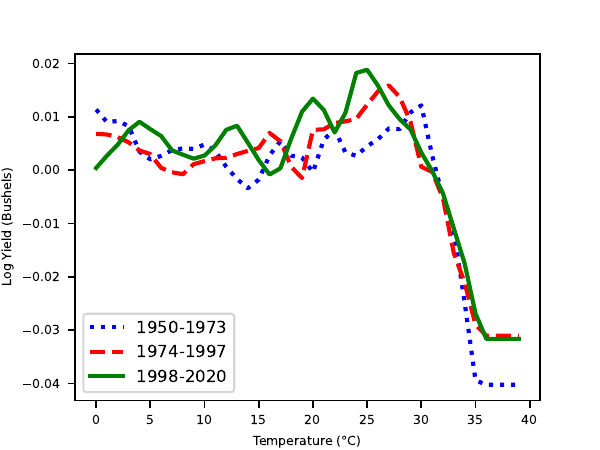}
			\caption{Impact of Temperature on Soybeans Yield}
		\end{subfigure}
	\end{center}
	\label{fig:temperature_effects_adaptation}
\end{figure}

Lastly, we look at how the nonlinear temperature effects have changed over time. Figure~\ref{fig:temperature_effects_adaptation} reports the estimated functional slope coefficient splitting the data into three subsamples: 1950-1973 (blue dot), 1974-1997 (red dash), and 1998-2020 (green curve). The results indicate that the negative temperature effects were larger during 1950-1973 compared  to the most recent 22 years, especially for the extreme temperatures. The mitigation of extreme temperature effects may come from two sources: the adaptation and the $CO_2$ fertilization. The $CO_2$ fertilization effects observed in our sample are likely to be small; see \cite{nordhaus2013climate} who argues that doubling the atmospheric concentration of $CO_2$ would increase crop yields by $10$-$15\%$ only. In contrast, the adaptation effect is likely to dominate over time. It can be attributed to the actions taken by farmers, such as adjusting the sowing and harvesting dates to maximize yields, using more resilient crops, or building efficient irrigation systems. Our results, therefore, suggest some evidence of adaptation in US agriculture which has been reported in \cite{burke2016adaptation} without properly accounting for nonlinearities. However, our confidence sets are too wide to have conclusive evidence when accounting for statistical uncertainty because there are only few observations of extreme temperatures observed in the data.

\section{Conclusions}
\label{sec:conc}

This paper proposes a new formulation of the functional PLS estimator related to the conjugate gradient method applied to an ill-posed inverse problem with a self-adjoint operator. We provide the first optimality result for functional PLS and consider a rate-adaptive early stopping rule to select the optimal number of functional components. The estimator has good estimation and prediction properties for a smaller number of principal components than PCA and the early stopping rule performs well in simulations. We find in an empirical application that the nonlinear temperature effects on crop yields have slightly decreased since 1950 which provides some evidence for adaptation of the US agriculture. However, this evidence is not conclusive due to statistical uncertainty arising from the limited number of observations of extreme temperatures in the data. Future studies need to develop more efficient methods to deal with this problem.

\bibliographystyle{plainnat}
\bibliography{bibliography}

\begin{thebibliography}{52}
\providecommand{\natexlab}[1]{#1}
\providecommand{\url}[1]{\texttt{#1}}
\expandafter\ifx\csname urlstyle\endcsname\relax
  \providecommand{\doi}[1]{doi: #1}\else
  \providecommand{\doi}{doi: \begingroup \urlstyle{rm}\Url}\fi

\bibitem[Aguilera et~al.(2010)Aguilera, Escabias, Preda, and
  Saporta]{aguilera2010basis}
A.~M. Aguilera, M.~Escabias, C.~Preda, and G.~Saporta.
\newblock Using basis expansion for estimating functional pls regression:
  Applications with chemometric data.
\newblock \emph{Chemometrics and Intelligent Laboratory Systems}, 104\penalty0
  (2):\penalty0 289--305, 2010.

\bibitem[Aleksandrov and Peller(2016)]{aleksandrov2016operator}
Alexei~B Aleksandrov and Vladimir~V Peller.
\newblock Operator lipschitz functions.
\newblock \emph{Russian Mathematical Surveys}, 71\penalty0 (4):\penalty0
  605--702, 2016.

\bibitem[Babii and Florens(2025 (forthcoming))]{babii2017completeness}
Andrii Babii and Jean-Pierre Florens.
\newblock Is completeness necessary? estimation in nonidentified linear models.
\newblock \emph{Econometric Theory}, 2025 (forthcoming).

\bibitem[Bartlett et~al.(2020)Bartlett, Long, Lugosi, and
  Tsigler]{bartlett2020benign}
Peter~L Bartlett, Philip~M Long, G{\'a}bor Lugosi, and Alexander Tsigler.
\newblock Benign overfitting in linear regression.
\newblock \emph{Proceedings of the National Academy of Sciences}, 117\penalty0
  (48):\penalty0 30063--30070, 2020.

\bibitem[Baíllo(2009)]{baillo2009}
Amparo Baíllo.
\newblock A note on functional linear regression.
\newblock \emph{Journal of Statistical Computation and Simulation}, 79\penalty0
  (5):\penalty0 657--669, 2009.
\newblock \doi{10.1080/00949650701836765}.

\bibitem[Blanchard and Kr{\"a}mer(2016)]{blanchard2016convergence}
Gilles Blanchard and Nicole Kr{\"a}mer.
\newblock Convergence rates of kernel conjugate gradient for random design
  regression.
\newblock \emph{Analysis and Applications}, 14\penalty0 (06):\penalty0
  763--794, 2016.

\bibitem[Blaz{\`e}re et~al.(2014{\natexlab{a}})Blaz{\`e}re, Gamboa, and
  Loubes]{blazere2014pls}
M{\'e}lanie Blaz{\`e}re, Fabrice Gamboa, and Jean-Michel Loubes.
\newblock {PLS}: a new statistical insight through the prism of orthogonal
  polynomials.
\newblock \emph{arXiv preprint arXiv:1405.5900}, 2014{\natexlab{a}}.

\bibitem[Blaz{\`e}re et~al.(2014{\natexlab{b}})Blaz{\`e}re, Gamboa, and
  Loubes]{blazere2014unified}
M{\'e}lanie Blaz{\`e}re, Fabrice Gamboa, and Jean-Michel Loubes.
\newblock A unified framework to study the properties of the {PLS} vector of
  regression coefficients.
\newblock In \emph{International Conference on Partial Least Squares and
  Related Methods}, pages 227--237. Springer, 2014{\natexlab{b}}.

\bibitem[Bosq(2000)]{bosq2000linear}
Denis Bosq.
\newblock \emph{Linear processes in function spaces: theory and applications},
  volume 149.
\newblock Springer Science \& Business Media, 2000.

\bibitem[Burke and Emerick(2016)]{burke2016adaptation}
Marshall Burke and Kyle Emerick.
\newblock Adaptation to climate change: Evidence from us agriculture.
\newblock \emph{American Economic Journal: Economic Policy}, 8\penalty0
  (3):\penalty0 106--140, 2016.

\bibitem[Cai and Hall(2006)]{cai2006prediction}
T~Tony Cai and Peter Hall.
\newblock Prediction in functional linear regression.
\newblock \emph{The Annals of Statistics}, 34\penalty0 (5):\penalty0
  2159--2179, 2006.

\bibitem[Cai and Yuan(2012)]{cai2012minimax}
T~Tony Cai and Ming Yuan.
\newblock Minimax and adaptive prediction for functional linear regression.
\newblock \emph{Journal of the American Statistical Association}, 107\penalty0
  (499):\penalty0 1201--1216, 2012.

\bibitem[Cardot et~al.(1999)Cardot, Ferraty, and Sarda]{cardot1999functional}
Herv{\'e} Cardot, Fr{\'e}d{\'e}ric Ferraty, and Pascal Sarda.
\newblock Functional linear model.
\newblock \emph{Statistics \& Probability Letters}, 45\penalty0 (1):\penalty0
  11--22, 1999.

\bibitem[Cardot et~al.(2003)Cardot, Ferraty, and Sarda]{cardot2003spline}
Herv{\'e} Cardot, Fr{\'e}d{\'e}ric Ferraty, and Pascal Sarda.
\newblock Spline estimators for the functional linear model.
\newblock \emph{Statistica Sinica}, 13\penalty0 (3):\penalty0 571--591, 2003.

\bibitem[Carrasco and Rossi(2016)]{carrasco2016sample}
Marine Carrasco and Barbara Rossi.
\newblock In-sample inference and forecasting in misspecified factor models.
\newblock \emph{Journal of Business \& Economic Statistics}, 34\penalty0
  (3):\penalty0 313--338, 2016.

\bibitem[Carrasco et~al.(2007)Carrasco, Florens, and
  Renault]{carrasco2007linear}
Marine Carrasco, Jean-Pierre Florens, and Eric Renault.
\newblock Linear inverse problems in structural econometrics estimation based
  on spectral decomposition and regularization.
\newblock \emph{Handbook of econometrics}, 6:\penalty0 5633--5751, 2007.

\bibitem[Chernozhukov et~al.(2024)Chernozhukov, Hansen, Kallus, Spindler, and
  Syrgkanis]{chernozhukov2024applied}
Victor Chernozhukov, Christian Hansen, Nathan Kallus, Martin Spindler, and
  Vasilis Syrgkanis.
\newblock Applied causal inference powered by {ML} and {AI}.
\newblock \emph{arXiv preprint arXiv:2403.02467}, 2024.

\bibitem[Crambes et~al.(2009)Crambes, Kneip, and Sarda]{crambes2009smoothing}
Christophe Crambes, Alois Kneip, and Pascal Sarda.
\newblock Smoothing splines estimators for functional linear regression.
\newblock \emph{The Annals of Statistics}, 37\penalty0 (1):\penalty0 35--72,
  2009.

\bibitem[Cui et~al.(2024)Cui, Gafarov, Ghanem, and Kuffner]{cui2024model}
Xiaomeng Cui, Bulat Gafarov, Dalia Ghanem, and Todd Kuffner.
\newblock On model selection criteria for climate change impact studies.
\newblock \emph{Journal of Econometrics}, 239\penalty0 (1):\penalty0 105511,
  2024.

\bibitem[Delaigle and Hall(2012)]{delaigle2012methodology}
Aurore Delaigle and Peter Hall.
\newblock Methodology and theory for partial least squares applied to
  functional data.
\newblock \emph{The Annals of Statistics}, 40\penalty0 (1):\penalty0 322--352,
  2012.

\bibitem[Engl et~al.(1996)Engl, Hanke, and Neubauer]{engl1996regularization}
Heinz~Werner Engl, Martin Hanke, and Andreas Neubauer.
\newblock \emph{Regularization of inverse problems}, volume 375.
\newblock Springer Science \& Business Media, 1996.

\bibitem[Febrero-Bande et~al.(2017)Febrero-Bande, Galeano, and
  González-Manteiga]{febrero2017overview}
M.~Febrero-Bande, P.~Galeano, and W.~González-Manteiga.
\newblock Functional principal component regression and functional partial
  least-squares regression: An overview and a comparative study.
\newblock \emph{International Statistical Review}, 85\penalty0 (1):\penalty0
  61--83, 2017.

\bibitem[Florens and Van~Bellegem(2015)]{florens2015instrumental}
Jean-Pierre Florens and S{\'e}bastien Van~Bellegem.
\newblock Instrumental variable estimation in functional linear models.
\newblock \emph{Journal of Econometrics}, 186\penalty0 (2):\penalty0 465--476,
  2015.

\bibitem[Frank and Friedman(1993)]{frank1993}
I.~E. Frank and J.~H. Friedman.
\newblock A statistical view of some chemometrics regression tools.
\newblock \emph{Journal of the American Statistical Association}, 80\penalty0
  (392):\penalty0 233--250, 1993.

\bibitem[Friedman et~al.(2009)Friedman, Tibshirani, and
  Hastie]{friedman2009elements}
Jerome~H. Friedman, Robert Tibshirani, and Trevor Hastie.
\newblock \emph{The Elements of Statistical Learning: Data Mining, Inference,
  and Prediction}.
\newblock Springer, New York, 2nd edition, 2009.
\newblock ISBN 978-0-387-84857-0.

\bibitem[Gupta et~al.(2023)Gupta, Sivananthan, and
  Sriperumbudur]{gupta2023convergence}
Naveen Gupta, S~Sivananthan, and Bharath~K Sriperumbudur.
\newblock Convergence analysis of kernel conjugate gradient for functional
  linear regression.
\newblock \emph{arXiv preprint arXiv:2310.02607}, 2023.

\bibitem[Hall and Horowitz(2007)]{hall2007methodology}
Peter Hall and Joel~L Horowitz.
\newblock Methodology and convergence rates for functional linear regression.
\newblock \emph{The Annals of Statistics}, 35\penalty0 (1):\penalty0 70--91,
  2007.

\bibitem[Hanke(1995)]{hanke1995conjugate}
Martin Hanke.
\newblock \emph{Conjugate gradient type methods for linear ill-posed problems}.
\newblock Pitman Research Notes in Mathematics Series, 1995.

\bibitem[Helland(1988)]{helland1988structure}
Inge~S Helland.
\newblock On the structure of partial least squares regression.
\newblock \emph{Communications in statistics-Simulation and Computation},
  17\penalty0 (2):\penalty0 581--607, 1988.

\bibitem[Hestenes and Stiefel(1952)]{hestenes1952methods}
Magnus~R Hestenes and Eduard Stiefel.
\newblock Methods of conjugate gradients for solving linear systems.
\newblock \emph{Journal of Research of the National Bureau of Standards},
  49\penalty0 (6):\penalty0 409--436, 1952.

\bibitem[IPCC(2021)]{lee2023ipcc}
IPCC.
\newblock Summary for policymakers. in: Climate change 2021: The physical
  science basis. contribution of working group i to the sixth assessment report
  of the intergovernmental panel on climate change (ipcc).
\newblock 2021.

\bibitem[Jolliffe(1982)]{jolliffe1982note}
Ian~T Jolliffe.
\newblock A note on the use of principal components in regression.
\newblock \emph{Applied Statistics}, \penalty0 (3):\penalty0 300--303, 1982.

\bibitem[Jong(1993)]{jong1993pls}
Sijmen~De Jong.
\newblock {PLS} fits closer than {PCR}.
\newblock \emph{Journal of Chemometrics}, 7\penalty0 (6):\penalty0 551--557,
  1993.

\bibitem[Kelly and Pruitt(2015)]{kelly2015three}
Bryan Kelly and Seth Pruitt.
\newblock The three-pass regression filter: A new approach to forecasting using
  many predictors.
\newblock \emph{Journal of Econometrics}, 186\penalty0 (2):\penalty0 294--316,
  2015.

\bibitem[Kr{\"a}mer et~al.(2008)Kr{\"a}mer, Boulesteix, and Tutz]{Kraemer2008}
Nicole Kr{\"a}mer, Anne-Laure Boulesteix, and Gerhard Tutz.
\newblock Penalized partial least squares with applications to b-spline
  transformations and functional data.
\newblock \emph{Chemometrics and Intelligent Laboratory Systems}, 94\penalty0
  (1):\penalty0 60--69, 2008.
\newblock \doi{10.1016/j.chemolab.2008.06.009}.

\bibitem[Kress(1999)]{kress1999linear}
Rainer Kress.
\newblock \emph{Linear Integral Equations}, volume~82.
\newblock Springer Science \& Business Media, 1999.

\bibitem[Li and Hsing(2007)]{li2007rates}
Yehua Li and Tailen Hsing.
\newblock On rates of convergence in functional linear regression.
\newblock \emph{Journal of Multivariate Analysis}, 98\penalty0 (9):\penalty0
  1782--1804, 2007.

\bibitem[Lin and Cevher(2021)]{lin2021kernel}
Junhong Lin and Volkan Cevher.
\newblock Kernel conjugate gradient methods with random projections.
\newblock \emph{Applied and Computational Harmonic Analysis}, 55:\penalty0
  223--269, 2021.

\bibitem[Nemirovski(1986)]{nemirovski1986regularizing}
Arkadi~S Nemirovski.
\newblock On regularizing properties of the conjugate gradient method for
  ill-posed problems (in russian).
\newblock \emph{Zhurnal Vychislitel'noi Matematiki i Matematicheskoi Fiziki},
  26\penalty0 (3):\penalty0 332--347, 1986.

\bibitem[Nocedal and Wright(1999)]{nocedal1999numerical}
Jorge Nocedal and Stephen~J Wright.
\newblock \emph{Numerical optimization}.
\newblock Springer, 1999.

\bibitem[Nordhaus(2013)]{nordhaus2013climate}
William Nordhaus.
\newblock \emph{The climate casino: Risk, uncertainty, and economics for a
  warming world}.
\newblock Yale University Press, 2013.

\bibitem[Phatak and de~Hoog(2002)]{phatak2002exploiting}
Aloke Phatak and Frank de~Hoog.
\newblock Exploiting the connection between {PLS}, {L}anczos methods and
  conjugate gradients: alternative proofs of some properties of {PLS}.
\newblock \emph{Journal of Chemometrics}, 16\penalty0 (7):\penalty0 361--367,
  2002.

\bibitem[Preda and Saporta(2005)]{preda2005pls}
C.~Preda and G.~Saporta.
\newblock Pls regression on a stochastic process.
\newblock \emph{Computational Statistics and Data Analysis}, 48\penalty0
  (1):\penalty0 149--158, 2005.

\bibitem[Ramsay and Silverman(2002)]{ramsay2002applied}
James~O Ramsay and Bernard~W Silverman.
\newblock \emph{Applied functional data analysis: methods and case studies}.
\newblock Springer, 2002.

\bibitem[Reiss and Ogden(2007)]{reiss2007functional}
Philip~T Reiss and R~Todd Ogden.
\newblock Functional principal component regression and functional partial
  least squares.
\newblock \emph{Journal of the American Statistical Association}, 102\penalty0
  (479):\penalty0 984--996, 2007.

\bibitem[Saricam et~al.(2022)Saricam, Beyaztas, Asikgil, and
  Shang]{saricam2022partial}
Semanur Saricam, Ufuk Beyaztas, Baris Asikgil, and Han~Lin Shang.
\newblock On partial least-squares estimation in scalar-on-function regression
  models.
\newblock \emph{Journal of Chemometrics}, 36\penalty0 (12):\penalty0 e3452,
  2022.

\bibitem[Schlenker and Roberts(2006)]{schlenker2006nonlinear}
Wolfram Schlenker and Michael~J Roberts.
\newblock Nonlinear effects of weather on corn yields.
\newblock \emph{Review of Agricultural Economics}, 28\penalty0 (3):\penalty0
  391--398, 2006.

\bibitem[Schlenker and Roberts(2009)]{schlenker2009nonlinear}
Wolfram Schlenker and Michael~J Roberts.
\newblock Nonlinear temperature effects indicate severe damages to us crop
  yields under climate change.
\newblock \emph{Proceedings of the National Academy of Sciences}, 106\penalty0
  (37):\penalty0 15594--15598, 2009.

\bibitem[Tsybakov(2009)]{tsybakov2008introduction}
Alexandre~B Tsybakov.
\newblock \emph{Introduction to nonparametric estimation}.
\newblock Springer Science \& Business Media, 2009.

\bibitem[Van Der~Vaart et~al.(1996)Van Der~Vaart, Wellner, van~der Vaart, and
  Wellner]{van1996weak}
Aad~W Van Der~Vaart, Jon~A Wellner, Aad~W van~der Vaart, and Jon~A Wellner.
\newblock \emph{Weak convergence}.
\newblock Springer, 1996.

\bibitem[Wold et~al.(1984)Wold, Ruhe, Wold, and Dunn]{wold1984collinearity}
Svante Wold, Arnold Ruhe, Herman Wold, and WJ~Dunn, III.
\newblock The collinearity problem in linear regression. {T}he partial least
  squares ({PLS}) approach to generalized inverses.
\newblock \emph{SIAM Journal on Scientific and Statistical Computing},
  5\penalty0 (3):\penalty0 735--743, 1984.

\bibitem[Yuan and Cai(2010)]{yuan2010reproducing}
Ming Yuan and T~Tony Cai.
\newblock A reproducing kernel hilbert space approach to functional linear
  regression.
\newblock \emph{The Annals of Statistics}, 38\penalty0 (6):\penalty0
  3412--3444, 2010.

\end{thebibliography}

\newpage
\spacingset{1.9} % DON'T change the spacing!

\begin{center}
	{\LARGE\bf SUPPLEMENTARY MATERIAL \\
		\vspace{20pt} 
		Functional Partial Least-Squares: Adaptive Estimation and Inference}
\end{center}

\setcounter{section}{0}
\renewcommand{\theequation}{S.\arabic{equation}}
\renewcommand\thetheorem{S.\arabic{theorem}}
\renewcommand\thelemma{S.\arabic{lemma}}
\renewcommand\theproposition{S.\arabic{proposition}}
\renewcommand\thetable{S.\arabic{table}}
\renewcommand\thefigure{S.\arabic{figure}}
\renewcommand\thesection{S.\arabic{section}}
\renewcommand\thesubsection{S.\arabic{section}.\arabic{subsection}}
\renewcommand\thepage{Supplementary Material - \arabic{page}}

\section{Notation and Preliminary Results}\label{appn:suppl}
In this section, we describe the notation and collect several propositions and lemmas. 

\paragraph{Notation:} For two sequences $(a_n)_{n\geq 1}$ and $(b_n)_{n\geq 1}$, we will use $a_n\lesssim b_n$ if there exists a constant $c>0$ such that $a_n\leq cb_n$ for all $n\geq 1$. We will also use $a_n\sim b_n$ if $a_n\lesssim b_n$ and $b_n\lesssim a_n$.  For two real numbers, we use $a\wedge b = \min\{a,b\}$ and $a\vee b = \max\{a,b\}$. 

The following proposition states that the PLS estimator $\hat\beta_m$ is unique for every $m\leq n_*$ and that the tuning parameter selected in Assumption~\ref{as:stopping} does not exceed the number of unique non-zero eigenvalues, $n_*$; see also \cite{blanchard2016convergence} for a kernel regression model setting.
\begin{proposition}\label{prop:uniquness}
	The solution in equation~(\ref{eq:pls_problem}) is unique for every $m\leq n_*$. Moreover, $\hat m\leq n_*$.
\end{proposition}
\begin{proof}[\rm \textbf{Proof of Proposition~\ref{prop:uniquness}}]
	Let $\mathcal{P}_m$ be the space of real polynomials of degree at most $m$ and let $\mathcal{P}_m^0$ be its subspace of polynomials with constant equal to one. The PLS problem in equation (\ref{eq:pls_problem}) amounts to fitting a polynomial of degree $m-1$ solving
	\begin{equation*}
		\hat P_m \in\argmin_{\phi\in\mathcal{P}_{m-1}}\left\|\left[I - \hat K\phi(\hat K)\right]\hat r\right\|^2
	\end{equation*}
	or equivalently a residual polynomial $\hat Q_m(\lambda) = 1 - \lambda \hat P_m(\lambda)$ solving
	\begin{equation}\label{eq:residual_problem}
		\hat Q_m \in \argmin_{\phi\in\mathcal{P}_{m}^0}\left\|\phi(\hat K)\hat r\right\|^2.
	\end{equation}
	By Parseval's identity, for every $\phi:[0,\hat\lambda_1]\to\mathbb{R}$, the objective function can be written as
	\begin{equation}\label{eq:inner_product}
		\left\| \phi(\hat K)\hat r \right\|^2 = \sum_{j=1}^{n_*}\phi(\hat\lambda_j)^2\langle \hat r,\hat v_j\rangle^2 = [\phi,\phi]_0,
	\end{equation}
	where $[.,.]_0$ is defined in equation~(\ref{eq:inner_prod_k}).	Therefore $\hat Q_m$ minimizes $\phi\mapsto [\phi,\phi]_0$ on $\mathcal{P}_m^0$. It is easy to see that $[.,.]_0$ is an inner product for every $m\leq n_*-1$. Therefore, $\hat Q_m$ is the unique projection of zero on a closed subspace $\mathcal{P}_m^0\subset \mathcal{P}_m$ with respect to $[.,.]_0$. For $m=n_*$, $[.,.]_0$ is not an inner product because we can take an $n_*$-degree polynomial $\phi\ne 0$ with roots equal to the distinct $n_*$ eigenvalues of $\hat K$, so that $[\phi,\phi]_0=0$. However, such a polynomial is unique. Therefore, $\hat P_m$ and $\hat \beta_m$ are unique for every $m\leq n_*$. This also shows that the PLS objective function is minimized to zero for $m\geq n_*$, so that the tuning parameter in Assumption~\ref{as:stopping} satisfies $\hat m\leq n_*$.
\end{proof}

We will need the following tail inequality in Hilbert spaces.
\begin{lemma}\label{lemma:chebyshev_inequality}
	Let $(\xi_i)_{i=1}^n$ be i.i.d.\ random variables in a Hilbert space $(\mathbb{H},\langle.,.\rangle)$ with the induced norm $\|.\|$. Suppose that $\E\xi_i=0$ and $\E\|\xi_i\|^2<\infty$. Then for every $\gamma\in(0,1)$
	\begin{equation*}
		\Pr\left(\left\|\frac{1}{n}\sum_{i=1}^n\xi_i \right\|\leq \sqrt{\frac{\E\|\xi_i\|^2}{\gamma n}}\right) \geq 1 - \gamma.
	\end{equation*}
\end{lemma}
\begin{proof}[\rm \textbf{Proof of Lemma~\ref{lemma:chebyshev_inequality}}]
	By Markov's inequality, $\forall u>0$
	\begin{equation*}
		\begin{aligned}
			\Pr\left(\left\|\frac{1}{n}\sum_{i=1}^n\xi_i \right\| > u\right) & \leq u^{-2}\E\left\|\frac{1}{n}\sum_{i=1}^n\xi_i \right\|^2 \\
			& = \frac{1}{u^2n^2}\sum_{i=1}^n\sum_{j=1}^n\E\left\langle \xi_i, \xi_j \right\rangle \\
			& = \frac{1}{u^2n}\E\left\|\xi_i\right\|^2,
		\end{aligned}
	\end{equation*}
	where the last two lines follow under the i.i.d. hypothesis. Setting $\gamma = \E\|\xi_i\|^2/(nu^2)$ and solving for $u$ gives the result.
\end{proof}

Lemma~\ref{lemma:chebyshev_inequality} allows us to control the tail probabilities for the PLS residual as well as the covariance operator errors on an event with probability at least $1-\gamma$. 
\begin{lemma}\label{lemma:probability_bounds}
	Suppose that Assumption~\ref{as:data} is satisfied. Then for every $\gamma\in(0,1)$
	\begin{equation*}
		\left\|\hat r - \hat K\beta\right\| \leq \sigma\sqrt{\frac{2\E\|X\|^2}{\gamma n}} \qquad \text{and}\qquad \left\|\hat K - K\right\|_{\rm HS} \leq \sqrt{\frac{2\E\|X\|^4}{\gamma n}}
	\end{equation*}
	with probability at least $1-\gamma$, where $\|.\|_{\rm HS}$ is the Hilbert-Schmidt norm.
\end{lemma}
\begin{proof}[\rm \textbf{Proof of Lemma~\ref{lemma:probability_bounds}}]
	We will apply Lemma~\ref{lemma:chebyshev_inequality}. First, we note that
	\begin{equation*}
		\left\|\hat r - \hat K\beta\right\| = \left\|\frac{1}{n}\sum_{i=1}^n\varepsilon_iX_i \right\|,
	\end{equation*}
	where $\E\|\varepsilon_iX_i\|^2 \leq \sigma^2\E\|X_i\|^2<\infty$ under Assumption~\ref{as:data}. Then by Lemma~\ref{lemma:chebyshev_inequality} with probability at least $1-\gamma/2$, we have $\|\hat r - \hat K\beta\| \leq \sigma\sqrt{2\E\|X\|^2/\gamma n}$. Second, the space of Hilbert-Schmidt operators is a Hilbert space and
	\begin{equation*}
		\left\|\hat K - K\right\|_{\rm HS} = \left\|\frac{1}{n}\sum_{i=1}^nX_i\otimes X_i - \E[X_i\otimes X_i] \right\|_{\rm HS},
	\end{equation*}
	where $\E\|X_i\otimes X_i - \E[X_i\otimes X_i]\|^2_{\rm HS}\leq \E\|X_i\otimes X_i\|_{\rm HS}^2= \E\|X_i\|^4$. Then by Lemma~\ref{lemma:chebyshev_inequality} with probability at least $1-\gamma/2$, we have $ \|\hat K - K\|_{\rm HS} \leq \sqrt{2\E\|X\|^4/\gamma n}$. The result follows by the union bound.
\end{proof}

We will also use the following two inequalities known in the perturbation theory.
\begin{lemma}\label{lemma:perturbation_inequality}
	Let $A:\mathbb{H}\to\mathbb{H}$ and $B:\mathbb{H}\to\mathbb{H}$ be two self-adjoint Hilbert-Schmidt operators. Then
	\begin{equation*}
		\|A^{\mu} - B^{\mu}\|_{\rm op} \leq c_{\mu} \|A - B\|_{\rm op}^{\mu},\qquad 0<\mu<1
	\end{equation*}
	and
	\begin{equation*}
		\|A^{\mu} - B^{\mu}\|_{\rm HS} \leq \mu\nu^{\mu-1}\|A - B\|_{\rm HS},\qquad \mu\geq 1,
	\end{equation*}
	where $\nu=\|A\|_{\rm op} \vee\|B\|_{\rm op}$.
\end{lemma}
\begin{proof}[\rm \textbf{Proof of Lemma~\ref{lemma:perturbation_inequality}}]
	See \cite{aleksandrov2016operator}, Theorem 1.7.2 for the first inequality. The second inequality follows from \cite{aleksandrov2016operator}, Theorem 3.5.1.
\end{proof}
As an immediate consequence of Lemmas~\ref{lemma:probability_bounds} and \ref{lemma:perturbation_inequality}, since $\|.\|_{\rm op}\leq \|.\|_{\rm HS}$,  for every $\gamma\in(0,1)$, we have
\begin{equation}\label{eq:power_inequality}
	\left\|\hat K^{\mu} - K^{\mu}\right\|_{\rm op} \leq (c_{\mu}\mathbf{1}_{\mu\leq 1} + \mu\nu^{\mu-1}\mathbf{1}_{\mu>1}) \left(\frac{2\E\|X\|^4}{\gamma n}\right)^{\frac{\mu\wedge 1}{2}}
\end{equation}
on an event that holds with probability at least $1-\gamma$, where $\nu = \|\hat K\|_{\rm op}\vee \|K\|_{\rm op}$.

The following Lemma presents some useful results on the residual polynomials $\hat Q_m(\lambda)=1-\lambda\hat P_m(\lambda)$; see \cite{engl1996regularization} and \cite{hanke1995conjugate}. For the completeness of the presentation, we sketch proofs for the key results and refer to the aforementioned monographs for others.
\begin{lemma}\label{lemma:polynomials}
	Let  $m\leq n_*$ be a positive integer. Then
	\begin{itemize}
		\item[(i)] $\hat{Q}_{m}$ has $m$ distinct positive real roots, denoted $\hat{\theta}_{1} > \hat{\theta}_{2} > ... >\hat{\theta}_{m} > 0$.
		\item[(ii)] $\hat Q_{m}$ is positive, decreasing, and convex on $[0,\hat{\theta}_{m}]$.
		\item[(iii)] $(\hat Q_l)_{l=0}^{n_*}$ are orthogonal with respect to $[.,.]_1$.
		\item[(iv)] $|\hat Q_m'(0)|^{-1}\leq \hat\theta_m$.
		\item[(v)] $\hat Q_m(\lambda) = \hat Q_{m-1}(\lambda)(1 - \lambda/\hat\theta_m)$.
		\item[(vi)] $\sup_{\lambda\in[0,\hat\theta_m]}\lambda^{\delta}\hat Q_m(\lambda)\sqrt{\hat\theta_m/(\hat\theta_m-\lambda)}\leq (2\delta)^{\delta}|\hat Q_m'(0)|^{-\delta}$ for every $\delta\geq 0$ with $0^0:=1$.
	\end{itemize}
\end{lemma}

\begin{proof}[\rm \textbf{Proof of Lemma~\ref{lemma:polynomials}}]
	(i) is known in the theory of orthogonal polynomials; see \cite{engl1996regularization}, Appendix A.2. For (iii) and (vi), see \cite{engl1996regularization}, Corollary 7.4, and equation (7.8). 
	
	Note that since $\hat Q_m(0)=1$, we can write
	\begin{equation*}
		\hat Q_m(\lambda) = \prod_{j=1}^m\left(1 - \frac{\lambda}{\hat\theta_j}\right).
	\end{equation*}
	This equation implies (v). Moreover, for all $\lambda\in[0,\hat\theta_m]$, we have $\hat Q_m(\lambda)\geq 0$ and by (i)
	\begin{equation*}
		\hat Q_m'(\lambda) = -\sum_{k=1}^m\frac{1}{\hat\theta_k}\prod_{j\ne k}\left(1 - \frac{\lambda}{\hat\theta_j}\right) \leq 0.
	\end{equation*}
	We also have $\hat Q_m''(\lambda)\geq 0$ for all $\lambda\in[0,\hat\theta_m]$ which proves (ii). 
	
	(iv) follows from (i) and
	\begin{equation*}
		|\hat Q_m'(0)| = \sum_{k=1}^m\frac{1}{\hat\theta_k} \geq \frac{1}{\hat\theta_m}.
	\end{equation*}
	
	Lastly, the proof of (vii) is similar to \cite{blazere2014pls}, Theorem 4.1.
\end{proof}

In what follows, for $k\in\mathbb{Z}$, consider the following measure
\begin{equation*}
	\hat \mu_k = \sum_{j=1}^{n_*}\hat\lambda_j^k\langle \hat r, \hat v_j\rangle^2\delta_{\hat\lambda_j},
\end{equation*}
where $\delta_{x}$ is the Dirac measure at $x\in\mathbb{R}$. For $\phi,\psi:[0,\hat\lambda_1]\to \mathbb{R}$, define
\begin{equation}\label{eq:inner_prod_k}
	\begin{aligned}
		[\phi,\psi]_{k} & = \int_0^\infty\phi(\lambda)\psi(\lambda)\mathrm{d}\hat\mu_{k}(\lambda) \\
		& = \sum_{j = 1}^{n_*} \phi(\hat \lambda_j)\psi(\hat \lambda_j)\hat \lambda_j^k\langle\hat r,\hat v_j\rangle^2.
	\end{aligned}
\end{equation}
Lastly, let
\begin{equation*}
	\Pi_{a}  = \sum_{j:\hat{\lambda}_{j}\leq a}\hat v_j\otimes \hat{v}_j
\end{equation*}
be the orthogonal projection operators on the eigenspaces of $\hat K$ corresponding to eigenvalues smaller or equal to $a$.

The following lemma allows us to control the residuals of the PLS estimator.
\begin{lemma}\label{lemma:residual}
	Suppose that Assumptions~\ref{as:data}, \ref{as:id}, and \ref{as:complexity} are satisfied. Then for every $1\leq m\leq n_*$ and $\gamma\in[1/n,1)$
	\begin{equation*}
		\left\|\hat r - \hat K\hat{\beta}_{m}\right\| 
		\lesssim \sigma\sqrt{\frac{2\E\|X\|^2}{\gamma n}} + |\hat Q'_{m}(0)|^{-1}\sqrt{\frac{2\E\|X\|^4}{\gamma n}} + |\hat{Q}_{m}'\left( 0\right)| ^{-(\mu+1)}\\
	\end{equation*}
	on an event with probability at least $1 - \gamma$.
\end{lemma}
\begin{proof}[\rm \textbf{Proof of Lemma~\ref{lemma:residual}}]
	Let $\varphi_m(\lambda):= \hat Q_m(\lambda)\sqrt{\hat\theta_m/(\hat\theta_m-\lambda)}$. We will first show that the following inequality holds
	\begin{equation*}
		\begin{aligned}
			\left\|\hat r - \hat K\hat \beta_m\right\| & = \left\|\hat{Q}_{m}(\hat K) \hat r\right\| \\
			& \leq  \left\|\Pi_{\hat\theta_m}\varphi_m(\hat K)\hat r\right\|,
		\end{aligned}
	\end{equation*}
	where the first line uses $\hat\beta_m = \hat P_m(\hat K)\hat r$ and $\hat Q_m(\lambda)=1-\lambda\hat P_m(\lambda)$. The inequality can be deduced from the proof of Theorem 7.9 in \cite{engl1996regularization}. For completeness, we provide an argument suitably tailored to our setting below. By Lemma~\ref{lemma:polynomials} (iii) and (v), since the polynomials $(\hat Q_m)_{m\geq0}$ are orthogonal with respect to $[.,.]_1$ (see equation (\ref{eq:inner_prod_k})) we get for $m\geq 1$
	\begin{equation*}
		\begin{aligned}
			0 & = \int_0^\infty\hat Q_m(\lambda)\hat Q_{m-1}(\lambda)\mathrm{d} \hat\mu_1(\lambda) \\
			& = \hat\theta_m\int_0^\infty\hat Q_m^2(\lambda)\frac{\lambda}{\hat\theta_m - \lambda} \mathrm{d} \hat\mu_0(\lambda) \\ 
			& = \hat\theta_m\int_0^{\hat\theta_m}\hat Q_m^2(\lambda)\frac{\lambda}{\hat\theta_m - \lambda} \mathrm{d} \hat\mu_0(\lambda) + \hat\theta_m\int_{\hat\theta_m}^\infty\hat Q_m^2(\lambda)\frac{\lambda}{\hat\theta_m - \lambda} \mathrm{d} \hat\mu_0(\lambda).
		\end{aligned}
	\end{equation*}
	Since $\hat\theta_m>0$, by Lemma~\ref{lemma:polynomials} (i), this shows that
	\begin{equation*}
		\int_0^{\hat\theta_m}\hat Q_m^2(\lambda)\frac{\lambda}{\hat\theta_m - \lambda} \mathrm{d} \hat\mu_0(\lambda) = \int_{\hat\theta_m}^\infty\hat Q_m^2(\lambda)\frac{\lambda}{\lambda - \hat\theta_m} \mathrm{d} \hat\mu_0(\lambda)
	\end{equation*}
	and so by equation~(\ref{eq:inner_product})
	\begin{equation*}
		\begin{aligned}
			\left\|\hat{Q}_{m}(\hat K) \hat r\right\|^2 & = \int_0^\infty\hat Q_m^2(\lambda)\mathrm{d} \hat\mu_0(\lambda) \\
			& = \int_0^{\hat\theta_m}\hat Q_m^2(\lambda)\mathrm{d} \hat\mu_0(\lambda) + \int_{\hat\theta_m}^\infty\hat Q_m^2(\lambda)\mathrm{d} \hat\mu_0(\lambda)  \\
			& \leq \int_0^{\hat\theta_m}\hat Q_m^2(\lambda)\mathrm{d}\hat \mu_0(\lambda) + \int_{\hat\theta_m}^\infty\hat Q_m^2(\lambda)\frac{\lambda}{\lambda - \hat\theta_m}\mathrm{d} \hat\mu_0(\lambda)\\
			& = \int_0^{\hat\theta_m}\hat Q_m^2(\lambda)\mathrm{d} \hat\mu_0(\lambda) + \int_{0}^{\hat\theta_m}\hat Q_m^2(\lambda)\frac{\lambda}{\hat\theta_m - \lambda}\mathrm{d} \hat\mu_0(\lambda)\\
			& = \int_{0}^{\hat\theta_m}\hat Q_m^2(\lambda) \frac{\hat\theta_m}{\hat\theta_m-\lambda} \mathrm{d} \hat\mu_0(\lambda) \\
			& = \left\|\Pi_{\hat\theta_m}\varphi_m(\hat K)\hat r\right\|^2
		\end{aligned}
	\end{equation*}
	where the third line follows since $1\leq \lambda/(\lambda-\hat\theta_m)$ for all $\lambda\geq \hat\theta_m$.
	
	Therefore,
	\begin{equation*}
		\begin{aligned}
			\left\|\hat r - \hat K\hat \beta_m\right\|  & \leq \left\| \Pi _{\hat{\theta }%
				_{m}}\varphi_{m}(\hat K)\hat r\right\|  \\
			&\leq \left\| \Pi _{\hat{\theta }_{m}}\varphi_{m}(\hat K)\hat K\beta \right\| + \left\| \Pi _{\hat{\theta }_{m}}\varphi_{m}(\hat K)(\hat r - \hat K\beta)\right\|.
		\end{aligned}
	\end{equation*}
	
	Under Assumption~\ref{as:complexity}, $\beta = K^{\mu} w$ with $\|w\|\leq R$, so that
	\begin{equation*}\small
		\begin{aligned}
			\left\| \Pi _{\hat{\theta }_{m}}\varphi_{m}(\hat K)
			\hat K\beta \right \Vert  & =\left \Vert \Pi _{\hat{\theta }_{m}}\varphi_{m}(\hat K)\hat KK ^{\mu}w\right\|  \\
			& = \left\| \Pi _{\hat{\theta }_{m}}\varphi_{m}(\hat K)\hat K\hat K^{\mu}w\right\| + \left\|\Pi _{\hat{\theta }_{m}}\varphi_{m}(\hat K) \hat K\left[K^{\mu} - \hat K^{\mu}\right]w\right\| \\
			& \leq  \sup_{\lambda\in[0,\hat{\theta }_{m}]}\left|\varphi_{m}(\lambda)\lambda^{1+\mu}\right| R + \sup_{\lambda\in[0,\hat\theta_m]}|\varphi_m(\lambda)\lambda|\left\|\hat K^{\mu} - K^{\mu}\right\|_{\rm op} R \\
			& \leq (2\mu+2)^{\mu+1}|\hat{Q}_{m}'\left( 0\right)|^{-(\mu+1)}R \\
			& \qquad\qquad + 2|\hat Q_m'(0)|^{-1}(c_{\mu}\mathbf{1}_{\mu\leq 1} + \mu\nu^{\mu-1}\mathbf{1}_{\mu>1})  \left(\frac{2\E\|X\|^4}{\gamma n}\right)^{\frac{\mu\wedge 1}{2}} R,
		\end{aligned}
	\end{equation*}
	where the last line follows by Lemma~\ref{lemma:polynomials} (vi) and the inequality in equation (\ref{eq:power_inequality}) on an event with probability at least $1-\gamma$. By Lemma~\ref{lemma:probability_bounds}
	\begin{equation}\label{eq:nu}
		\begin{aligned}
			\nu & = \|\hat K\|_{\rm op}\vee \|K\|_{\rm op}  \leq \|K\|_{\rm op} + \left\|\hat K - K\right\|_{\rm op} \\
			& \leq \lambda_1 + \sqrt{\frac{2\E\|X\|^4}{\gamma n}} \lesssim 1
		\end{aligned}
	\end{equation}
	on an event with probability at least $1-\gamma$.
	
	Lastly, 
	\begin{equation*}
		\begin{aligned}
			\left\| \Pi _{\hat{\theta }_{m}}\varphi_{m}(\hat K)(\hat r - \hat K\beta)\right\| & \leq \sup_{\lambda\in[0,\hat\theta_m]}|\varphi_m(\lambda)| \left\|\hat r - \hat K\beta\right\| \\
			& = \sup_{\lambda\in[0,\hat\theta_m]}\left|\hat Q_m(\lambda)\sqrt{\hat\theta_m/(\hat\theta_m-\lambda)}\right| \left\|\hat r - \hat K\beta\right\| \\
			& \leq \sigma\sqrt{\frac{2\E\|X\|^2}{\gamma n}},
		\end{aligned}
	\end{equation*}
	where the last line follows by Lemma~\ref{lemma:polynomials} (vi) with $\delta=0$ and Lemma~\ref{lemma:probability_bounds}.
\end{proof}

The next lemma provides an upper bound for the derivative of the residual polynomial of degree selected by the stopping rule in Assumption~\ref{as:stopping} with some fixed $\delta\in(0,1)$. 
\begin{lemma}\label{lemma:Q_prime}
	Suppose that Assumptions~\ref{as:data}, \ref{as:id}, \ref{as:complexity}, and \ref{as:stopping} are satisfied with $\delta\geq 1/n$. Then
	$$|\hat Q'_{\hat m}(0)|\lesssim (\delta n)^\frac{1}{2(\mu+1)}$$
	on an event with probability at least $1-\delta$.
\end{lemma}
\begin{proof}[\rm \textbf{Proof of Lemma~\ref{lemma:Q_prime}}]
	We have
	\begin{equation*}
		|\hat Q'_{\hat m}(0)|\leq |\hat Q'_{\hat m-1}(0)| +  |\hat Q'_{\hat m}(0) -  \hat Q'_{\hat m-1}(0)|,
	\end{equation*}
	where each of the two terms will be bounded separately.
	
	By the virtue of Assumption~\ref{as:stopping}
	\begin{equation*}
		\begin{aligned}
			\tau\sigma\sqrt{\frac{2\E\|X\|^2}{\delta n}} & \leq \left\|\hat r - \hat K\hat\beta_{\hat m - 1}\right\| \\
			& \leq c\left\{ \sigma\sqrt{\frac{2\E\|X\|^2}{\delta n}} + |\hat Q'_{\hat m-1}(0)|^{-1}\sqrt{\frac{2\E\|X\|^4}{\delta n}} + |\hat{Q}_{\hat m-1}'\left( 0\right)| ^{-(\mu+1)}\right\},
		\end{aligned}
	\end{equation*}
	where the second line follows by Lemma~\ref{lemma:residual} for some $c>0$.
	
	Therefore,
	\begin{equation*}
		(\tau - c)\sigma\sqrt{\frac{2\E\|X\|^2}{\delta n}} \leq c\max\left\{|\hat Q'_{\hat m-1}(0)|^{-1}\sqrt{\frac{1}{\delta n}}, |\hat{Q}_{\hat m-1}'\left( 0\right)| ^{-(\mu+1)}\right\}.
	\end{equation*}
	If the first term inside the maximum is larger, then $|\hat Q'_{\hat m-1}(0)| \lesssim 1$ while if the second term is larger, then $|\hat Q'_{\hat m - 1}(0)| \lesssim (\delta n)^\frac{1}{2(\mu+1)}$ provided that $\tau>c$. Therefore, we always have $|\hat Q'_{\hat m - 1}(0)| \lesssim (\delta n)^\frac{1}{2(\mu+1)}$.
	
	For the second term, by \cite{hanke1995conjugate}, Corollary 2.6, for every $1\leq m\leq n_*$
	\begin{equation}\label{eq:increments}
		0\leq \hat Q_{m-1}'(0) - \hat Q_{m}'(0) =  \frac{[\hat Q_{m-1},\hat Q_{m-1}]_0 - [\hat Q_{m},\hat Q_{m}]_0}{[\hat Q_{m-1}^{[2]},\hat Q_{m-1}^{[2]}]_{1}} \leq \frac{[\hat Q_{m-1},\hat Q_{m-1}]_0}{[\hat Q_{m-1}^{[2]},\hat Q_{m-1}^{[2]}]_{1}},
	\end{equation} 
	where $(\hat Q_l^{[2]})_{l\geq 0}$ are the polynomials orthogonal with respect to $[.,.]_2$ and constant equal to $1$; see equation~(\ref{eq:inner_prod_k}).
	
	Take $a\in(0,\hat\theta_{m-1}]$ and let $\hat K^+ = \sum_{j=1}^{n_*}\hat v_j\otimes\hat v_j / \hat\lambda_j$ be the generalized inverse of $\hat K$. Then
	\begin{equation*}
		\begin{split}
			\sqrt{[\hat{Q}_{m-1}, \hat{Q}_{m-1}]_0} & = \left\|\hat{Q}_{m-1}(\hat{K})\hat r\right\|  \leq \left\|\hat{Q}^{[2]}_{m-1}(\hat{K})\hat r\right\| \\
			& \leq \left\|\Pi_{a}\hat{Q}^{[2]}_{m-1}(\hat{K})\hat r\right\|  + \left\|\Pi_a^\perp \sqrt{\hat{K}^{+}}\hat{K}^{1/2}\hat{Q}^{[2]}_{m-1}(\hat{K})\hat r\right\| \\
			& \leq \left\|\Pi_{a}\hat{Q}^{[2]}_{m-1}(\hat{K})\right\|_{\rm op}\left\|\Pi_a\hat r\right\|  + \left\|\Pi_a^\perp \sqrt{\hat{K}^{+}}\right\|_{\rm op} \left\|\hat{K}^{1/2}\hat{Q}^{[2]}_{m-1}(\hat{K})\hat r\right\| \\
			& \leq \sup_{\lambda \in [0,a]}\left|\hat{Q}^{[2]}_{m-1}(\lambda)\right|\left\|\Pi_{a}\hat r\right\| + \sup_{\lambda\geq a}\frac{1}{\sqrt{\lambda}}\left\|\hat{K}^{1/2}\hat{Q}^{[2]}_{m-1}(\hat{K})\hat r\right\| \\
			& \leq  \left\|\Pi_{a}\hat r\right\| + \sqrt{[\hat{Q}^{[2]}_{m-1},\hat{Q}^{[2]}_{m-1}]_{1}/a},
		\end{split} 
	\end{equation*}
	where the second line holds since $\hat Q_m$ solves the problem in equation (\ref{eq:residual_problem}) and the last line since $|\hat{Q}^{[2]}_{m-1}(\lambda)| \leq 1,\forall \lambda \in [0,a]$; see the proof of Lemma~\ref{lemma:polynomials}.
	
	Next, under Assumption~\ref{as:complexity}, $\beta = K^{\mu} w$ with $\|w\|\leq R$, so that
	\begin{equation*}
		\begin{split}
			\left\|\Pi_{a}\hat r\right\| & \leq \left\|\Pi_a(\hat r - \hat K\beta)\right\| + \left\|\Pi_a\hat K\beta\right\| \\
			& \leq \left\| \hat r - \hat K\beta\right\| + \left\|\Pi_a\hat K\hat K^{\mu} w\right\| + \left\|\Pi_a \hat K(\hat K^{\mu} - K^{\mu}) w\right\| \\
			& \leq \sigma\sqrt{\frac{2\E\|X\|^2}{\delta n}} + \sup_{\lambda\in[0,a]}\lambda^{1+\mu}R + a \left\| \hat K^{\mu} - {K}^{\mu} \right\|_{\rm op}R \\
			& \lesssim  \sigma\sqrt{\frac{2\E\|X\|^2}{\delta n}} + a^{\mu+1} + a \left(\frac{1}{\delta n}\right)^{\frac{\mu\wedge 1}{2}},
		\end{split} 
	\end{equation*}
	where we use the inequality in equation~(\ref{eq:power_inequality}) with $\gamma=\delta$. Take $a=(c_1\sigma\sqrt{2\E\|X\|^2/\delta n})^{1/(\mu+1)}$ with a sufficiently small $c_1>0$, so that $a\leq |\hat Q'_{\hat m-1}(0)|^{-1}\leq \hat \theta_{\hat m - 1}$, cf. Lemma~\ref{lemma:polynomials} (iv). Such a constant exists since as we've already shown $|\hat Q'_{\hat m-1}(0)|\lesssim(n\delta)^{\frac{1}{2(\mu+1)}}$. Then for some $c_3>0$ 
	\begin{equation*}
		\begin{aligned}
			\sqrt{[\hat{Q}_{\hat m-1}, \hat{Q}_{\hat m-1}]_0} & \leq c_3\sigma\sqrt{\frac{2\E\|X\|^2}{\delta n}} + \sqrt{[\hat{Q}^{[2]}_{\hat m-1},\hat{Q}^{[2]}_{\hat m-1}]_{1}/a} \\
			& \leq \frac{c_3}{\tau}\left\|\hat r - \hat K\hat\beta_{\hat m-1}\right\| + \sqrt{[\hat{Q}^{[2]}_{\hat m-1},\hat{Q}^{[2]}_{\hat m-1}]_{1}/a} \\
			& = \frac{c_3}{\tau}\sqrt{[\hat{Q}_{\hat m-1}, \hat{Q}_{\hat m-1}]_0} + \sqrt{[\hat{Q}^{[2]}_{\hat m-1},\hat{Q}^{[2]}_{\hat m-1}]_{1}/a},
		\end{aligned}
	\end{equation*}
	where we use Assumption~\ref{as:stopping} and equation (\ref{eq:inner_product}). If $\tau$ is selected so that $\tau>c_3$ in Assumption~\ref{as:stopping}, then
	\begin{equation*}
		[\hat{Q}_{\hat m-1}, \hat{Q}_{\hat m-1}]_0 \leq  \left(\frac{\tau}{\tau-c_3}\right)^2[\hat{Q}^{[2]}_{\hat m-1},\hat{Q}^{[2]}_{\hat m-1}]_{1}/a.
	\end{equation*}
	Plugging this into equation~(\ref{eq:increments}) and with our choice of $a$, we get
	\begin{equation*}
		\left|\hat Q_m'(0) - \hat Q_{m-1}'(0)\right|  \lesssim \left(\delta n\right)^\frac{1}{2(\mu+1)}.
	\end{equation*}
\end{proof}

\section{Proofs of Main Results}\label{appn:proofs}
In this section, we provide detailed proofs of theorems. 

\begin{proof}[\rm \textbf{Proof of Theorem~\ref{thm:pls_rate}}]
	Take any $m\leq n_*$, $\gamma\in(0,1)$, and let $a>0$ be such that $a\leq |\hat Q'_{m}(0)|^{-1}$. By Lemma~\ref{lemma:polynomials} (iv) this ensures that $a\leq \hat\theta_m$ which we will use repeatedly in the proof.
	
	Decompose
	\begin{equation*}
		\begin{aligned}
			\hat\beta_m - \beta & = \Pi_a\hat P_m(\hat K)(\hat r - \hat K\beta) + \Pi_a \left[\hat P_m(\hat K)\hat K -I\right] \beta + \Pi_a^\perp(\hat\beta_m - \beta),
		\end{aligned}
	\end{equation*}
	where $\Pi_a=\sum_{j:\hat\lambda_j\leq a}\hat v_j\otimes \hat v_j$ and $\Pi_a^\perp = I-\Pi_a$. Then for $s\in[0,1]$, we have 
	\begin{equation*}
		\begin{split}
			\left\|\hat K^s(\hat{\beta}_m - \beta)\right\|  & \leq \left\|\Pi_a\hat K^s\hat P_m(\hat K)(\hat r - \hat K\beta) \right\| + \left\|\Pi_a\hat K^s\hat Q_m(\hat K)\beta\right\|+ \left\|\Pi_a^\perp\hat K^s(\hat{\beta}_m - \beta)\right\| \\
			& =: I + II + III.
		\end{split}
	\end{equation*}
	We will derive an upper bound for each of these three terms separately. For the first term, note that for every $s\in[0,1]$, 
	\begin{equation*}
		\begin{split}
			I &  \leq \left\|\Pi_{a}\hat K^s \hat{P}_{m}(\hat{K})\right\|\left\|\hat r - \hat K\beta \right\| \\
			& \leq \sup_{\lambda\in[0,a]}|\lambda^s\hat P_{m}(\lambda)| \sigma\sqrt{\frac{2\E\|X\|^2}{\gamma n}} \\
			& \leq a^s|\hat Q_{m}'(0)| \sigma\sqrt{\frac{2\E\|X\|^2}{\gamma n}},
		\end{split}
	\end{equation*}
	where the second line follows on an event with probability at least $1-\gamma$ by Lemma~\ref{lemma:probability_bounds} and equation~(\ref{eq:operator_bound}), and the last one by the convexity of $\hat Q_m$ on $[0,a]$:
	\begin{equation*}
		\hat{P}_{m}(\lambda)=\frac{1-\hat{Q}_{m}(\lambda)}{\lambda}\leq -\hat{Q}_{m}'(0),
	\end{equation*}
	and $\hat Q_m(\lambda)\leq \hat Q_m(0)=1$ for every $\lambda\in[0,a]$; see Lemma~\ref{lemma:polynomials} (ii).
	
	For the second term, under Assumption~\ref{as:complexity}, we have $\beta = K^{\mu}w$ with $\|w\|\leq R$, so that
	\begin{equation*}
		\begin{split}
			II &   =   \left\|\Pi_{a} \hat K^s\hat{Q}_{m}(\hat{K})K^{\mu}w \right\| \\
			&  \leq \left\|\Pi_{a} \hat K^s\hat{Q}_{m}(\hat{K})\hat{K}^{\mu}w\right\| +  \left\|\Pi_{a}\hat K^s\hat{Q}_{m}(\hat{K}) \left[K^{\mu} - \hat{K}^{\mu}\right]w\right\| \\
			& \leq  \sup_{\lambda \in [0,a]} |\lambda^{\mu+s}\hat{Q}_{m}(\lambda)| R + \sup_{\lambda \in [0,a]} |\lambda^s\hat{Q}_{m}(\lambda)|\left\|K^{\mu} - \hat{K}^{\mu}\right\|_{\rm op} R \\
			& \leq a^{\mu+s}R + a^s(c_{\mu}\mathbf{1}_{\mu\leq 1} + \mu\nu^{\mu-1}\mathbf{1}_{\mu>1})  \left(\frac{2\E\|X\|^4}{\gamma n}\right)^{\frac{\mu\wedge 1}{2}}R,
		\end{split}
	\end{equation*}
	where the last line follows since $|\hat Q_{m}(\lambda)|\leq1$ by Lemma \ref{lemma:polynomials} (ii) and equation (\ref{eq:power_inequality}) on an event with probability at least $1-\gamma$.
	
	Lastly, let $\hat K^+ = \sum_{j=1}^{n_*}\hat v_j\otimes\hat v_j / \hat\lambda_j$ be the generalized inverse of $\hat K$. Then we bound the third term as follows
	\begin{equation*}
		\begin{split}
			III &  \leq  \left\|\Pi_a^\perp\hat K^s\hat{K}^{+}\right\| \left\|\hat K(\hat\beta_m - \beta)\right\|\\
			& \leq \sup_{\lambda\geq a}\lambda^{1}\left\|(\hat K\hat\beta_m - \hat r) + (\hat r - \hat K\beta) \right\|  \\
			& \leq a^{1}\left\{\left\|\hat K\hat\beta_{m} - \hat r\right\| + \sigma \sqrt{\frac{2\E\|X\|^2}{\gamma n}}\right\},
		\end{split}
	\end{equation*}
	where we use Lemma~\ref{lemma:probability_bounds} on an event with probability at least $1-\gamma$. Combining the three bounds, we obtain for $m\leq n_*$
	\begin{equation}\label{eq:bound_m}
		\left\|\hat K^s(\hat{\beta}_m - \beta)\right\|  \lesssim a^{1}\left\{\left\|\hat K\hat\beta_{m} - \hat r\right\| +(\gamma n)^{-1/2}\right\} + a^{\mu+s} +  a^s(\gamma n)^{-\frac{\mu\wedge 1}{2}} +  a^s|\hat Q_{ m}'(0)| (\gamma n)^{-1/2},
	\end{equation}
	where we use  $\nu\lesssim 1$; cf. equation (\ref{eq:nu}). Taking $a=|\hat Q'_m(0)|^{-1}$, this gives
	\begin{equation*}
		\left\|\hat K^s(\hat{\beta}_m - \beta)\right\|  = O_P\left( |\hat Q'_m(0)|^{-s+1}\left\{\left\|\hat K\hat\beta_{m} - \hat r\right\| + n^{-1/2}\right\} + |\hat Q'_m(0)|^{-(\mu+s)} +  |\hat Q'_m(0)|^{-s}n^{-\frac{\mu\wedge 1}{2}}\right).
	\end{equation*}
	By Lemma~\ref{lemma:residual}
	\begin{equation*}
		\left\|\hat K\hat{\beta}_{m} - \hat r\right\| 
		= O_P\left(n^{-1/2} + |\hat Q'_{ m}(0)|^{-1}n^{-1/2} + |\hat{Q}_{m}'\left( 0\right)| ^{-(\mu+1)} \right).
	\end{equation*}
	Therefore, 
	\begin{equation*}
		\left\|\hat K^s(\hat{\beta}_m - \beta)\right\|  = O_P\left(|\hat Q_m'(0)|^{-s+1}n^{-1/2} + |\hat Q_m'(0)|^{-(\mu+s)} + |\hat Q'_m(0)|^{-s}n^{-\frac{\mu\wedge 1}{2}}\right),\qquad \forall s\in[0,1].
	\end{equation*}
	This proves the result if $s=0$. If $s\in(0,1]$, then 
	
	\begin{equation*}
		\begin{aligned}
			\left\|K^s(\hat{\beta}_{m} - \beta)\right\| & = \left\|\hat K^s(\hat{\beta}_{m} - \beta) + (K^s - \hat K^s)(\hat{\beta}_{m} - \beta)\right\| \\
			& \leq \left\|\hat K^s(\hat{\beta}_{m} - \beta)\right\|  + \left\|\hat K^s - K^s\right\|_{\rm op} \left\|\hat{\beta}_{m} - \beta\right\|  \\
			& = O_P\left(|\hat Q_m'(0)|^{-s+1}n^{-1/2} + |\hat Q_m'(0)|^{-(\mu+s)} + |\hat Q'_m(0)|^{-s}n^{-\frac{\mu\wedge 1}{2}} \right) \\
			& \qquad + O_P\left(|\hat Q_m'(0)|n^{-\frac{1+s}{2}} + |\hat Q_m'(0)|^{-\mu}n^{-\frac{s}{2}} + n^{-\frac{s + \mu\wedge 1}{2}} \right) \\
			& = O_P\left(|\hat Q_m'(0)|^{-s+1}n^{-1/2} + |\hat Q_m'(0)|^{-(\mu+s)} + |\hat Q'_m(0)|^{-s}n^{-\frac{\mu\wedge 1}{2}} \right),
		\end{aligned}
	\end{equation*}
	provided that $|\hat Q_m'(0)|=O_P(n^{1/2})$.
\end{proof}

\begin{proof}[\rm \textbf{Proof of Theorem~\ref{thm:lower_bound}}]
	We adopt an approach similar to \cite{cai2012minimax} Theorem 1; see also \cite{tsybakov2008introduction}, Chapter 2. Recall that the lower bound for a restricted class of models yields the lower bound for the general case. Therefore, we can assume without loss of generality that $\varepsilon_i|X_i\sim N(0,\sigma^2)$ and $K:\mathbb{H}\to\mathbb{H}$ has a spectral decomposition $(\lambda_j,v_j)_{j\geq 1}$ with $\lambda_1=1$ and $\lambda_j=1/(j\log^aj)$ for $j=2,3,\dots$ for some $a>1$. We will also consider the family of slope parameters
	\begin{equation*}
		\beta_{\theta} = Rm^{-1/2}\sum_{l=m+1}^{2m}\theta_l\lambda_l^{\mu}v_l,\qquad \theta=(\theta_{m+1},\dots,\theta_{2m})\in\{0,1\}^m
	\end{equation*}
	for some $R>0$ and $m$ specified below. It is easy to see that by the orthonormality of $(v_l)_{l\geq 1}$
	\begin{equation*}
		\sum_{j=1}^\infty\frac{\langle \beta_{\theta},v_j\rangle^2}{\lambda_j^{2\mu}} = \frac{R^2}{m}\sum_{j=m+1}^{2m}\theta_j^2 \leq R^2
	\end{equation*}
	and that
	\begin{equation*}
		\sum_{j=1}^\infty \lambda_j = 1 + \sum_{j=2}^\infty \frac{1}{j\log^aj} \leq C
	\end{equation*}
	for some $C>0$. Therefore, $(\beta_{\theta},K)\in\mathcal{S}(\mu,R,C),\forall\theta\in\{0,1\}^m$, cf. Assumption~\ref{as:complexity}.
	
	Let $H(\theta,\theta')=\sum_{j=1}^m\mathbf{1}\{\theta_j\ne \theta_j' \}$ be the Hamming distance between the binary sequences $\theta,\theta'\in\{0,1\}^m$. By the Varshamov-Gilbert bound, see \cite{tsybakov2008introduction}, Lemma 2.9, if $m\geq 8$, there exists $\{\theta^{(0)},\dots,\theta^{(M)} \}\subset \{0,1\}^m$ such that 
	\begin{itemize}
		\item[(a)] $\theta^{(0)}=(0,\dots,0)$; 
		\item[(b)] $H(\theta^{(j)},\theta^{(k)})\geq \frac{m}{8}, \forall \;0\leq j<k\leq M$;
		\item[(c)] $M\geq 2^{m/8}$.
	\end{itemize}
	For every $A>0$,
	\begin{equation}\label{eq:lower_bound}
		\begin{aligned}
			& \sup_{(\beta,K)\in\mathcal{S}(\mu,R,C)}\mathrm{Pr}\left(\left\|K^s(\hat\beta - \beta) \right\| \geq An^{-\frac{\mu+s}{2(\mu+1)}}(\log n)^{-a(\mu+s)} \right) \\
			& \geq  \max_{\theta\in\{\theta^{(0)},\dots,\theta^{(M)} \}}\mathrm{Pr}\left(\left\|K^s(\hat\beta - \beta_{\theta}) \right\| \geq An^{-\frac{\mu+s}{2(\mu+1)}}(\log n)^{-a(\mu+s)} \right).
		\end{aligned}
	\end{equation}
	To obtain the lower bound for the right-hand side of the equation (\ref{eq:lower_bound}), we will use \cite{tsybakov2008introduction}, Theorem 2.5 and a specific choice of $A<\infty$. To that end, we need to check the following conditions:
	\begin{itemize}
		\item[(i)] $\|K^s(\beta_{\theta^{(j)}} - \beta_{\theta^{(k)}})\| \geq 2An^{-\frac{\mu+s}{2(\mu+1)}}(\log n)^{-a(\mu+s)}$ for all $0\leq j<k\leq M$;
		\item[(ii)] $P_j<<P_0,\forall j=1,\dots,M$, where $P_j$ denotes the distribution of $(Y_i,X_i)_{i\geq 1}$ for the slope parameter $\beta_{\theta^{(j)}}$;
		\item[(iii)] For $\alpha\in(0,1/8)$,
		\begin{equation*}
			\frac{1}{M}\sum_{j=1}^MKL(P_j,P_0)\leq \alpha\log M,
		\end{equation*}
		where $KL$ is the Kullback-Leibler divergence between $P_j$ and $P_0$. 
	\end{itemize}
	
	For the first condition, note that since $H(\theta^{(j)},\theta^{(k)}) = \sum_{l=m+1}^{2m}(\theta_l^{(j)} - \theta_l^{(k)})^2$, we have
	\begin{equation*}
		\begin{aligned}
			\left\|K^s(\beta_{\theta^{(j)}} - \beta_{\theta^{(k)}})\right\|^2 & = \left\|Rm^{-1/2}\sum_{l=m+1}^{2m}(\theta_l^{(j)} - \theta_l^{(k)}) \lambda_l^{\mu+s}v_l \right\|^2 \\
			& = \frac{R^2}{m}\sum_{l=m+1}^{2m}(\theta_l^{(j)} - \theta_l^{(k)})^2 \lambda_l^{2(\mu+s)} \\
			& \geq \frac{R^2}{m}\lambda_{2m}^{2(\mu+s)} H(\theta^{(j)},\theta^{(k)}) \\
			& \geq \frac{R^2}{8}\lambda_{2m}^{2(\mu+s)} = \frac{R^2}{8}(2m)^{-2(\mu+s)}\log^{-2a(\mu+s)}(2m) \\
			& \geq 4A^2n^{-\frac{\mu+s}{\mu+1}}\log^{-2a(\mu+s)} n
		\end{aligned}
	\end{equation*}
	where the last two inequalities follow from (b) provided that $m\leq n^\frac{1}{2(\mu+1)}$ for some $A>0$. This verifies (i).
	
	Next, since $Y_i|X_i\sim N(\langle X_i,\beta_{\theta^{(j)}}\rangle,\sigma^2)$ under $P_j$, we have $P_j<<P_0,\forall j=1,\dots,M$ with the log-likelihood ratio
	\begin{equation*}
		\log \frac{\mathrm{d} P_j}{\mathrm{d} P_0} = \frac{1}{\sigma ^{2}}\sum_{i=1}^{n}\left( Y_{i}-\left \langle X_{i}, \beta _{\theta^{(j)} } \right
		\rangle \right) \left \langle X_{i},\beta_{\theta^{(j)}}  - \beta _{\theta^{(0)}}\right \rangle + \frac{1}{2\sigma ^{2}}\sum_{i=1}^{n}\left \langle
		X_{i},\beta _{\theta^{(j)}} - \beta _{\theta^{(0)}}\right \rangle ^{2}.
	\end{equation*}
	To verify (iii), we compute the Kullback--Leibler divergence:
	\begin{equation*}
		\begin{aligned}
			KL(P_j,P_0) & = \int\log\frac{\mathrm{d} P_j}{\mathrm{d} P_0}\mathrm{d} P_j \\
			& = \frac{n}{2\sigma^2}\mathbb{E}\left\langle X_{i},\beta _{\theta^{(0)}} - \beta _{\theta^{(j)}}\right\rangle^{2} \\
			& = \frac{n}{2\sigma^2}\left\|K^{1/2}(\beta _{\theta^{(0)}} - \beta _{\theta^{(j)}})\right\|^2 \\
			& = \frac{n}{2\sigma^2}\left\|Rm^{-1/2}\sum_{l=m+1}^{2m}\theta_l^{(j)}\lambda_l^{\mu+1/2} v_l  \right\|^2 \\
			& =  \frac{nR^2}{2\sigma^2m}\sum_{l=m+1}^{2m}\left(\theta_l^{(j)}\right)^2\lambda_l^{2\mu+1} \\
			& \leq \frac{nR^2}{2\sigma^2}\lambda_m^{2\mu+1} = \frac{nR^2}{2\sigma^2}m^{-(2\mu+1)}\left(\log m\right)^{-a(2\mu+1)} \\
			& \leq \alpha\frac{m}{8}\log 2 = \alpha\log 2^{m/8},
		\end{aligned}
	\end{equation*}
	provided that $m\geq (c_0n)^\frac{1}{2(\mu+1)}(\log m)^{-a\frac{2\mu+1}{2(\mu+1)}}$ with $c_0=4R^2/(\sigma^2\alpha\log 2)$ some $\alpha\in(0,1/8)$.\footnote{To ensure that this constraint holds and that $m\leq n^\frac{1}{2(\mu+1)}$, we can take $m$ as a fraction of  $n^\frac{1}{2(\mu+1)}$.} This verifies (iii) in light of (c).
	
	Therefore, by \cite{tsybakov2008introduction}, Theorem 2.5
	\begin{equation*}
		\liminf_{n\to\infty}\inf_{\hat\beta}\max_{\theta\in\{\theta^{(0)},\dots,\theta^{(M)} \}}\mathrm{Pr}\left(\left\|K^s(\hat\beta - \beta_{\theta}) \right\| \geq An^{-\frac{\mu +s}{2(\mu +1)}} \log^{-a(\mu+s)} n\right) \geq 1-2\alpha>0
	\end{equation*}
	which implies the result in light of the inequality (\ref{eq:lower_bound}).
\end{proof}

\begin{proof}[\rm \textbf{Proof of Theorem~\ref{thm:pls_adaptation}}]
	Setting $m=\hat m$ and $\gamma=\delta$ in equation~(\ref{eq:bound_m}), under Assumption~\ref{as:stopping}, we obtain
	\begin{equation}\label{eq:K_hat_s_error}
		\left\|\hat K^s(\hat{\beta}_{\hat m}- \beta)\right\|  \lesssim a^{1}(\delta n)^{-1/2} + a^{\mu+s} +  a^s(\delta n)^{-\frac{\mu\wedge 1}{2}} +  a^s|\hat Q_{\hat m}'(0)| (\delta n)^{-1/2}
	\end{equation}
	for every $a\leq |\hat Q_{\hat m}'(0)|^{-1}$. Now we will choose the truncation level $a$. Suppose that $s\in[0,1)$. Then the function $a\mapsto a^{1}(\delta n)^{-1/2} + a^{\mu+s} $	is minimized at $a^* = \left\{(\delta n)^{1/2}(\mu+s)/(1-s)\right\}^{-\frac{1}{\mu+1}}$. If $a^* \leq |\hat Q_{\hat m}'(0)|^{-1}$, we shall choose $a=a^*$, in which case since $\delta\geq 1/n$, we obtain
	\begin{equation*}
		\left\|\hat K^s(\hat{\beta}_{\hat m} - \beta)\right\| \lesssim  (\delta n)^{-\frac{\mu+s}{2(\mu+1)}} + (\delta n)^{-\frac{s+\mu+1}{2(\mu+1)}}|\hat Q'_{\hat m}(0)| \lesssim (\delta n)^{-\frac{\mu+s}{2(\mu+1)}}.
	\end{equation*}
	On the other hand, if $a^*> |\hat Q_{\hat m}'(0)|^{-1}$, we shall choose $a=|\hat Q_{\hat m}'(0)|^{-1}$. Then
	\begin{equation*}
		\begin{split}
			\left\|\hat K^s(\hat{\beta}_{\hat m} - \beta)\right\| & \lesssim |\hat Q_{\hat m}'(0)|^{1-s}(\delta n)^{-1/2} + |\hat Q_{\hat m}'(0)|^{-(\mu+s)} +  |\hat Q_{\hat m}'(0)|^{-s}(\delta n)^{-\frac{\mu\wedge 1}{2}}  \\
			& \lesssim   (\delta n)^{-\frac{\mu+s}{2(\mu+1)}},
		\end{split}
	\end{equation*}
	where the last line follows from $|\hat Q'_{\hat m}(0)|\lesssim (\delta n)^{\frac{1}{2(\mu+1)}}$ by Lemma~\ref{lemma:Q_prime}, and from $|\hat Q_{\hat m}'(0)|^{-1}< a^* \lesssim (\delta n)^{-\frac{1}{2(\mu+1)}}$ and $(\delta n)^{-1}\leq 1$.
	
	If $s=1$, then setting $a=c(\delta n)^{-\frac{1}{2(\mu+1)}}$ in equation (\ref{eq:K_hat_s_error}), for some $c>0$ such that $a\leq |\hat Q'_{\hat m}(0)|^{-1}$, cf. Lemma~\ref{lemma:Q_prime}, we get
	\begin{equation*}
		\begin{aligned}
			\left\|\hat K^s(\hat{\beta}_{\hat m} - \beta)\right\|  & \lesssim (\delta n)^{-1/2} + (\delta n)^{-\frac{1+(\mu\wedge 1)(\mu+1)}{2(\mu+1)}} +  (\delta n)^{-\frac{\mu+2}{2(\mu+1)}}|\hat Q_{\hat m}'(0)| \\
			& \lesssim (\delta n)^{-1/2},
		\end{aligned}
	\end{equation*}
	where the last line follows from Lemma~\ref{lemma:Q_prime} and $(\delta n)^{-1}\leq 1$. Therefore, we've just established for every $s\in[0,1]$
	\begin{equation}\label{eq:K_hat_rate}
		\left\|\hat K^s(\hat{\beta}_{\hat m} - \beta)\right\|  \lesssim (\delta n)^{-\frac{\mu+s}{2(\mu+1)}}.
	\end{equation}
	This proves the statement of the theorem in the special case when $s=0$. On the other hand, if $s\in(0,1]$, we have
	\begin{equation*}
		\begin{aligned}
			\left\|K^s(\hat{\beta}_{\hat m} - \beta)\right\| & \leq \left\|\hat K^s(\hat{\beta}_{\hat m} - \beta)\right\|  + \left\|\hat K^s - K^s\right\|_{\rm op} \left\|\hat{\beta}_{\hat m} - \beta\right\|  \\
			& \lesssim (\delta n)^{-\frac{\mu+s}{2(\mu+1)}}  + (\delta n)^{-\frac{s\wedge 1}{2}}(\delta n)^{-\frac{\mu}{2(\mu+1)}} \\
			& \lesssim (\delta n)^{-\frac{\mu+s}{2(\mu+1)}},
		\end{aligned}
	\end{equation*}
	where we use equation~(\ref{eq:K_hat_rate}) and (\ref{eq:power_inequality}), and $(\delta n)^{-1}\leq 1$.
\end{proof}

\begin{proof}[\rm \textbf{Proof of Theorem~\ref{thm:number_of_components}}]
	For $1\leq m\leq n_*$ and $\nu\geq 0$, let $\tilde P_m^{(\nu)}$ be a Jacobi polynomial of degree $m$ on $[-1,1]$, i.e. a polynomial, orthogonal with respect to the weight $\lambda\mapsto (1-\lambda)^{\alpha}(1+\lambda)^{\beta}$, where we set $\alpha=-1/2$ and $\beta=2\nu-1/2$ with $\nu>0$. Let $P_m^{(\nu)}(\lambda)=\tilde P_m^{(\nu)}(2\lambda/\hat \lambda _1- 1)/\tilde P_m^{(\nu)}(-1),\forall\lambda\in[0,\hat\lambda_1]$ be a shifted Jacobi polynomial, normalized so that $P^{(\nu)}_m(0)=1$. By \cite{engl1996regularization}, Appendix A.2, p.294, there exists $c_{\nu}>0$ such that
	\begin{equation}\label{eq:jacobi_bound}
		|P_m^{(\nu)}(\lambda)| \leq c_{\nu}(1+m^2\lambda)^{-\nu},\qquad \forall \lambda\in[0,\hat\lambda_1],\;m\geq 0.
	\end{equation}	
	Recall also that under Assumption~\ref{as:complexity}, we have $\beta = K^{\mu}w$ with $\|w\| \leq R$. Then
	\begin{equation*}
		\begin{split}
			\left\|\hat r - \hat K\hat{\beta}_{m}\right\|  & = \left\| \hat{Q}_{m}(\hat{K})\hat r \right\| \\
			& \leq  \left\| P_m^{(\mu+1)}(\hat{K})\hat r \right\| \\
			& \leq  \left\|P_m^{(\mu+1)}(\hat{K})(\hat r - \hat K\beta)\right\| + \left\|P_m^{(\mu+1)}(\hat{K})\hat K\beta \right\| \\
			& \leq \left\|P_m^{(\mu+1)}(\hat K)\right\|_{\rm op} \left\|\hat r - \hat K\beta\right\|  + \left\|P_m^{(\mu+1)}(\hat{K})\hat{K}\hat K^{\mu}w\right\| \\
			& \qquad\qquad + \left\|P_m^{(\mu+1)}(\hat{K})\hat{K}(\hat K^{\mu} - K^{\mu})w\right\| \\
			& \lesssim  \sup_{\lambda\in[0,\hat\lambda_1]}|P_m^{(\mu+1)}(\lambda)| \left\|\hat r - \hat K\beta\right\| +  \sup_{\lambda\in[0,\hat\lambda_1]} \left|\lambda^{\mu+1}P_m^{(\mu+1)}(\lambda)\right| \\
			& \qquad\qquad  +  \sup_{\lambda\in[0,\hat\lambda_1]}|\lambda P_m^{(\mu+1)}(\lambda)|\left\|\hat K^{\mu} - K^{\mu}\right\|_{\rm op}  \\
			& \lesssim \left\|\hat r - \hat K\beta\right\| + m^{-2(\mu+1)} + \hat\lambda_1\left\|\hat K^{\mu} - K^{\mu}\right\|_{\rm op} \\
			& \lesssim \sigma\sqrt{\frac{2\E\|X\|^2}{\gamma n}}  + m^{-2(\mu+1)} + \|\hat K\|_{\rm op}\left(\frac{2\E\|X\|^4}{\gamma n}\right)^{\frac{\mu\wedge 1}{2}},
		\end{split}
	\end{equation*}
	where the second line follows since $\hat Q_m$ minimizes the problem in equation~(\ref{eq:residual_problem}); the sixth from the inequality (\ref{eq:jacobi_bound}); and the last by Lemma~\ref{lemma:probability_bounds} and the inequality (\ref{eq:power_inequality}) with probability at least $1-\gamma$ for every $\gamma\in(0,1)$. 
	
	Therefore, if $\mu\geq 1$, under Assumption~\ref{as:data} by Lemma~\ref{lemma:probability_bounds} and the inequality (\ref{eq:power_inequality}), we obtain
	\begin{equation*}
		\left\|\hat r - \hat K\hat{\beta}_{m}\right\| \leq  c \left\{\sigma\sqrt{\frac{2\E\|X\|^2}{\delta n}}  + m^{-2(\mu+1)}  \right\}
	\end{equation*}
	for some $c>0$ and $\delta\geq 1/n$ on the event with probability at least $1-\delta$. 	According to the stopping rule in the Assumption~\ref{as:stopping}, we also know that
	\begin{equation*}
		\tau \sigma\sqrt{\frac{2\E\|X\|^2}{\delta n}} \leq \left\|\hat r - \hat K\hat{\beta}_{\hat m-1}\right\|
	\end{equation*}
	Therefore, 
	\begin{equation*}
		\begin{aligned}
			(\tau-c)\sigma\sqrt{\frac{2\E\|X\|^2}{\delta n}} & \lesssim (\hat m - 1)^{-2(\mu+1)},
		\end{aligned}
	\end{equation*}
	and whence $\hat m \lesssim (\delta n)^{\frac{1}{4(\mu+1)}}$, provided that $\tau >c$.
\end{proof}

\begin{proof}[\rm \textbf{Proof of Theorem~\ref{thm:number_of_components2}}]
	For some integers $m\geq k\geq 0$, put
	\begin{equation*}
		G_{m}(\lambda) :=  \prod^{k}_{j = 1}\left(1 - \frac{\lambda}{\hat \lambda_{j}}\right)P_{m-k}^{(\mu+1)} \left( \frac{\lambda}{\hat{\lambda}_{k+1}}\right)
	\end{equation*}
	where $P_m^{(\nu)}(\lambda)=\tilde P_m^{(\nu)}(2\lambda/\hat \lambda_{k+1}- 1)/\tilde P_m^{(\nu)}(-1)$ is a shifted and normalized Jacobi polynomial on $[0,\hat\lambda_{k+1}]$, defined to be zero outside of this interval; see the proof of Theorem~\ref{thm:number_of_components}. Then $G_m$ is an $m^{\rm th}$ degree polynomial with $G_m(0)=1$ and
	\begin{equation*}
		\sup_{\lambda\in[0,\hat\lambda_{k+1}]}\left|\lambda^{\mu+1} G_m(\lambda) \right| \leq \sup_{\lambda\in[0,\hat\lambda_{k+1}]}\left|\lambda^{\mu+1} P_{m-k}^{(\mu+1)} \left( \frac{\lambda}{\hat{\lambda}_{k+1}}\right) \right|  \leq \hat \lambda_{k+1}^{\mu+1}c_{\mu}(m-k)^{-2(\mu+1)},
	\end{equation*}	
	where we use the inequality (\ref{eq:jacobi_bound}). Then since under Assumption~\ref{as:complexity}, $\beta= K^{\mu}w$ with $\|w\|\leq R$, we have with probability at least $1-\gamma$ for every $\gamma\in(0,1)$
	\begin{equation*}
		\begin{split}
			\left\|\hat r - \hat K\hat\beta_m\right\| & =\left\|\hat{Q}_{m}(\hat{K})\hat r\right\| \leq \left\|G_{m}(\hat{K})\hat r\right\| \\
			& \leq \left\|G_{m}(\hat{K})(\hat r - \hat K\beta)\right\| + \left\|G_{m}(\hat{K})\hat K\hat K^{\mu} w\right\| + \left\|G_{m}(\hat{K})\hat K(\hat K^{\mu} - K^{\mu}) w\right\| \\
			& \leq \sup_{\lambda\in[0,\hat\lambda_{k+1}]}|G_m(\lambda)|\left\|\hat r - \hat K\beta\right\| + R\sup_{\lambda\in[0, \hat\lambda_{k+1}]}|\lambda^{\mu+1}G_m(\lambda)| \\
			& \qquad\qquad + R\sup_{\lambda\in[0, \hat\lambda_{k+1}]}|\lambda G_m(\lambda)|\left\|\hat K^{\mu} - K^{\mu}\right\| \\
			& \lesssim \sigma\sqrt{\frac{2\E\|X\|^2}{\gamma n}} + \hat\lambda_{k+1}^{\mu+1}(m-k)^{-2(\mu+1)} +  \|\hat K\|_{\rm op}\left(\frac{2\E\|X\|^2}{\gamma n}\right)^\frac{\mu\wedge 1}{2} \\
		\end{split}
	\end{equation*}
	where the first inequality follows since $\hat Q_m$ solves the problem in equation (\ref{eq:residual_problem}) and for the last inequality, we use $|G_m(\lambda)|\leq |P_{m-k}^{(\mu+1)}(\lambda/\hat\lambda_{k+1})|\lesssim 1$; see equation~(\ref{eq:jacobi_bound}). Recall that $\|\hat K\|_{\rm op}\lesssim 1$ on an event with probability at least $1-\gamma$ for $\gamma\geq 1/n$; see Lemma~\ref{lemma:probability_bounds}. Therefore, since $\mu\geq 1$, we obtain
	\begin{equation*}
		\left\|\hat r - \hat K\hat{ \beta}_{m}\right\| \leq  c \left\{\sigma\sqrt{\frac{2\E\|X\|^2}{\delta n}}  + \hat\lambda_{k+1}^{\mu+1}(\hat m-k)^{-2(\mu+1)}  \right\}
	\end{equation*}
	for some $c>0$ and $\delta\geq 1/n$ on an event with probability at least $1-\delta$. 
	
	According to the stopping rule in the Assumption~\ref{as:stopping}, we also know that for some $\delta\in(0,1)$
	\begin{equation*}
		\tau \sigma\sqrt{\frac{2\E\|X\|^2}{\delta n}} \leq \left\|\hat r - \hat K\hat{\beta}_{\hat m-1}\right\|. 
	\end{equation*}
	Therefore, if $\tau>c$, we obtain
	\begin{equation*}
		\begin{aligned}
			(\tau-c)\sigma\sqrt{\frac{2\E\|X\|^2}{\delta n}} & \leq c \hat\lambda_{k+1}^{\mu+1}(\hat m-k-1)^{-2(\mu+1)}
		\end{aligned}
	\end{equation*}
	which implies that
	\begin{equation}\label{eq:stopping_eigen}
		\begin{aligned}
			\hat m-k-1 & \lesssim \hat\lambda_{k+1}^{1/2}(\delta n)^{\frac{1}{4(\mu+1)}} \\
			& \leq (\delta n)^{\frac{1}{4(\mu+1)}}\left\{\lambda_{k+1}^{1/2} + \left\|\hat K^{1/2} - K^{1/2}\right\|_{\rm op} \right\} \\
			& \lesssim (\delta n)^{\frac{1}{4(\mu+1)}}\left\{\lambda_{k+1}^{1/2} + (\delta n)^{-1/4} \right\},
		\end{aligned}
	\end{equation}
	where we use Weyl's inequality and equation~(\ref{eq:power_inequality}).
	
	\textbf{Case (i):} if $\lambda_k = O(k^{-2\kappa})$ with $\kappa>0$, we can take $k \sim \hat m/2$.  In this case equation (\ref{eq:stopping_eigen}) implies
	\begin{equation*}
		\hat m \lesssim (\delta n)^{\frac{1}{4(\mu+1)}} \hat m^{-\kappa} + (\delta n)^{-\frac{\mu}{4(\mu+1)}}.
	\end{equation*}
	If the second term in this upper bound dominates the first one, then $\hat m\lesssim (\delta n)^{-\frac{\mu}{4(\mu+1)}}=o(1)$, which is a contradiction. Therefore, $\hat m \lesssim (\delta n)^\frac{1}{4(\kappa+1)(\mu+1)}$.
	
	\textbf{Case (ii):} if $\lambda_{j} = O(q^{j})$ with $q\in(0,1)$, we can take $k=\hat m-2$. In this case equation (\ref{eq:stopping_eigen}) implies 
	\begin{equation*}
		1\lesssim (\delta n)^{\frac{1}{4(\mu+1)}} q^{(\hat m - 1)/2} + (\delta n)^{-\frac{\mu}{4(\mu+1)}}.
	\end{equation*}
	If the second term in this upper bound dominates the first one, then $1\lesssim (\delta n)^{-\frac{\mu}{4(\mu+1)}}=o(1)$, which is a contradiction. Therefore, $\hat m \lesssim 1+ \log(\delta n)$ since $\log q<0$.
\end{proof}

\begin{proof}[\rm \textbf{Proof of Theorem~\ref{thm:test}}]
	Observe that
	\begin{equation*}
		\hat{K}(\hat{\beta}_m - b) = (\hat{r} - \hat{K}b) + (\hat{K} \hat{\beta}_m - \hat{r}).
	\end{equation*}
	Then
	\begin{align*}
		T_n & = n \big\| \hat{K}(\hat{\beta}_m - b) \big\|^2 \\
		& = \left\langle \sqrt{n} \hat{K}(\hat{\beta}_m - b), \sqrt{n} \hat{K}(\hat{\beta}_m - b) \right\rangle \\
		& = \| \sqrt{n}(\hat{r} - \hat{K}b) \|^2 + n \| \hat{K} \hat{\beta}_m - \hat{r} \|^2 + 2 \left\langle \sqrt{n}(\hat{r} - \hat{K}b), \sqrt{n}(\hat{r} - \hat{K} \hat{\beta}_m) \right\rangle \\
		& =: I_n + II_n + III_n.
	\end{align*}
	
	By assumption, $II_n = o_P(1)$. Under $H_0$, by the Hilbert space central limit theorem, see \cite{bosq2000linear}, Theorem 2.7,
	\begin{equation*}
		\sqrt{n}(\hat{r} - \hat{K} b) = \frac{1}{\sqrt{n}} \sum_{i=1}^n \varepsilon_i X_i \xrightarrow{d} G,
	\end{equation*}
	where $G$ is a zero-mean Gaussian element in $\mathbb{H}$ with covariance operator $V$. By the continuous mapping theorem, see \cite{van1996weak}, Theorem 1.3.6
	\begin{equation*}
		I_n \xrightarrow{d} \|G\|^2,
	\end{equation*}
	By the Karhunen–Loève expansion, $G =  \sum_{j=1}^\infty \omega_j^{1/2} Z_j\varphi_j$, so we obtain
	\begin{equation*}
		I_n \xrightarrow{d} \sum_{j=1}^\infty \omega_jZ_j^2
	\end{equation*}
	under the null hypothesis. Note that this also shows that $\|\sqrt{n}(\hat{r} - \hat{K} \beta)\| = O_P(1)$. Then, under $H_0$
	\begin{equation*}
		|III_n| \leq 2 \|\sqrt{n}(\hat{r} - \hat{K} b)\| \cdot \sqrt{n} \|\hat{r} - \hat{K} \hat{\beta}_m\| = o_P(1),
	\end{equation*}
	since $\|\hat{r} - \hat{K} \hat{\beta}_m\| = o_P(n^{-1/2})$. This shows that under the null hypothesis, we have
	\begin{equation*}
		T_n\xrightarrow{d} \sum_{j=1}^\infty \omega_j Z_j^2.
	\end{equation*}
	
	\medskip 
	
	Next, under the fixed alternative hypothesis $\beta\ne b$, we have
	\begin{equation*}
		\begin{aligned}
			\hat{K}(\hat{\beta}_m - b) = (\hat{r} - \hat{K}\beta) + \hat{K}(\beta-b)  + (\hat{K} \hat{\beta}_m - \hat{r}).
		\end{aligned}
	\end{equation*}
	Note that the last term is $o_P(1)$ under the maintained assumptions. On the other hand, by the Hilbert space law of large numbers, see \cite{bosq2000linear}, Theorem 2.4, 
	\begin{equation*}
		\hat r - \hat K\beta =o_P(1) \qquad \text{and}\qquad \hat K(\beta - b) \xrightarrow{p}K(\beta-b).
	\end{equation*}
	This shows that
	\begin{equation*}
		n^{-1}T_n \xrightarrow{p} \|K(\beta-b)\|^2>0
	\end{equation*}
	as long as $K:\mathbb{H}\to\mathbb{H}$ does not have zero eigenvalues. Therefore, by Slutsky's theorem $T_n\xrightarrow{\rm a.s.}\infty$.
	
	\medskip
	
	Lastly, under the local alternative hypothesis, we have
	\begin{equation*}
		\begin{aligned}
			\sqrt{n}\hat K(\hat\beta_m - b) & = \sqrt{n}(\hat r - \hat K\beta) + \hat K\Delta + \sqrt{n}(\hat K\hat\beta_m - \hat r) \\
			& \xrightarrow{d} G + K\Delta.
		\end{aligned}
	\end{equation*}
	This shows that
	\begin{equation*}
		\begin{aligned}
			T_n & = n \big\| \hat{K}(\hat{\beta}_m - b) \big\|^2 \\
			& \xrightarrow{d} \|G + K\Delta\|^2 \\
			& = \|G\|^2 + 2\langle G,K\Delta\rangle + \|K\Delta\|^2 \\
			& = \sum_{j=1}^\infty\omega_jZ_j^2 + 2\sum_{j=1}^\infty\omega_j^{1/2}Z_j\langle \varphi_j,K\Delta\rangle + \|K\Delta\|^2,
		\end{aligned}
	\end{equation*}
	where the last line follows by the Karhunen–Loève expansion, $G =  \sum_{j=1}^\infty \omega_j^{1/2} Z_j\varphi_j$.
\end{proof}

\begin{proof}[\rm \textbf{Proof of Corollary~\ref{cor:size_power}}]
	The first two statements follow trivially from Theorem~\ref{thm:test}. The last statement follows since by Markov's inequality
	\begin{equation*}
		\lim_{n\to\infty}\Pr(T_n > z_{1-\alpha}) =  z_{1-\alpha}^{-1}\sum_{j=1}^\infty\omega_j + z_{1-\alpha}^{-1}\|K\Delta\|^2,
	\end{equation*}
	where we use the fact that $Z_j\sim N(0,1)$.
\end{proof}

\section{Comparison to PCA}\label{sec:pca}
In this section, we shed some light on the behavior of functional PLS relative to PCA. We will show that for the same fixed number of components $m$, PLS fits the empirical moment better than PCA, hence, it may require a smaller number of components to obtain a comparable fit. We also show that the regularization bias part of the estimation and prediction risk of PLS is smaller than the one of the PCA. Therefore, the adaptive PLS basis is better suited for approximating the slope coefficient. 

In what follows, we will use
\begin{equation*}
	\hat\beta_m^{\rm PLS}= \sum_{j=1}^{n_*}\hat P_m(\hat\lambda_j)\langle \hat r,\hat v_j\rangle\hat v_j\qquad \text{and}\qquad \hat\beta_m^{\rm PCA} = \sum_{j=1}^{m}\frac{1}{\hat\lambda_j}\langle \hat r,\hat v_j\rangle\hat v_j
\end{equation*}
to denote the functional PLS and PCA estimators. Note that the PLS estimator uses supervised regularization $\hat P_m$ while for the PCA estimator the regularization is fixed to select the terms related to the inverse of the largest $m$ eigenvalues of $\hat K$. We will also use 
\begin{equation*}
	\beta_m^{\rm PLS}=\sum_{j=1}^\infty P_m(\lambda_j)\langle r,v_j\rangle v_j\qquad \text{and}\qquad \beta_m^{\rm PCA} =  \sum_{j=1}^m\lambda_j^{-1}\langle r,v_j\rangle v_j
\end{equation*}
to denote the population counterparts.

\begin{theorem}\label{thm:fit_and_bias}
	If $n_*=n$, then for every $m \leq n_*$,
	\begin{equation*}
		\left\|\hat r - \hat K\hat\beta_m^{\rm PLS}\right\| \leq \left\|\hat r - \hat K\hat\beta_m^{\rm PCA}\right\|.
	\end{equation*}
	and
	\begin{equation*}
		\left\|K^s(\beta_m^{\rm PLS} - \beta)\right\| \leq \left\|K^s(\beta_m^{\rm PCA} - \beta)\right\|,\qquad \forall s\in[0,1].
	\end{equation*}
\end{theorem}

\begin{proof}[\rm \textbf{Proof of Theorem~\ref{thm:fit_and_bias}}]
	If $m=n_*$, then the PLS objective is zero, so the result is trivial. Suppose that $m<n_*$. Recall that $\hat Q_m(\lambda)=1-\lambda \hat P_m(\lambda)$. Let $(\hat v_j)_{j=1}^\infty$ be a basis of $\mathbb{H}$, where the first $n_*$ terms correspond to the eigenbasis of $\hat K$. Then $\hat r = \sum_{j=1}^\infty\langle \hat r,\hat v_j\rangle\hat v_j$ and for every $m< n_*$ by \cite{blazere2014unified}, Proposition 6.2
	\begin{equation*}
		\begin{aligned}
			\left\|\hat r - \hat K\hat\beta_m^{\rm PLS}\right\|^2  & = \left\|\hat Q_m(\hat K)\hat r\right\|^2 \\
			& \leq  \sum_{j=m+1}^{n_*}\prod_{k=1}^m\left(1 - \frac{\hat\lambda_j}{\hat\lambda_{k}}\right)^2\langle \hat r,\hat v_j\rangle^2 \\
			& \leq \sum_{j=m+1}^{n_*} \langle\hat r,\hat v_j\rangle^2 \\
			& = \left\|\sum_{j=1}^\infty \langle \hat r,\hat v_j\rangle\hat v_j - \sum_{j=1}^m\frac{1}{\hat\lambda_j}\langle \hat r,\hat v_j\rangle\hat K\hat v_j \right\|^2 \\
			& = \left\|\hat r - \hat K\hat\beta_m^{\rm PCA}\right\|^2,
		\end{aligned}
	\end{equation*}
	where the third line follows since $\hat\lambda_1\geq \hat\lambda_2\geq \dots\geq \hat\lambda_{n_*}>0$; and the fourth by Parseval's identity.
	
	For the second part, put $Q_m(\lambda)=1-\lambda P_m(\lambda)$, where $\beta_m^{\rm PLS}=P_m(K)r$ solves the population counterpart to the problem in equation (\ref{eq:pls_problem}). 	Similarly, since $\beta = \sum_{j=1}^\infty\langle \beta,v_j\rangle v_j$ and $K\beta=r$, we have
	\begin{equation*}
		\begin{aligned}
			\left\|K^s(\beta_m^{\rm PLS} - \beta)\right\|^2 & = \left\|K^sQ_m(K)\beta\right\|^2 \\
			& \leq \sum_{j=1}^\infty\lambda_j^{2s}\prod_{k=1}^m\left(1 - \frac{\lambda_j}{\lambda_{j_k}}\right)^2 \langle\beta,v_j\rangle^2 \\
			& \leq \sum_{j=m+1}^\infty\lambda_j^{2s}\langle\beta,v_j\rangle^2 \\
			& = \left\|K^s\left(\sum_{j=1}^\infty\langle \beta,v_j\rangle v_j - \sum_{j=1}^m\frac{1}{\lambda_j}\langle K\beta,v_j\rangle v_j \right)\right\|^2 \\
			& =  \left\|K^s(\beta - \beta_m^{\rm PCA})\right\|^2.
		\end{aligned}
	\end{equation*}
\end{proof}

The first part of Theorem~\ref{thm:fit_and_bias} shows that the PLS estimator fits the data better than PCA for the same number of components $1\leq m\leq n_*$. This is the functional version of a result of \cite{jong1993pls}; see also \cite{phatak2002exploiting} and \cite{blazere2014unified}. For the second part of Theorem~\ref{thm:fit_and_bias}, it is worth recalling that the estimation and prediction errors in Theorem~\ref{thm:pls_rate} can be decomposed as
\begin{equation*}
	K^s\left(\hat\beta_m^{\rm PLS } - \beta\right) = K^s\left(\hat\beta_m^{\rm PLS } - \beta^{\rm PLS }_m \right) + K^s\left(\beta_m^{\rm PLS } - \beta\right),\qquad s\in\{0,1/2\},
\end{equation*}
where the second term is the so-called regularization bias. This shows that the PLS basis is more adapted for approximating the slope $\beta$ than the PCA basis.

\section{Comparison to Standard PLS}\label{sec:standard_pls}
We provide some comparison between our PLS which is a functional conjugate gradient method applied to the moment equation with the standard PLS. Note that the standard PLS solves
\begin{equation}
	\min_{b\in\mathbb{H}_m}\|\mathbf{y} - T_nb\|_n^2,
\end{equation}
where $\mathbb{H}_{m}=\mathrm{span}\left \{ \hat{r},\hat{K}\hat{r},...,%
\hat{K}^{m-1}\hat{r}\right \}$. Equivalently, we solve%
\[
\min_{\alpha_1,\dots,\alpha_m\in\mathbb{R} }\left\| \mathbf{y} - \sum_{j=1}^{m}\alpha _{j}T_{n}\hat{K}^{j-1}%
\hat{r}\right\|_n^{2}.
\]%
The first order conditions give%
\begin{equation}
	-\left \langle T_{n}\hat{K}^{j-1}\hat{r},\mathbf{y}\right \rangle +\left \langle
	T_{n}\hat{K}^{j-1}\hat{r},\sum_{l=1}^{m}\alpha _{l}T_{n}\hat{K}%
	^{l-1}\hat{r}\right \rangle =0,\qquad j=1,2,...,m.  \label{eq1}
\end{equation}%
Remark that%
\begin{eqnarray*}
	T_{n}\hat{K}^{j-1}\hat{r} &=&T_{n}\hat{K}^{j-1}T_{n}^{\ast
	}\mathbf{y}=T_{n}\left( T_{n}^{\ast }T_{n}\right) ^{j-1}T_{n}^{\ast }\mathbf{y}=\left(
	T_{n}T_{n}^{\ast }\right) ^{j}\mathbf{y}, \\
	\left \langle T_{n}\hat{K}^{j-1}\hat{r},T_{n}\hat{K}^{l-1}%
	\hat{r}\right \rangle  &=&\left \langle \hat{K}^{j-1}\hat{r}%
	,T_{n}^{\ast }T_{n}\hat{K}^{l-1}\hat{r}\right \rangle  \\
	&=&\left \langle \hat{K}^{j-1}\hat{r},\hat{K}^{l}\hat{r}%
	\right \rangle  \\
	&=&\left \langle \hat{r},\hat{K}^{j+l-1}\hat{r}\right \rangle  \\
	&=&\left \langle T_{n}^{\ast }\mathbf{y},\hat{K}^{j+l-1}T_{n}^{\ast
	}\mathbf{y}\right \rangle  \\
	&=&\left \langle \mathbf{y},\left( T_{n}T_{n}^{\ast }\right) ^{j+l}\mathbf{y}\right \rangle  \\
	&=&\mathbf{y}^\top\left( T_{n}T_{n}^{\ast }\right) ^{j+l}\mathbf{y}.
\end{eqnarray*}%
where $T_{n}T_{n}^{\ast }$ is the $n\times n$ matrix with $\left( i,j\right) 
$ element $\left \langle X_{i},X_{j}\right \rangle$. Let $v_{1}$ be the $%
m\times 1$ vector and $M_{1}$ the $m\times m$ matrix such that

\begin{equation*}
	\begin{aligned}
		v_{1} & =
		\begin{bmatrix}
			\left \langle T_{n}\hat{K}\hat{r},\mathbf{y}\right \rangle  \\ 
			\vdots  \\ 
			\left \langle T_{n}\hat{K}^{m-1}\hat{r},\mathbf{y}\right \rangle 
		\end{bmatrix}%
		=
		\begin{bmatrix}
			\mathbf{y}^\top\left( T_{n}T_{n}^{\ast }\right) \mathbf{y} \\ 
			\vdots  \\ 
			\mathbf{y}^\top\left( T_{n}T_{n}^{\ast }\right) ^{m}\mathbf{y}%
		\end{bmatrix}, \\
		M_1 & = \left[ \left \langle T_{n}\hat{K}^{j-1}\hat{r}%
		,T_{n}\hat{K}^{l-1}\hat{r}\right \rangle \right] _{1\leq j,l\leq m}=\left[
		\mathbf{y}^\top\left( T_{n}T_{n}^{\ast }\right) ^{j+l}\mathbf{y}\right] _{1\leq j,l\leq m}.
	\end{aligned}	
\end{equation*}

Equation (\ref{eq1}) can be rewritten as $M_{1}\alpha =v_{1}$ where $\alpha =%
\left[ \alpha _{1},\alpha _{2},...,\alpha _{m}\right]^\top.$ So that $%
\hat\alpha _{PLS}=M_{1}^{-1}v_{1},$ provided that the matrix $M_1$ is invertible.

\bigskip 

Our version of PLS solves 
\[
\min_{b\in \mathbb{H}_{m}}\left \Vert T_{n}^{\ast }\mathbf{y}-T_{n}^{\ast
}T_{n}b\right \Vert ^{2}
\]%
where $\mathbb{H}_{m}=\mathrm{span}\left \{ \hat{r},\hat{K}\hat{r},...,%
\hat{K}^{m-1}\hat{r}\right \} .$ This can be understood as the PLS with respect to the weighted norm. Equivalently, we want to solve%
\[
\min_{\alpha_1,\dots,\alpha_m\in\mathbb{R} }\left \Vert T_{n}^{\ast }\mathbf{y}-\sum_{j=1}^{m}\alpha _{j}\hat{K}%
^{j}\hat{r}\right \Vert ^{2}.
\]%
Using the first order condition, we obtain the equation $v_{2}=M_{2}\alpha $
with 
\begin{equation*}
	\begin{aligned}
		v_{2} & = \left[ 
		\begin{array}{c}
			\left \langle \hat{K}\hat{r},T_{n}^{\ast }\mathbf{y}\right \rangle  \\ 
			\vdots  \\ 
			\left \langle \hat{K}^{m}\hat{r},T_{n}^{\ast }\mathbf{y}\right \rangle 
		\end{array}%
		\right] =\left[ 
		\begin{array}{c}
			\mathbf{y}^{\top }\left( T_{n}T_{n}^{\ast }\right) ^{2}\mathbf{y} \\ 
			\vdots  \\ 
			\mathbf{y}^{\top }\left( T_{n}T_{n}^{\ast }\right) ^{m+1}\mathbf{y}%
		\end{array}%
		\right] , \\
		M_{2} & = \left[ \left \langle \hat{K}^{j}\hat{r},%
		\hat{K}^{l}\hat{r}\right \rangle \right] _{j,l}=\left[ \mathbf{y}^{\top
		}\left( T_{n}T_{n}^{\ast }\right) ^{j+l+1}\mathbf{y}\right] _{j,l}.
	\end{aligned}	
\end{equation*}
If we compare $v_{1}$ and $v_{2}$, $M_{1}$ and $M_{2},$ we see that the only
difference is in the power of $\left( T_{n}T_{n}^{\ast }\right) $. Note also that 
$M_{1}$ and $M_{2}$ are Hankel matrices.

\section{Additional Simulations: Early Stopping and Confidence Sets}\label{suppl:simulations}
In this section we report additional simulation results for our early stopping rule and confidence sets. 

\subsection{Early Stopping}
Recall that our PLS amounts to fitting the norm of a ``sample moment" (or a score)
$$\left\| \hat{r}-\hat{K}\hat{\beta }_{m}\right\|.$$ The norm decreases monotonically to zero in $m$ and it becomes zero when the number of conjugate gradient steps reaches the number of non-zero eigenvalues of $\hat K$, i.e. $m=n_*$. To prevent overfitting, we select the number of PLS components as the first value of $m$ for which the norm of residual drops below a certain threshold:
\begin{equation}\label{eq:threshold}
	\left \Vert \hat{r}-\hat{K}\hat{\beta }_{m}\right \Vert \leq \tau
	\sigma \sqrt{\frac{2\E\left \Vert X\right \Vert ^{2}}{\delta n}},
\end{equation}
where $\tau >1$ is a constant and $1-\delta$ is the confidence level of the rule; see Assumption~\ref{as:stopping}. To describe the practical implementation, we focus on the functional linear model
\begin{equation*}
	Y_i = \int_0^1\beta(s)X_i(s)\mathrm{d} s + \varepsilon_i,\qquad i=1,\dots,n.
\end{equation*}
We compute the fitted moment as:
\begin{equation*}
	\hat r(u) - (\hat K\hat\beta_m)(u) = \frac{1}{n}\sum_{i=1}^n\left(Y_i - \int_0^1\hat\beta_m(s)X_i(s)\mathrm{d}s\right)X_i(u),
\end{equation*}
To compute the threshold in equation (\ref{eq:threshold}), we estimate $\E\left \Vert X\right \Vert ^{2}$ by $\frac{1}{n}\sum_{i=1}^{n}\int_0^1 X_{i}^{2}\left( s\right)
ds$ and $\sigma ^{2}$ by $\hat\sigma^2 = \frac{1}{n}\sum_{i=1}^{n}\hat{\varepsilon}_{i}^{2}$, where 
\begin{equation*}
	\hat{\varepsilon}_{i}=Y_{i}-\int_0^1 \hat{\beta }\left( s\right) X_{i}\left(
	s\right) ds
\end{equation*}
and $\hat{\beta }$ is a preliminary estimator of $\beta $, described below. All integrals are discretized at the uniform grid of $200$ points. 

We consider an iterative approach for estimating $\sigma ^{2}$ inspired by \cite{chernozhukov2024applied}, Section 3.A. To that end, we first obtain a pilot estimator of $\beta$, denoted $\hat\beta_0$. Next, we compute the estimator
of $\sigma ^{2}:$%
\[
\hat{\sigma}_{0}^{2}=\frac{1}{n}\sum_{i=1}^{n}\left( Y_{i}-\langle X_i,\hat\beta_0\rangle\right) ^{2}.
\]%
If $n>>T$, we can use the OLS estimator. Otherwise, any other regularized estimator can be used, e.g. PCA with generalized cross-validation. We set $k=0$ and specify a small constant $\xi \geq 0$ as a tolerance level and the maximum number of iterations $k_{\max }.$ The iterative procedure is described in Algorithm~\ref{algor2}.

\begin{algorithm}[H]\label{algor2}
	\SetAlgoLined
	\KwResult{$\hat\sigma^2_{k}$ }
	\textbf{Initialisation:}
	$\xi$,  $k_{\max}$\;
	\For{$k = 0,1,\dots, k_{\max}$}{
		1. Compute $\hat{\beta }$ using the early stopping rule with $\sigma^2$
		replaced by $\hat{\sigma}_{k}^{2}$;
		
		2. Set $\hat{\sigma}_{k+1}^{2}=\frac{1}{n}\sum_{i=1}^{n}(
		Y_{i}- \langle \hat{\beta},X_{i} \rangle ) ^{2}$;
		
		3. If $\left \vert \hat{\sigma}_{k+1}^{2}-\hat{\sigma}_{k}^{2}\right \vert
		\leq \xi $ stop. 
	}
	\caption{Iterative Estimation of $\sigma^2$.}
\end{algorithm}

\begin{figure}
	\caption{Slope coefficient $\beta$ (solid black) and averaged early stopped PLS estimator $\hat\beta_{\hat m}$ (dashed red) with 90\% pointwise confidence bands (shaded gray), calculated from $5,000$ samples of size $n=1,000$. The median number of selected functional components is reported below.}
	\begin{center}
		\begin{subfigure}{0.45 \linewidth}
			\includegraphics[width = \linewidth]{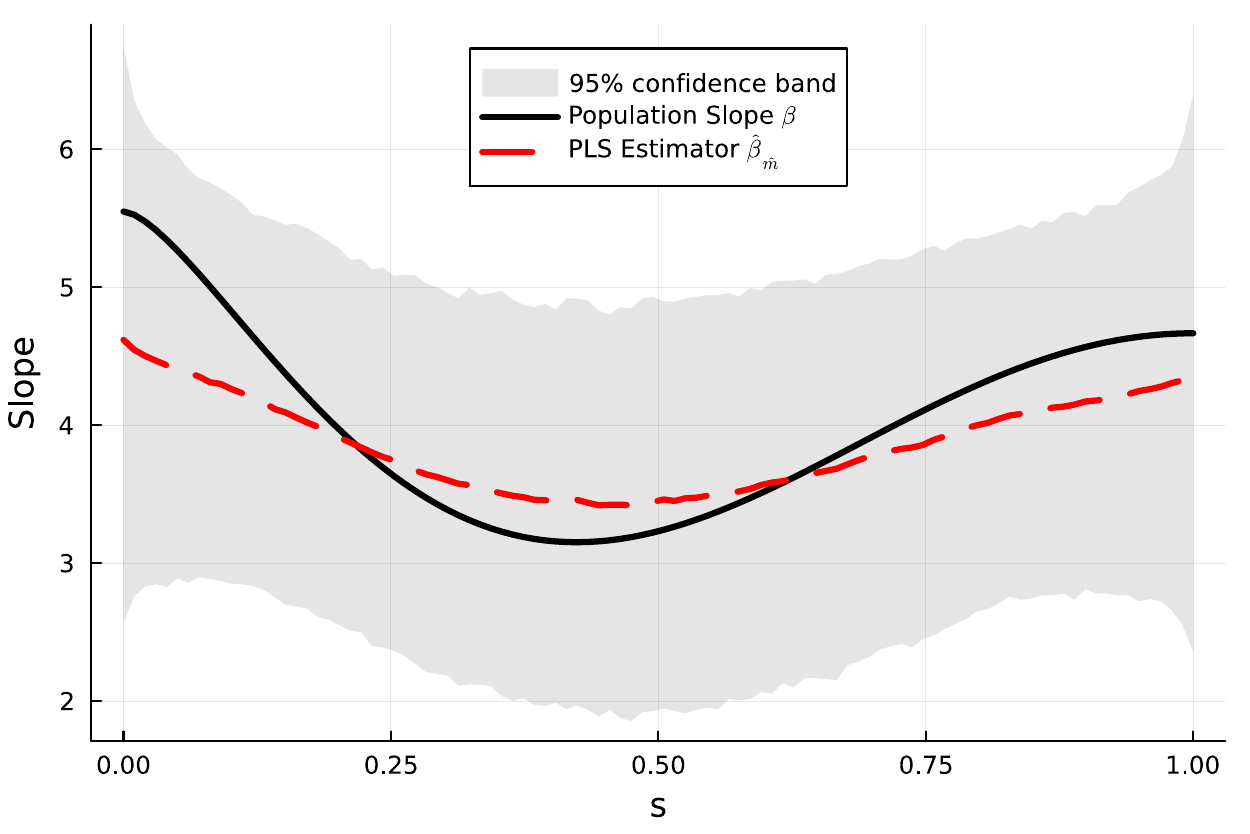} 
			\caption{Model 1: $\mathrm{med}(\hat m) = 2$}
		\end{subfigure}
		\begin{subfigure}{0.45 \linewidth}
			\includegraphics[width = \linewidth]{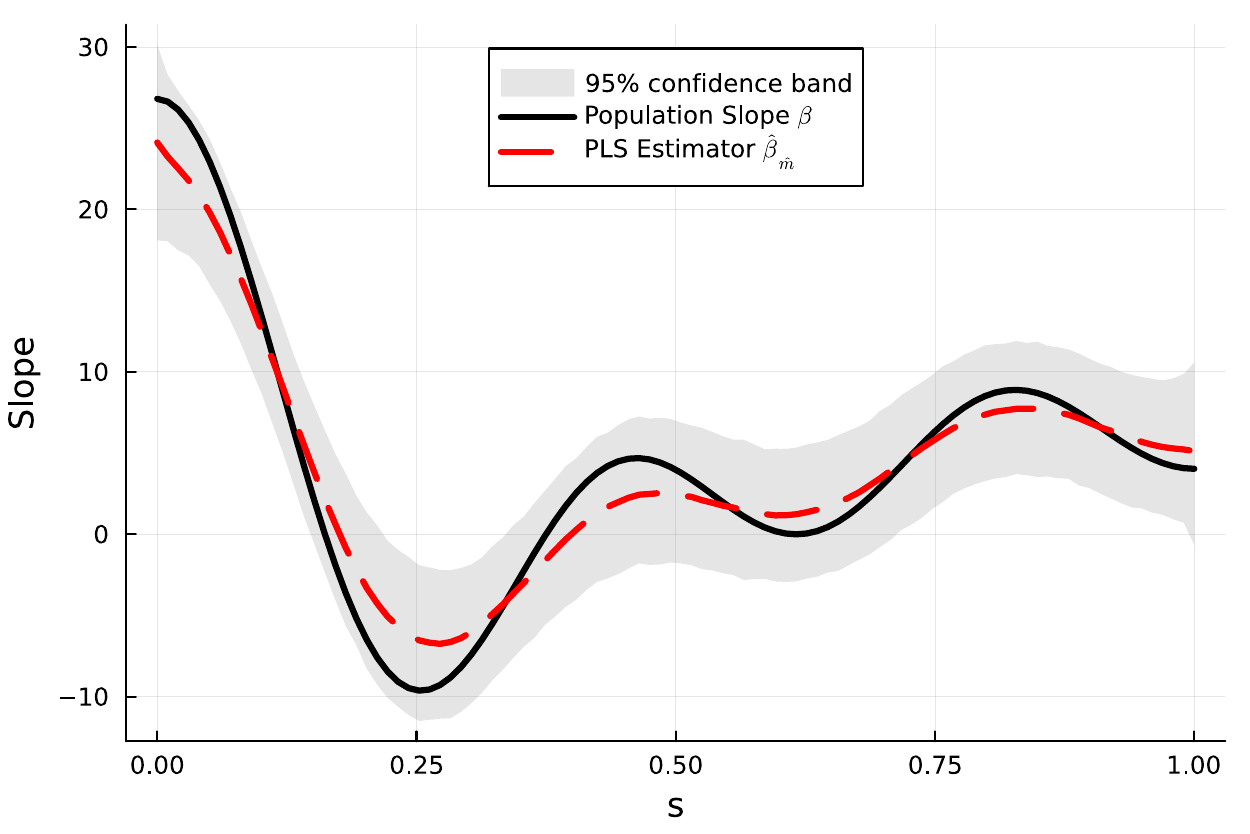}
			\caption{Model 2: $\mathrm{med}(\hat m)= 4$}
		\end{subfigure}
		\begin{subfigure}{0.45 \linewidth}
			\includegraphics[width = \linewidth]{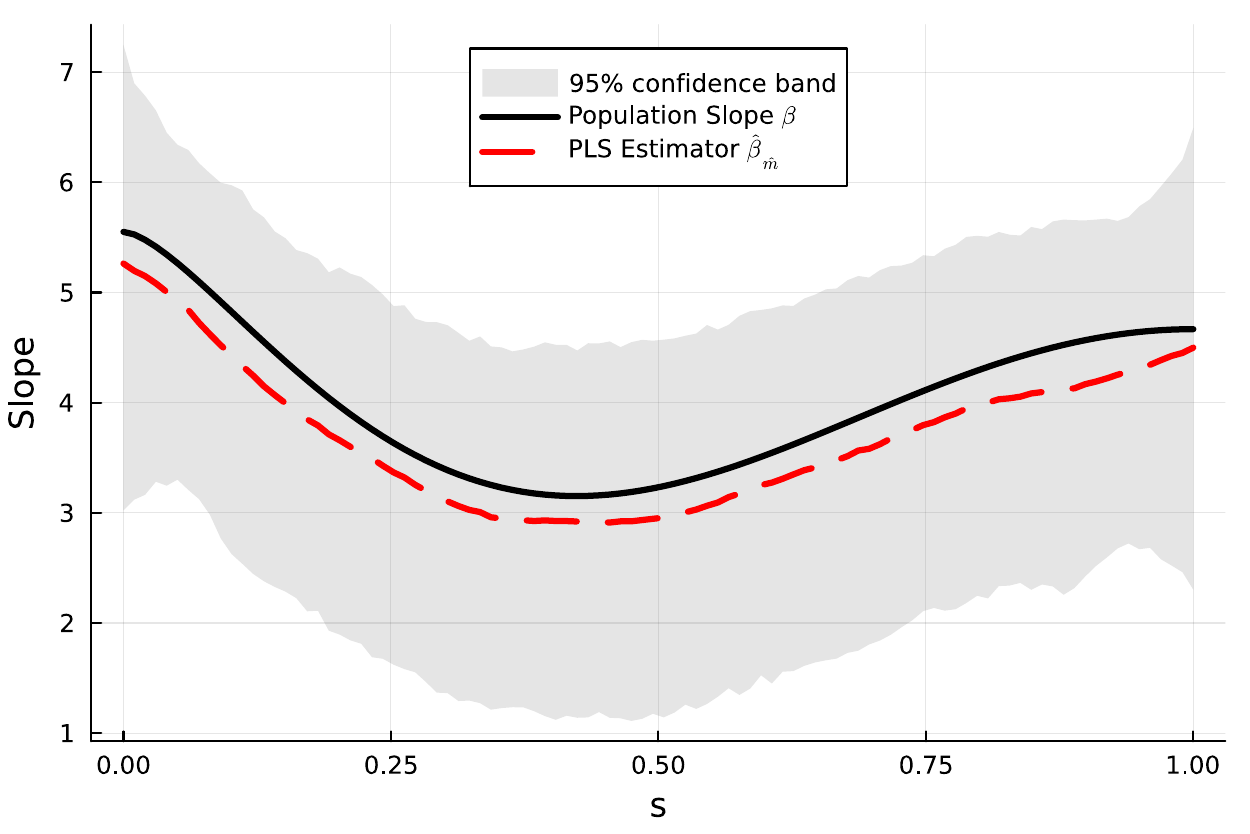} 
			\caption{Model 3: $\mathrm{med}(\hat m) = 4$}
		\end{subfigure}
	\end{center}
	\label{fig:betas}
\end{figure}

We use the following values in all simulation designs: $\delta=0.1$, $k_{\max}=10$, and $\xi=0.01$. These values correspond to a $90\%$ confidence level, a sufficiently high number of iterations, and a sufficiently low tolerance level  respectively. Our theory also tells us that we need $\tau>1$, so we set $\tau=1.01$. The same values are also used in the empirical application. We noticed in simulations that the algorithm tends to overestimate $\sigma^2$ which leads to underestimated values of $m$.  Figure~\ref{fig:betas} reports the population slope coefficient (solid black) and the average values of estimated slope coefficients $\hat\beta_{\hat m}$ (dashed red) when the selected number of PLS components $\hat m$ for the four simulation designs considered in the paper, Section~\ref{sec:mc}. We also report the median values of $\hat m$ selected by our data-driven adaptive rule in each case. Overall, we can see that the PLS estimator with the number of components selected using early stopping can successfully recover the global shape of the slope parameter across all simulation designs. It is worth noting that the early stopping rule selects the number of PLS components to achieve the minimax-optimal convergence rate for both the mean integrated squared error (MISE) and the mean squared prediction error (MSPE). In doing so, it simultaneously balances and minimizes the sum of the squared bias and the variance. This behavior is consistent with the results shown in Figure~\ref{fig:betas}, where a visible bias is observed as part of the optimal bias-variance trade-off.

Figure~\ref{fig:stopping} visualizes our early stopping rule. The left vertical axis corresponds to the values of MSPE (solid black) while the right vertical axis to the early stopping rule (blue). The number of PLS components is selected as the first $m$ (blue dot) when the norm of the fitted moment (dashed blue) drops below the value of the threshold (dotted blue). We also plot the minimum (black dot) of the simulated MSPE (solid black) together with $90\%$ confidence band (shaded gray).  Given the uncertainty behind the lowest value of the MSPE as well as the norm and the threshold, the early stopping rule produces reasonable values of tuning parameters that are compatible with the global minimum of MSPE once the statistical uncertainty is accounted for. The simulation results also confirm that there is no reason to select higher values of $\tau$ since they would increase the threshold and lead to more convervative choices of $\hat m$. We conclude that our early stopping rule is prone to oversmoothing.
\begin{figure}
	\caption{Adaptive Early Stopping Rule Visualization, calculated from $5,000$ samples of size $n=1,000$. We plot the average MSPE (solid black curve) with a 95\% confidence interval obtained from simulations. We also plot the average norm (dashed blue). The number of components is selected as the first $m$ (blue dot) when the norm (dashed blue) drops below the the threshold (dotted blue). The black dot correspond to the lowest values of the average MSPE (solid black) which is plotted with $90\%$ confidence band (shaded gray).}
	\begin{center}
		\begin{subfigure}{0.45 \linewidth}
			\includegraphics[width = \linewidth]{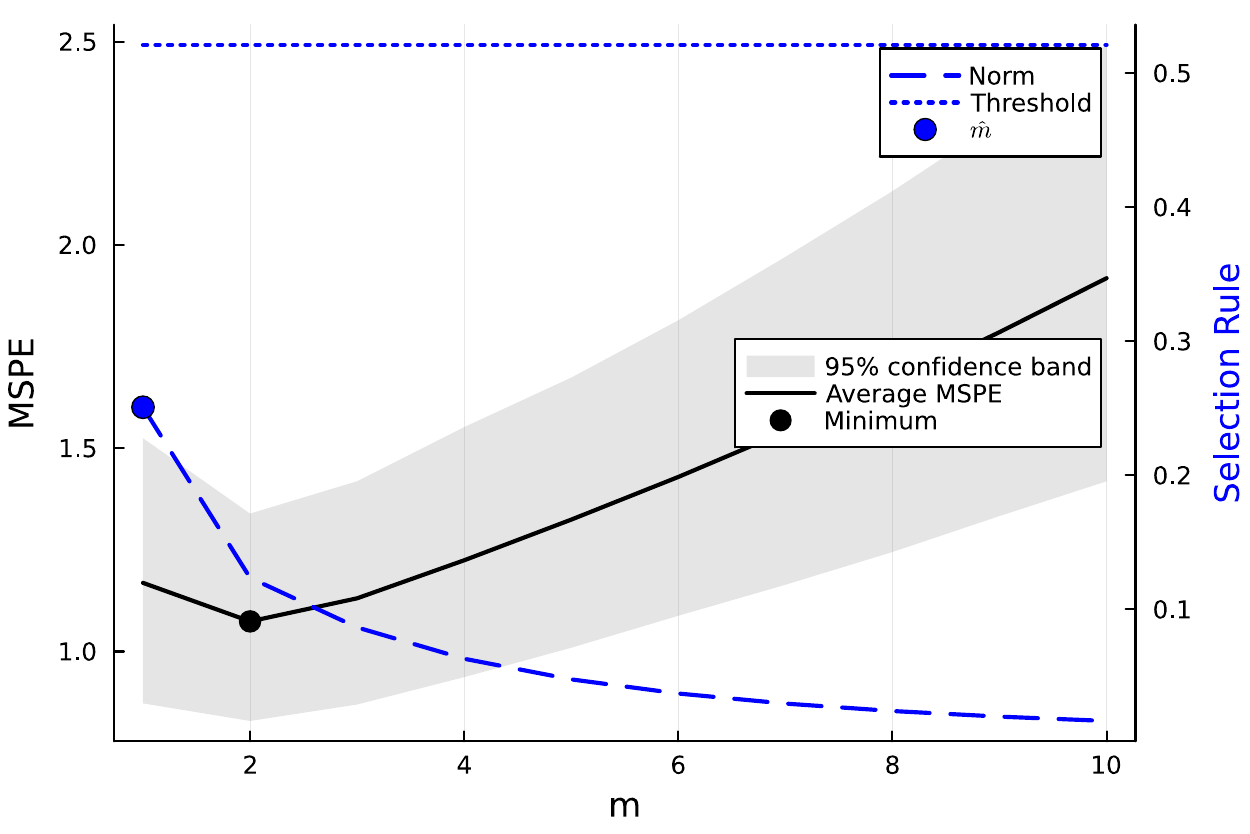} 
			\caption{Model 1: $\mathrm{med}(\hat m) = 2$}
		\end{subfigure}
		\begin{subfigure}{0.45 \linewidth}
			\includegraphics[width = \linewidth]{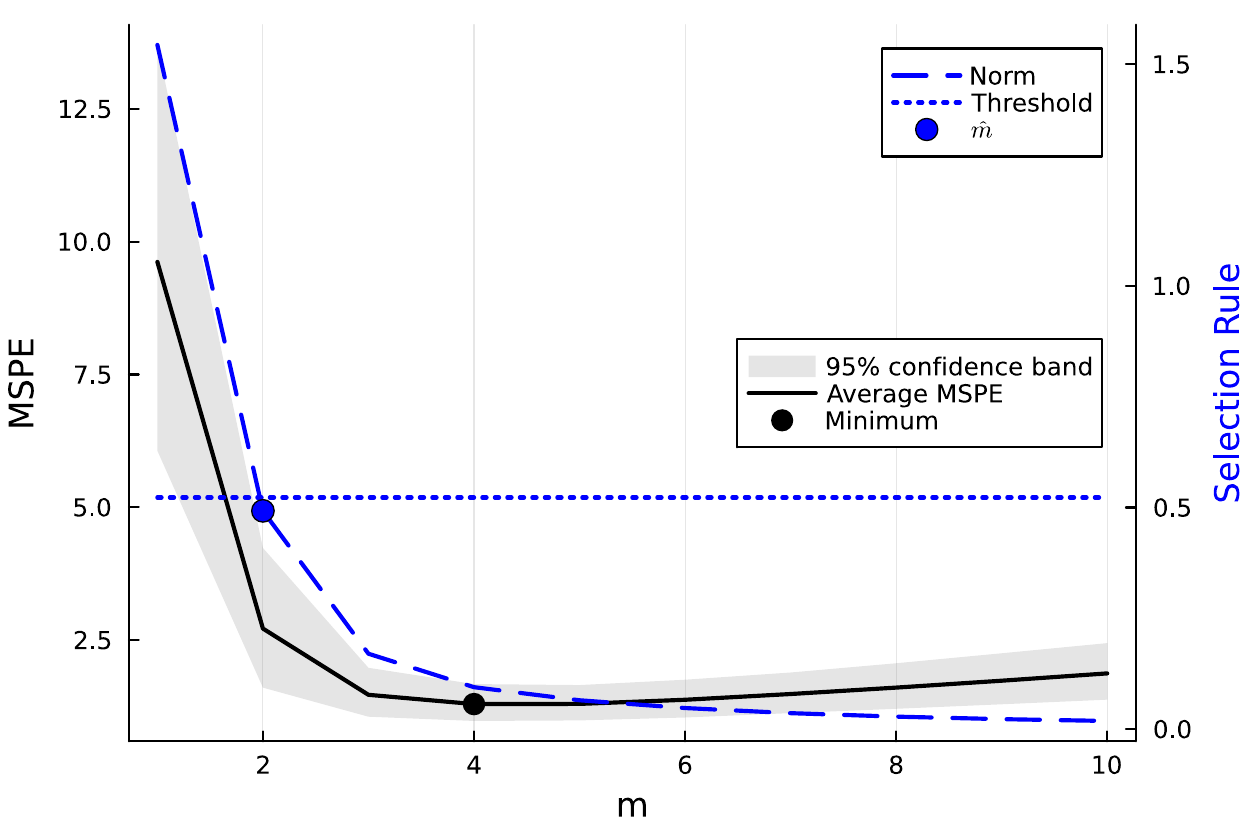}
			\caption{Model 2: $\mathrm{med}(\hat m)= 4$}
		\end{subfigure}
		\begin{subfigure}{0.45 \linewidth}
			\includegraphics[width = \linewidth]{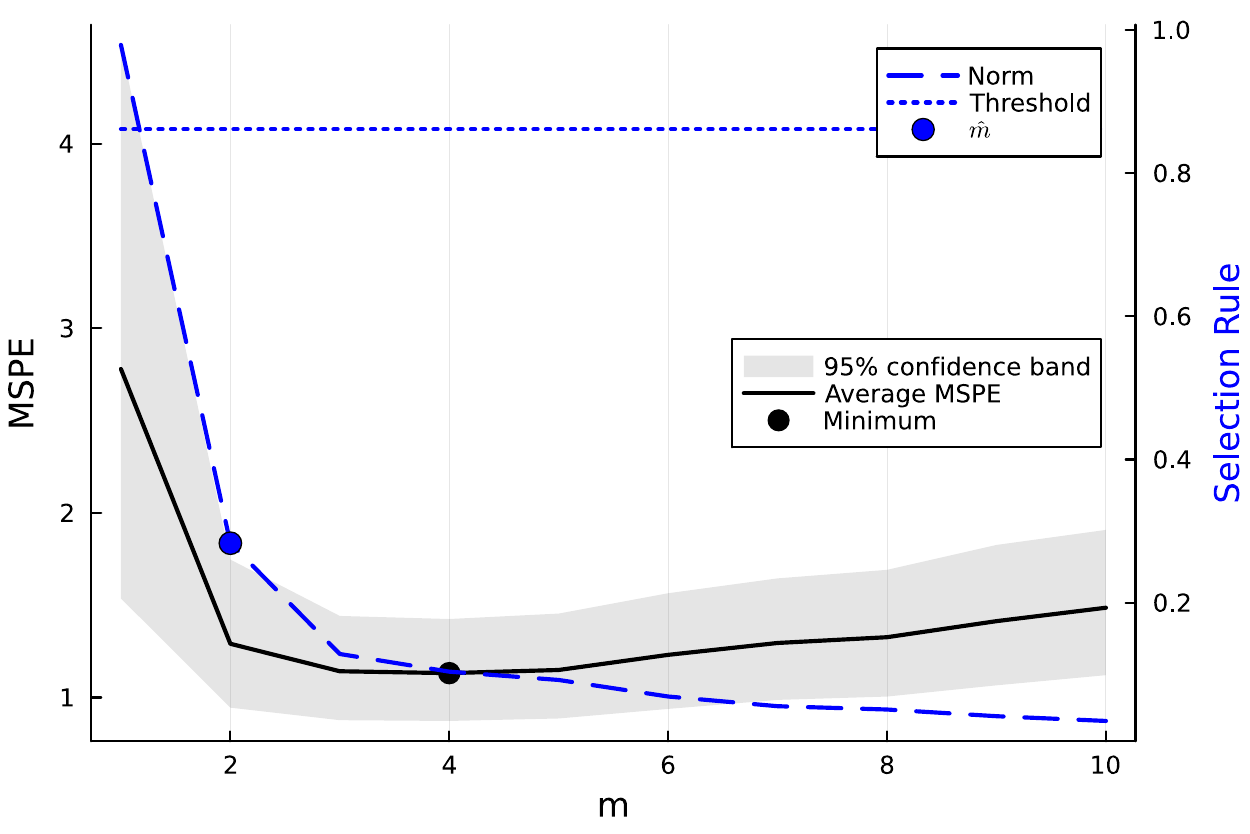} 
			\caption{Model 3: $\mathrm{med}(\hat m) = 4$}
		\end{subfigure}
	\end{center}
	\label{fig:stopping}
\end{figure}

\subsection{Confidence Sets}
To compute the confidence sets using test inversion, we note that if $(h_j)_{j=1}^\infty$ is a basis of $L_2[0,1]$, then
\begin{equation*}
	\beta(s) = \sum_{j=1}^\infty b_jh_j(s)
\end{equation*}
with coefficients $b_j = \langle\beta,h_j\rangle$. Since $b_j\downarrow 0$, to reduce the computational cost, we search over functions $\beta$, where the sum is truncated to include only the first five basis functions. We use the cosine basis and create a uniform grid of $20$ points on $[0,4.5]^5$ corresponding to the first five coefficients $b_1,\dots,b_5$. We simulate the critical value using the approximation to the asymptotic distribution $\sum_{j=1}^{100}\lambda_jZ_j^2$, where $(\lambda_j)_{j=1}^\infty$ are the eigenvalues of $K$, and we use $50,000$ replications to compute the critical value $z_{0.95}$.

Figure~\ref{fig:confidence_set} displays the $95\%$ confidence sets. These confidence sets are quite informative about the global shape properties of the estimated slope parameter. 
\begin{figure}[htbp]
	\centering
	% Model 1
	\begin{subfigure}[b]{0.5\textwidth}
		\includegraphics[width=\textwidth]{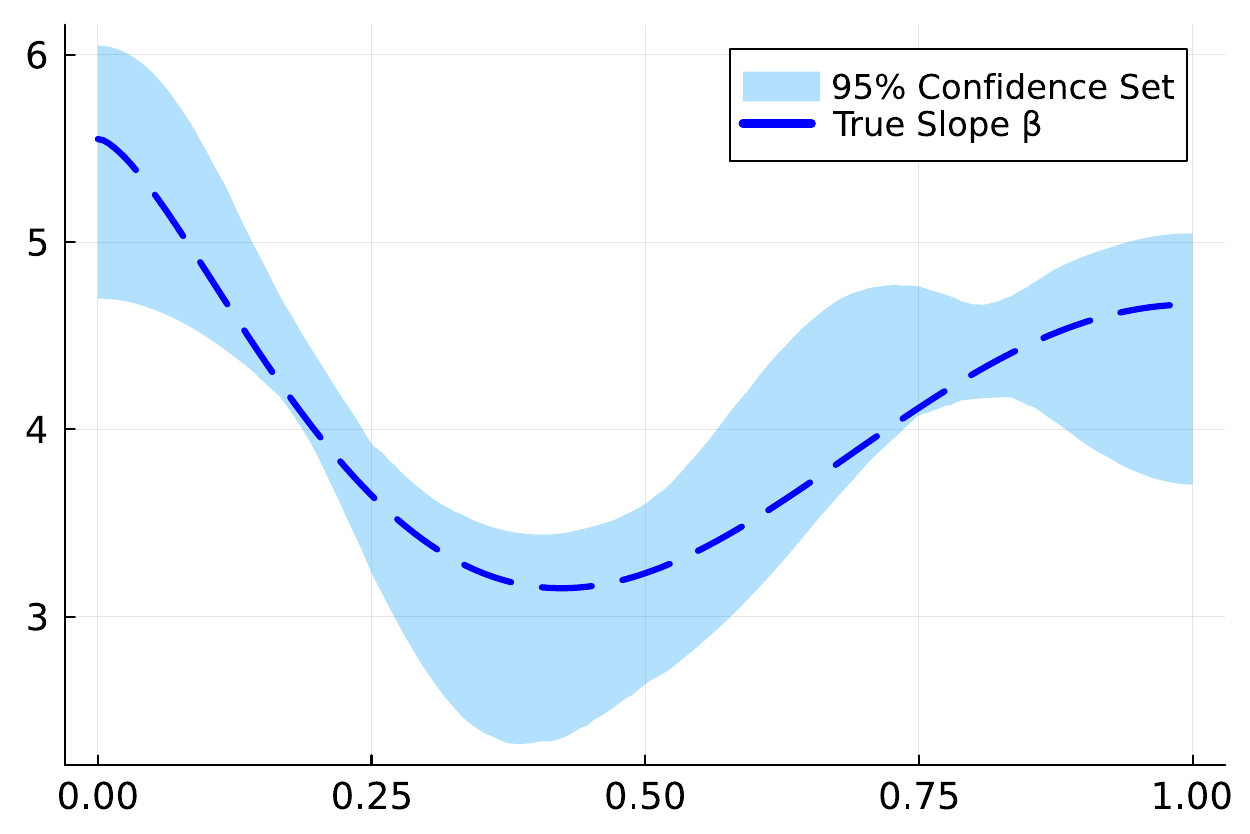}
		\caption{Model 1}
		\label{fig:cs1}
	\end{subfigure}\hfill
	\begin{subfigure}[b]{0.5\textwidth}
		\includegraphics[width=\textwidth]{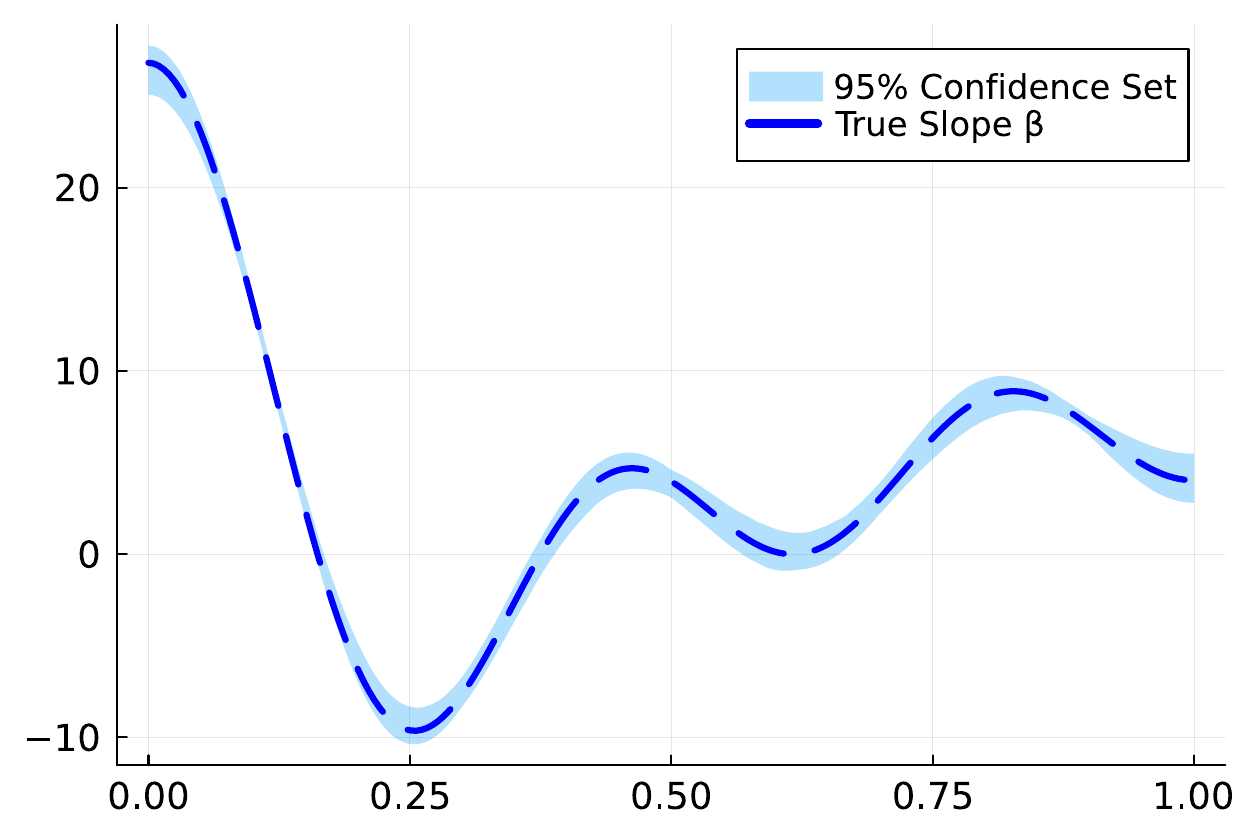}
		\caption{Model 2}
		\label{fig:cs2}
	\end{subfigure}\hfill
	% Model 2
	\begin{subfigure}[b]{0.5\textwidth}
		\includegraphics[width=\textwidth]{cs_model1.pdf}
		\caption{Model 3}
		\label{fig:cs3}
	\end{subfigure}\hfill
	\caption{Confidence Sets. The figure plots the pointwise median values of the upper and lower bound of a $95\%$ confidence band across $5,000$ simulations.}
	\label{fig:confidence_set}
\end{figure}

\end{document}